\documentclass[12pt]{extarticle}
\usepackage[utf8]{inputenc}

\usepackage{caption}
\oddsidemargin \evensidemargin
\usepackage{amssymb}
\usepackage{amsmath}
\usepackage{comment}
\usepackage{amsthm}
\usepackage{amssymb}
\usepackage{mathtools}
\usepackage{hyperref}
\hypersetup{
  colorlinks   = true,    
  urlcolor     = mycolor1,    
  linkcolor    = black,    
  citecolor    = mycolor1    
}
\usepackage[margin=1cm]{caption}
\usepackage{subcaption}
\usepackage{mathrsfs}
\usepackage{xcolor}
\usepackage{tikz,tikz-3dplot,tikz-cd}
\usepackage{pgfplots}
\usepackage{url}
\usepackage[frozencache]{minted}
\setminted[julia]{frame=lines,
rulecolor=\color{white!80!black},
fontsize=\small,
numbers=right,
numbersep=-5pt,
obeytabs=true,
tabsize=4}
\usepackage{textgreek}

\definecolor{mycolor1}{rgb}{0.00000,0.44700,0.74100}
\definecolor{mycolor2}{rgb}{0.8500, 0.3250, 0.0980}
\definecolor{mycolor3}{rgb}{0.9290, 0.6940, 0.1250}
\definecolor{mycolor4}{rgb}{0.4940, 0.1840, 0.5560}
\definecolor{mycolor5}{rgb}{0.4660, 0.6740, 0.1880}

\theoremstyle{remark}
\newtheorem{remark}{Remark}[section]

\theoremstyle{definition}
\newtheorem{definition}[remark]{Definition}
\newtheorem{proposition}[remark]{Proposition}
\newtheorem{corollary}[remark]{Corollary}
\newtheorem{theorem}[remark]{Theorem}
\newtheorem{exercise}[remark]{Exercise}
\newtheorem{lemma}[remark]{Lemma}
\theoremstyle{definition}
\newtheorem{examplex}[remark]{Example}
\newenvironment{example}
  {\pushQED{\qed}\examplex}
  {\popQED\endexamplex}
 \newtheorem{examplestarx}[remark]{Example*}
\newenvironment{examplestar}
  {\pushQED{\qed}\examplestarx}
  {\popQED\endexamplestarx}
  \usepackage{cleveref}

\usepackage{mathtools} 
\numberwithin{equation}{section}

\newcommand{\R}{\mathbb{R}}
\newcommand{\Z}{\mathbb{Z}}
\newcommand{\C}{\mathbb{C}}
\newcommand{\N}{\mathbb{N}}
\newcommand{\Q}{\mathbb{Q}}
\newcommand{\PP}{\mathbb{P}}
\newcommand{\A}{\mathscr{A}}
\newcommand{\J}{\mathscr{J}}
\newcommand{\id}{\mathrm{id}}
\newcommand{\p}{\mathfrak{p}}
\renewcommand{\S}{\mathsf{S}}
\newcommand{\im}{\operatorname{im}}
\newcommand{\Specm}{\operatorname{Specm}}
\newcommand{\Hom}{\operatorname{Hom}}
\newcommand{\blue}[1]{\textcolor{mycolor1}{#1}}
\newcommand{\red}[1]{\textcolor{mycolor2}{#1}}

\newcommand{\Div}{\mathrm{Div}}
\newcommand{\CDiv}{\mathrm{CDiv}}
\newcommand{\Cl}{\mathrm{Cl}}
\newcommand{\Pic}{\mathrm{Pic}}
\renewcommand{\div}{\mathrm{div}}
\newcommand{\OO}{\mathscr{O}}
\newcommand{\f}{\hat{f}}

\newcommand{\V}{V}
\newcommand{\Irrel}{B}

\title{Introduction to Toric Geometry}
\author{Simon Telen}
\date{}

\begin{document}
\maketitle

\begin{abstract}
These notes are based on a series of lectures given by the author at the Max Planck Institute for Mathematics in the Sciences in Leipzig. Addressed topics include affine and projective toric varieties, abstract normal toric varieties from fans, divisors on toric varieties and Cox's construction of a toric variety as a GIT quotient. We emphasize the role of toric varieties in solving systems of polynomial equations and provide many computational examples using the Julia package \texttt{Oscar.jl}. 
\end{abstract}

\tableofcontents

\section{Introduction}
Toric varieties form a well-understood subclass of algebraic varieties whose geometric properties are encoded by combinatorial objects. These objects include lattices, polyhedra, polyhedral cones and fans. 
Loosely speaking, toric varieties are those varieties that can be parametrized by monomials. Though this might seem like a narrow class of varieties, many familiar varieties are toric. Examples are $(\C^*)^n = (\C \setminus \{0\})^n$, $\C^n$, $\PP^n$ and $\PP^k \times \PP^{n-k}$. Other toric varieties we will encounter are Hirzebruch surfaces and the cuspidal cubic $\{y^2 - x^3 = 0 \}$, which is parametrized by $t \mapsto (t^2, t^3)$. Toric varieties often serve as test cases for conjectures, or illustrative examples for theorems. Many general concepts from algebraic geometry have a nice, concrete description in the toric setting. Examples are the gluing construction of abstract varieties (Section \ref{subsec:gluing}) and groups of divisors (Section \ref{subsec:backgrounddiv}). Toric geometry provides nice connections between algebraic geometry, polyhedral geometry (Sections \ref{subsec:cones}, \ref{subsec:polytopes}), tropical geometry (Example \ref{ex:tropical}) and geometric invariant theory (Section \ref{subsec:coxconstr}).

Application areas of toric varieties include geometric modelling \cite[Ch.~3]{cox2020applications}, statistics, chemistry and phylogenetics \cite[Sec.~8.3]{michalek2021invitation}, interior point methods for linear programming \cite{sturmfels2022toric} and mirror symmetry \cite{Batyrev94dualpolyhedra}. They also naturally pop up when solving systems of polynomial equations. Toric varieties are the natural solution spaces for families of polynomial systems with a fixed monomial support. This means that we fix the monomials that can occur in the equations, while the coefficients can vary. An example is worked out below (Example \ref{ex:example}). Exploiting this toric structure is crucial for the efficiency, but also for the accuracy of numerical solution methods. Examples of such toric solution methods include polyhedral homotopies \cite{duff2020polyhedral,huber1995polyhedral}, toric resultants \cite{canny2000subdivision,emiris1999matrices} and toric normal form methods \cite{bender2020toric,telen2019numerical}.

These notes are based on a series of lectures given by the author at the Max Planck Institute for Mathematics in the Sciences in Leipzig. They provide a first introduction to the subject, with an emphasis on solving equations and effective computations with toric varieties. In particular, we will highlight some of the functionalities of the Julia package \texttt{Oscar.jl} \cite{OSCAR}. We assume that the reader is familiar with basic concepts from algebraic geometry, at the level of \cite{cox2013ideals}. We mostly follow the presentation in the first four chapters of the standard textbook \cite{cox2011toric}. Another classical reference is \cite{fulton1993introduction}. Our intention is not to be as complete and self-contained as these books. Proofs are given when considered particularly insightful for the rest of the material and we provide full references where proofs are skipped or sketched. Apart from \cite{cox2011toric,fulton1993introduction}, the reader can consult the lecture notes \cite{sottile2017ibadan} by Frank Sottile, and the paper \cite{cox2003toric} by David Cox for further reading. The discussion on Cox rings and the toric interpretation of polynomial systems in Section \ref{sec:homogeneous} is mostly taken from \cite{telen2020thesis}.

We start out with affine toric varieties in Section \ref{sec:affineTV}. These are the building blocks of the projective and abstract toric varieties discussed in Sections \ref{sec:projectiveTV} and \ref{sec:abstractTV}. Sections \ref{sec:affineTV}-\ref{sec:abstractTV} feature the standard constructions of affine toric varieties from cones, projective toric varieties from polytopes and abstract toric varieties from fans. A particularly interesting result for polynomial system solving is Kushnirenko's theorem (Theorem \ref{thm:kush}), which we prove in Section \ref{sec:kush}. Section \ref{sec:divisors} discusses the theory of divisors on normal toric varieties, which is the starting point for defining homogeneous coordinates in Section \ref{sec:homogeneous}.

The text contains many examples, some of which are marked with an asterisk, e.g., \textbf{Example* \ref{ex:chiara1}}. The asterisk indicates that the computations were performed using software. All Julia code related to these `examples*' can be found at the MathRepo page  \url{https://mathrepo.mis.mpg.de/ToricGeometry}. The reader is encouraged to try out these computations on their own machine. For all computations and reported timings, we used \texttt{Oscar.jl} (v0.8.1) in Julia (v1.7.1) on a macOS (v12.0.1) machine with a Quad-Core Intel Core i7 processor working at 2,8 GHz. Some exercises are included in the text as well, and many more can be found in \cite{cox2011toric}. The author is glad to receive any suggestions for further computational examples or exercises for future editions of this course. We conclude the introduction with an example of how toric varieties arise in the context of equation solving. 

\begin{example} \label{ex:example}
Suppose we are interested in solving $f_1 = f_2 = 0$, where 
\begin{align}\label{eq:sys_full}
    f_1 &= a_0 + a_1 x + a_2 y + a_3 x y + a_4 x^2 + a_5 y^2 \\
    f_2 &= b_0 + b_1 x + b_2 y + b_3 x y + b_4 x^2 + b_5 y^2.
\end{align}
If we view $f_1, f_2$ as elements of the polynomial ring $\C[x,y]$, the solutions live naturally in $\C^2$. By Bézout's Theorem, for a generic choice of the complex parameters $a_i, b_i$, there are $4$ such solutions. In our toric context, we will view $f_1$ and $f_2$ as \emph{Laurent polynomials}, i.e., as elements of the larger ring $\C[x^{\pm 1}, y^{\pm 1}]$. Solutions are then points in the \emph{algebraic torus} $(\C^*)^2 \subset \C^2$, where $\C^* = \C \setminus \{0\}$. Generically, all $4$ affine solutions lie in $(\C^*)^2$.

Next, we re-interpret \eqref{eq:sys_full} as a system of \emph{linear} equations on a projective variety. To that end, we define the following map:
\begin{align}
    \Phi_\triangle : (\C^*)^2 \to \PP^5 \quad \text{given by} \quad (x,y) \mapsto (1 : x : y : xy : x^2 : y^2).
\end{align}
We denote by $X_\triangle$ the Zariski closure of its image: $X_\triangle = \overline{\im \Phi_\triangle} \subset \PP^5$. As we will see, $X_\triangle$ is a \emph{projective toric variety}. It is equal to the 2-uple Veronese embedding $\nu_2\left(\PP^2 \right)$ of $\PP^2$. Using homogeneous coordinates $(u_0: \ldots:u_5)$ on $\PP^5$, let us consider the linear space defined by
\begin{equation} \label{eq:Ltriangle}
    L_\triangle = \left\lbrace
    \begin{aligned}
        a_0 u_0 + a_1 u_1 + a_2 u_2 + a_3 u_3 + a_4 u_4 + a_5 u_5 &= 0 \\
        b_0 u_0 + b_1 u_1 + b_2 u_2 + b_3 u_3 + b_4 u_4 + b_5 u_5 &= 0
    \end{aligned}
    \right\rbrace \subset \PP^5.
\end{equation}
We observe that solutions $(x,y)\in (\C^*)^2$ of \eqref{eq:sys_full} are one-to-one with points $u$ such that
\begin{equation} \label{eq:incl1}
(u_0: \cdots: u_5) \, \, \in \, \, \im \Phi_\triangle \cap L_\triangle \, \, \subset \,\, X_\triangle \cap L_\triangle.    
\end{equation}
The right-hand-side of the inclusion counts the number of solutions of the homogenized version of \eqref{eq:sys_full} in projective space. Its cardinality, for generic parameters, is $|X_\triangle \cap L_\triangle| = \deg X_\triangle = \deg \nu_2\left(\PP^2 \right) = 4$. As discussed above, the inclusion \eqref{eq:incl1} is generically an equality. Intuitively, by this procedure we moved from non-linear equations \eqref{eq:sys_full} on the flat space $(\C^*)^2$ to linear equations \eqref{eq:Ltriangle} on the curvy space $X_\triangle \subset \PP^5$.
We now force the coefficients $a_4, a_5, b_4, b_5$ to be zero. That is, we consider the polynomial system $f_1=f_2=0$, where 
\begin{align}\label{eq:sys_sparse}
    f_1 &= a_0 + a_1 x + a_2 y + a_3 x y \\
    f_2 &= b_0 + b_1 x + b_2 y + b_3 x y.
\end{align}
Again, the system is defined by quadratic equations. B\'ezout's theorem predicts 4 solutions. However, one easily checks that there are at most two solutions in $(\C^*)^2$ (we leave this as an exercise). To explain this discrepancy, we homogenize $f_1$ and $f_2$ to obtain the quadrics 
\begin{align}\label{eq:sys_sparse_homog}
    F_1 &= a_0 Z^2+ a_1 XZ + a_2 YZ + a_3 XY \\
    F_2 &= b_0 Z^2+ b_1 XZ + b_2 YZ + b_3 XY
\end{align}
with homogeneous coordinates $(X:Y:Z)$ on $\PP^2$.   
This reveals the two 'missing' solutions: the points $(0:1:0)$ and $(1:0:0)$ are solutions to $F_1 = F_2 = 0$ for any choice of the parameters $a_i$ and $b_i$. 
Forcing some coefficients to be zero causes some solutions to lie on the boundary of $(\C^*)^2$ in $\PP^2$. The linear space $\tilde{L}_\triangle$ given by \eqref{eq:Ltriangle} with $a_4 = a_5 = b_4 = b_5 = 0$ is non-generic. In particular, there is a \emph{strict} inclusion $\im \Phi_\triangle \cap \tilde{L}_\triangle \subsetneq X_\triangle \cap \tilde{L}_\triangle$. Indeed, interpreting \eqref{eq:sys_sparse} as a system of equations on $X_\triangle$, we introduce two spurious intersection points in $(X_\triangle \setminus \im \Phi_\triangle) \cap \tilde{L}_\triangle$, independently of the choice of $a_0, \ldots, a_3, b_0, \ldots, b_3$. To remedy this, we define a new map
\begin{align}
    \Phi_\square : (\C^*)^2 \to \PP^3  \quad \text{given by} \quad 
    (x,y) \mapsto (1 : x : y : xy)
\end{align}
and let $X_\square$ be the Zariski closure of its image. One checks that $X_\square = \sigma(\PP^1 \times \PP^1)$ is the Segre embedding of $\PP^1 \times \PP^1$ in $\PP^3$. Our new linear space is the line
\begin{equation}
    L_\square = \left\lbrace
    \begin{aligned}
        a_0 u_0 + a_1 u_1 + a_2 u_2 + a_3 u_3 &= 0 \\
        b_0 u_0 + b_1 u_1 + b_2 u_2 + b_3 u_3 &= 0
    \end{aligned}
    \right\rbrace \subset \PP^3.
\end{equation}
For general coefficients, $|X_\square \cap L_\square| = \deg X_\square = 2$ and the inclusion $\im \Phi_\square \cap L_\square \subset X_\square \cap L_\square$ is an equality. The conclusion is that a natural compact space containing the solutions of \eqref{eq:sys_sparse} is $X_\square \simeq \PP^1 \times \PP^1$, rather than $X_\triangle \simeq \PP^2$.

We point out that this example features several familiar toric varieties: $$(\C^*)^2, \, \C^2, \, \PP^2 \simeq X_\triangle, \, \PP^5, \, \PP^1 \times \PP^1 \simeq X_\square, \, \PP^3.$$
All of these are \emph{normal toric varieties} (in particular, they are smooth). We will see that such varieties have a very nice combinatorial description in terms of fans, and sometimes in terms of polytopes. As a teaser, we justify the notation $X_\triangle, X_\square$. The \emph{Newton polytopes} (see Section \ref{sec:homog}) $\triangle$ and $\square$ of the equations in \eqref{eq:sys_full} and \eqref{eq:sys_sparse} are shown in Figure \ref{fig:newton_polytopes}.
\begin{figure}[h!]
    \centering
    \begin{subfigure}{0.4\textwidth}
    \centering
    \begin{tikzpicture}[scale=1]
\begin{axis}[%
width=1.5in,
height=1.5in,
scale only axis,
xmin=-0.5,
xmax=2.5,
ymin=-0.5,
ymax=2.5,
ticks = none, 
ticks = none,
axis background/.style={fill=white},
axis line style={draw=none} 
]


\addplot [color=mycolor1,solid,fill opacity=0.2,fill = mycolor4,forget plot]
  table[row sep=crcr]{%
 0 0\\
0 2\\	
2 0\\
0 0 \\
};

\addplot [very thick, color=mycolor2,solid,fill opacity=0.2,fill = mycolor1,forget plot]
  table[row sep=crcr]{%
 0 0 \\
 0 2\\
};

\addplot [very thick, color=mycolor2,solid,fill opacity=0.2,fill = mycolor1,forget plot]
  table[row sep=crcr]{%
 0 1 \\
 1 1\\
};

\addplot [very thick, color=mycolor2,solid,fill opacity=0.2,fill = mycolor1,forget plot]
  table[row sep=crcr]{%
 1 0 \\
 1 1\\
};

\addplot [very thick, color=mycolor2,solid,fill opacity=0.2,fill = mycolor1,forget plot]
  table[row sep=crcr]{%
 1 0 \\
 0 1\\
};

\addplot [very thick, color=mycolor2,solid,fill opacity=0.2,fill = mycolor1,forget plot]
  table[row sep=crcr]{%
 0 0 \\
 2 0 \\
};

\addplot [very thick, color=mycolor2,solid,fill opacity=0.2,fill = mycolor1,forget plot]
  table[row sep=crcr]{%
 0 2\\
 2 0 \\
};


\addplot[only marks,mark=*,mark size=1.1pt,black
        ]  coordinates {
   (0,0) (1,0) (2,0) (0,1) (1,1) (2,1) (0,2) (1,2) (2,2)
};

\end{axis}
\end{tikzpicture} 
    \end{subfigure}
    \begin{subfigure}{0.4\textwidth}
    \centering
   \begin{tikzpicture}[scale=1]
\begin{axis}[%
width=1.5in,
height=1.5in,
scale only axis,
xmin=-0.5,
xmax=2.5,
ymin=-0.5,
ymax=2.5,
ticks = none, 
ticks = none,
axis background/.style={fill=white},
axis line style={draw=none} 
]


\addplot [color=mycolor1,solid,fill opacity=0.2,fill = mycolor4,forget plot]
  table[row sep=crcr]{%
 0 0\\
0 1\\	
1 1\\
1 0 \\
0 0\\
};

\addplot [very thick, color=mycolor2,solid,fill opacity=0.2,fill = mycolor1,forget plot]
  table[row sep=crcr]{%
 0 1 \\
 1 1\\
};

\addplot [very thick, color=mycolor2,solid,fill opacity=0.2,fill = mycolor1,forget plot]
  table[row sep=crcr]{%
 1 0 \\
 1 1\\
};

\addplot [very thick, color=mycolor2,solid,fill opacity=0.2,fill = mycolor1,forget plot]
  table[row sep=crcr]{%
 1 0 \\
 0 1\\
};

\addplot [very thick, color=mycolor2,solid,fill opacity=0.2,fill = mycolor1,forget plot]
  table[row sep=crcr]{%
 0 0 \\
 1 0 \\
};

\addplot [very thick, color=mycolor2,solid,fill opacity=0.2,fill = mycolor1,forget plot]
  table[row sep=crcr]{%
 0 0\\
 0 1 \\
};


\addplot[only marks,mark=*,mark size=1.1pt,black
        ]  coordinates {
   (0,0) (1,0) (2,0) (0,1) (1,1) (2,1) (0,2) (1,2) (2,2)
};

\end{axis}
\end{tikzpicture} 
    \end{subfigure}
    \caption{The Newton polytopes associated to the polynomial system \eqref{eq:sys_full} on the left and to \eqref{eq:sys_sparse} on the right.}
    \label{fig:newton_polytopes}
\end{figure}
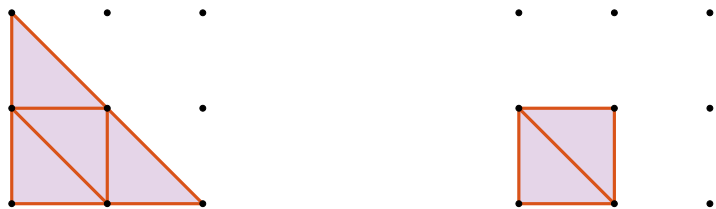
Note that from these polytopes $P$ we can deduce the following geometric information on $X_P$: 
 \begin{enumerate}
     \item the dimension of $P$ equals that of $X_P$,
     \item the degree of $X_P$ is the \emph{normalized volume} of $P$.
 \end{enumerate}
The second statement is an instance of \emph{Kushnirenko's Theorem}, see Section \ref{sec:kush}. The normalized volume of $P\subset \R^n$ equals its Euclidean volume divided by $n!$. It counts the number of standard simplices that can fit in $P$, see Figure \ref{fig:newton_polytopes}.
\end{example}

\paragraph{Acknowledgements.} 
I would like to thank all course attendants for their active participation. Some exercises suggested by the author were worked out by students and included as examples in the text. Special thanks to Chiara Meroni for helping with Examples* \ref{ex:chiara1}, \ref{ex:chiara3}, \ref{ex:chiara4}, \ref{ex:chiara5}, \ref{ex:chiara6}, \ref{ex:chiara7}, \ref{ex:chiara8}, \ref{ex:chiara9}, \ref{ex:chiara10} and for setting up an accompanying MathRepo page at \url{https://mathrepo.mis.mpg.de/ToricGeometry}. Examples* \ref{ex:claudia1}, \ref{ex:claudia2} were contributed by Claudia Fevola, and Mara Belotti and Chiara Meroni worked on Example* \ref{ex:chiaramara}. Thanks to Michael Joswig for several useful suggestions, including Example* \ref{ex:chiara7}. The lectures were supported by Jupyter notebooks written in Julia. I want to thank Lars Kastner and Martin Bies for implementing desired features and for several swift bug fixes. Their help was indispensable in underpinning this course with computational examples. Thanks to Claudia Fevola, Chiara Meroni and Pedro Tamaroff for their help with proofreading. 

\section{Affine toric varieties} \label{sec:affineTV}
Affine toric varieties are the building blocks of more general toric varieties. An \emph{affine variety} in $\C^s$ is the common zero set of finitely many polynomials in $\C[x_1, \ldots, x_s]$. For affine \emph{toric} varieties, we will see that all these polynomials can be taken of the simple form $x^a - x^b$, where $x^a = x_1^{a_1} \cdots x_s^{a_s}$,  $a = (a_1, \ldots, a_s) \in \N^s$ and similarly for $x^b$. We will also see that toric varieties come in the form of a simple parametrization given by monomials, and a more intrinsic description involves \emph{semigroup algebras}. The nicest affine toric varieties, i.e., those that are \emph{normal}, correspond to semigroups associated to \emph{rational polyhedral cones}. 

An affine variety $V \subset \C^s$, has coordinate ring $\C[V]$. Its vanishing ideal in $\C[x_1, \ldots, x_s]$ is denoted $I(V)$, i.e.~$\C[V] \simeq \C[x_1, \ldots, x_s]/I(V)$. To associate an affine variety to a nilpotent free finitely generated $\C$-algebra $R$, we use the notation $R \mapsto \Specm(R)$. This stands for \emph{maximal spectrum}: $\Specm(R)$ is the set of all maximal ideals of $R$, endowed with the Zariski topology. We have  $V \simeq \Specm(\C[V]) = \Specm(\C[x_1, \ldots, x_s]/I(V))$ and $\C[\Specm(R)] = R$.

\subsection{Tori and their lattices}
Recall that $(\C^*)^n$ is a group under component-wise multiplication: for $t, t' \in (\C^*)^n$,
\[ t \cdot t' = (t_1, \ldots, t_n) \cdot (t'_1, \ldots, t'_n) = (t_1t'_1, \ldots, t_nt'_n) \in (\C^*)^n. \] 

\begin{definition}[Torus]
A \emph{torus} $T$ is an affine variety which is isomorphic to $(\C^*)^n$, where $T$ inherits a group structure from this isomorphism.
\end{definition}

By definition, the coordinate ring $\C[T]$ of an $n$-dimensional torus $T$ is isomorphic to the \emph{Laurent polynomial ring} $\C[t_1^{\pm 1}, \ldots, t_n^{\pm 1}]$. These are the morphisms $T \rightarrow \C$. Among them, \emph{characters} are the non-vanishing morphisms on $T$ which respect the group structure.
\begin{definition}[Character]
A \textit{character} of a torus $T$ is a morphism $T \rightarrow \C^*$ which is a group homomorphism. 
\end{definition}

\begin{proposition}\label{prop:characters}
The characters of $(\C^*)^n$ are the morphisms $(\C^*)^n \rightarrow \C^*$ given by
$$(t_1,\dots,t_n) \mapsto t_1^{a_1}\cdots t_n^{a_n} \qquad \text{for } (a_1,\dots a_n)\in \Z^n.$$
\end{proposition}
\begin{proof}
The only morphisms from $(\C^*)^n \to \C^*$ are those given by $t \rightarrow \lambda t^a$, for $a \in \Z^n$ and $\lambda \in \C^*$ (see \cite[Prop. 2.3]{michalek2018selected}). Among those, the 
group homomorphisms have $\lambda = 1$. 
\end{proof}

\begin{corollary}\label{cor:freeAbelianGroup}
The characters of a torus $T\simeq (\C^*)^n$ form a lattice $M\simeq \Z^n$, called the \emph{character lattice} of $T$.
\end{corollary}

By a \emph{lattice} we mean a free abelian group of finite rank. Conventionally, for a point $m\in M$ in the character lattice, we write $\chi^m:T\rightarrow \C^*$ for the corresponding character. Note that when $T = (\C^*)^n$, Proposition \ref{prop:characters} says that characters are Laurent monomials.

The dual lattice $N = \text{Hom}_{\Z}(M,\Z) \simeq \Z^n$ of the character lattice $M$ is called the \textit{cocharacter lattice} or the \textit{lattice of one-parameter subgroups} of $T$. More explicitly, $N$ consists of the morphisms $\C^* \rightarrow T$ that are group homomorphisms. 

\begin{proposition} 
The one-parameter subgroups of $(\C^*)^n$ are given by the morphisms 
$$t \mapsto (t^{u_1},\dots,t^{u_n}) \qquad \text{for } u = (u_1,\dots, u_n)\in \Z^n.$$
\end{proposition}
We write $\lambda^u: \C^* \rightarrow T$ for the one-parameter subgroup corresponding to $u\in \Z^n$. When $T = (\C^*)^n$, a one-parameter subgroup is a monomial curve. 

An isomorphism $\phi:T \simeq (\C^*)^n$ determines a basis $e_1,\dots, e_n$ of $M$ and a dual basis ${e'}_1,\dots, e'_n$ of $N$. Explicitly, $e_i$ corresponds to the character $\chi^{e_i}(t) = \phi(t)_i$, and $e_i'$ represents the cocharacter $\lambda^{e_i'}(t) = \phi^{-1}(1, \ldots, t, \ldots, 1)$, where $t$ is in the $i$-th position.
The bases $\{e_i\}$ and $\{e_i'\}$ fix isomorphisms $\phi_M: M\to \Z^n$ and $\phi_N: N \to \Z^n$. The
natural pairing: 
$$\langle \cdot, \cdot \rangle: M\times N \rightarrow \Hom_\Z(\Z, \Z) \simeq \Z, \quad (m,n)\mapsto \chi^m \circ \lambda^n$$
is given in coordinates by $\langle m,u\rangle = \phi_M(m) \cdot \phi_N(u)$, where $ \cdot $ is the usual dot product in $\Z^n$. 

\begin{example}
We work out an explicit example. Let $n = 2$ and take $m = 2 e_1 + 3 e_2 \in M, u = e_1' - 4e_2' \in N$. We have $\phi_M(m) = (2,3), \phi_N(u) = (1,-4)$. From 
\begin{align*}
(\chi^m \circ \lambda^u)(t) &=\chi^m ( \lambda^{e_1' - 4 e_2'}(t)) =\chi^m(\lambda^{e_1'}(t) \lambda^{e_2'}(t)^{-4}) = \chi^m(\phi^{-1}(t,t^{-4})) \\ &= \chi^{2e_1 + 3e_2}(\phi^{-1}(t,t^{-4})) = \chi^{e_1}(\phi^{-1}(t,t^{-4}))^2 \chi^{e_2}(\phi^{-1}(t,t^{-4}))^3 = t^2 t^{-12} = t^{-10},
\end{align*}
we see that $\chi^m \circ \lambda^u: \Z \rightarrow \Z$ sends $t \mapsto t^{-10}$. Hence $\langle m, u \rangle = -10 = (2,3) \cdot (1, -4)$. 
\end{example}
Given a lattice $N \simeq \Z^n$, the torus $T_N = N \otimes_\Z \C^* \simeq (\C^*)^n$ has cocharacter lattice $N$. One-parameter subgroups are given explicitly by $\lambda^u(t) = u \otimes t$. The torus $T_N$ is canonically associated to $N$, which we emphasize with the subscript $N$. A basis $e_1', \ldots, e_n'$ of $N$ fixes an isomorphism $T_N \rightarrow (\C^*)^n$ which sends $u \otimes t = \left( \sum_{i = 1}^n u_i  \, e_i'  \right ) \otimes t$ to $(t^{u_1}, \ldots, t^{u_n})$.

\begin{remark}
One can often assume that dual bases for the lattices $M$, $N$ are fixed. In such cases, and in most statements below, one can set $M = N = \Z^n$ and $T  = T_N = (\C^*)^n$. Characters are Laurent monomials and co-characters are monomial curves. The general set-up presented here will be useful in contexts where several tori and lattices are involved.
\end{remark}
\subsection{Definition of an affine toric variety}
\begin{definition}[Affine toric variety] \label{def:ATV}
An \textit{affine toric variety} is an irreducible affine variety $V$ containing a torus $T \simeq (\C^*)^n$ as a Zariski dense open subset such that the action of $T$ on itself extends to an algebraic action of $T$ on $V$.
\end{definition}
In this definition, an \emph{algebraic action} is a morphism of varieties $T\times V\rightarrow V$.
\begin{example}
Here are some examples of affine toric varieties:
\begin{itemize}
    \item A torus $T$ is an affine toric variety. In particular, the action of $T$ on itself is given by a morphism $T \times T \rightarrow T$, as $(\C^*)^n \times (\C^*)^n \rightarrow (\C^*)^n$ is given by Laurent monomials.  
    \item Affine $n$-space $\C^n$ is an affine toric variety. It has a dense torus $T = (\C^*)^n \hookrightarrow \C^n$ whose group operation $T \times T \rightarrow T$ extends trivially to an algebraic action $T \times \C^n \rightarrow \C^n$. 
    \item The curve $C =  \{x^3-y^2 = 0 \} \subseteq \C^2$ contains the torus $T \simeq \C^*$ via $t \mapsto (t^2, t^3)$. The action of $T$ on $C$ is given by $(t, (x,y))\mapsto(t^2x,t^3y) $;
    \item The threefold $V = \{xy-zw = 0 \} \subset \C^4$ contains a torus $T = (V \cap (\C^*)^4) \simeq (\C^*)^3$ via $(t_1, t_2, t_3) \mapsto (t_1, t_2, t_3, t_1t_2t_3^{-1})$. The action of $T$ on itself extends to an algebraic action of $T$ on $V$ as $((t_1,t_2,t_3), (x,y,z,w))\mapsto(t_1 x,t_2y, t_3 z, t_1 t_2 t_3^{-1} w) $. \qedhere
\end{itemize}
\end{example}
Definition \ref{def:ATV} seems somewhat technical. Below, we will discuss three different, very explicit ways of constructing affine toric varieties. The first construction uses monomial maps, the second uses toric ideals, and the third involves affine semigroups.

\subsection{Affine toric varieties from monomial maps} \label{subsec:affmonomial}
Let $T\simeq (\C^*)^n$ be a torus with character lattice $M$ and fix a subset $\mathscr{A}=\{m_1,\dots , m_s\}\subset M$. We consider the map
\begin{align*}
    \Phi_{\mathscr{A}}\;:\; T \to \C^s \quad \text{given by} \quad  t \mapsto (\chi^{m_1}(t),\dots, \chi^{m_s}(t) ).
\end{align*}
We denote the Zariski closure of its image by $Y_{\mathscr{A}} = \overline{\text{im} \, \Phi_{\mathscr{A}}}\subset \C^s$. 

\begin{proposition}\label{prop:charLattice}
$Y_{\mathscr{A}}$ is an affine toric variety whose dense torus has character lattice $\Z \mathscr{A} = \{ \sum_{i=1}^s a_i m_i ~|~ a_i \in \Z \}$.
\end{proposition}
For proving this proposition, we will need a few lemmas. 
We will sometimes think of $\Phi_\A$ as a map of tori $T \to (\C^*)^s$. 
\begin{lemma} \label{lem:phihat}
A morphism of tori $\phi: T_1 \rightarrow T_2$ that is also a group homomorphism induces a map $\widehat{\phi}: M_2 \rightarrow M_1$, where $M_i$ is the character lattice of $T_i$. This map is such that $\chi^{\widehat{\phi}(m_2)}$ is the character $\chi^{m_2} \circ \phi$. Moreover, if $\phi$ is injective, $\widehat{\phi}$ is surjective and vice versa.  
\end{lemma}
\begin{exercise} \label{ex:phihat}
Prove Lemma \ref{lem:phihat} and check that the map $\widehat{\Phi}_\A : \Z^s \rightarrow M$ is given by $a = (a_1, \ldots, a_s) \mapsto \sum_{i=1}^s a_i m_i$. 
\end{exercise}
In the proof of Proposition \ref{prop:charLattice}, we will also use the following result from the theory of linear algebraic groups \cite[Sec.~16]{humphreys2012linear}.

\begin{lemma}\label{lem:imageSubtorus}
Let $\phi: T_1\to T_2$ be a morphism of tori that is a group homomorphism. The image $\phi(T_1)\subset T_2$ is a closed subtorus.
\end{lemma}

\begin{proof}[proof of Proposition \ref{prop:charLattice}]
Let us denote $T'= \im \Phi_{\mathscr{A}}$. The lemma above implies that $T'$ is a closed subtorus of $(\C^*)^s$. Hence we have $T' = \overline{T'}\cap (\C^*)^s = Y_{\mathscr{A}} \cap (\C^*)^s$ where the closure is taken in $\C^s$. This shows that $Y_{\mathscr{A}} = \overline{T'}$ is irreducible and contains a Zariski dense torus. We need to prove that the action of $T'$ on itself extends to $Y_{\mathscr{A}}$. An element $t \in T' \subset (\C^*)^s$ acts on $\C^s$ by coordinate-wise scaling. In particular, if $V \subset \C^s$ is a closed subvariety, then so is $t \cdot V$ for $t \in T'$. Since
$T' = t\cdot T' \subset t\cdot Y_{\mathscr{A}}$ and $Y_{\mathscr{A}}=\overline{T'}$, we find $Y_{\mathscr{A}}\subset t \cdot Y_{\mathscr{A}}$ by taking the closure on both sides. Replacing $t$ with its inverse $t^{-1}$, we see that $Y_{\mathscr{A}}\subset t^{-1} \cdot Y_{\mathscr{A}}$ and hence $t\cdot Y_{\mathscr{A}} \subset Y_{\mathscr{A}} \subset t\cdot Y_{\mathscr{A}}$. This shows that the action extends, and it is given by a morphism obtained from restricting $(\C^*)^s \times \C^s \rightarrow \C^s$ to $T' \times Y_\A \rightarrow Y_\A$. In order to understand the character lattice of $Y_{\mathscr{A}}$, we let $M'$ be the character lattice of $T'$ and consider the diagrams 
\begin{center}
\begin{tikzcd}
T \ar[r] \ar[dr,->>]
    & (\C^*)^s  \\
    & T' \ar[u,hookrightarrow]
\end{tikzcd}
\hspace{2cm}
\begin{tikzcd}
M 
    & \Z^s \ar[l] \ar[d,->>] \\
    & M' \ar[ul,hookrightarrow]
\end{tikzcd}
\end{center}
where the right diagram is induced by the left diagram via Lemma \ref{lem:phihat}. Since the image of $\Z^s \rightarrow M$ is $\Z \A$ (Exercise \ref{ex:phihat}), we see that $M' \simeq \Z \mathscr{A}$.
\end{proof}
\begin{corollary} \label{cor:dimYA}
The dimension of $Y_\A$ is the rank of the sublattice $\Z \A \subset M$. 
\end{corollary}
If $T = (\C^*)^n$, then $m_i \in \Z^n$ are integer vectors. Collecting these in the columns of a matrix $A \subset \Z^{n \times s}$, Corollary \ref{cor:dimYA} implies that $\dim Y_\A = \textup{rank} \, A$. Moreover, the character lattice of the dense torus of $Y_\A$ is the image of $A$ in $\Z^n$. 

\begin{example} \label{ex:parabola}
Let $n = 1, T = \C^*$ and $\A = \{ 2, 4 \} \subset M = \Z$. The map $\Phi_\A$ is given by $\Phi_\A(t) = (t^2, t^4)$. The affine toric variety $Y_\A$ is the parabola $\{y-x^2 = 0\} \subset \C^2$. Its dense torus is $T' = Y_\A \setminus \{0\}$. The matrix $A$ is $[2 ~ ~ 4]: \Z^2 \rightarrow \Z$, which has rank 1. Its image is $2 \Z \subset \Z$, which is the character lattice of $T'$. The same affine toric variety is obtained from $\A = \{1, 2\}$, for which $A : \Z^2 \rightarrow \Z$ is surjective and $\Phi_\A: T \rightarrow T'$ is an isomorphism. 
\end{example}

\begin{exercise}
Let $n = 2, T = (\C^*)^2$ and $\A = \{(1,0),(0,1),(1,1)\}$. What is the dimension of $Y_\A$? What are its defining equations? Describe $Y_\A \setminus T'$ and the character lattice $M'$ of $T'$. Answer the same questions for $\A = \{ (1,1), (2,2), (3,3) \}$. 
\end{exercise}

\begin{exercise} \label{ex:closed}
Find an example where $s \geq 2$ and the image of $\Phi_\A$ is closed in $\C^s$.
\end{exercise}

\subsection{Affine toric varieties from toric ideals}
In this section our goal is to describe the vanishing ideal $I(Y_\A) \subset \C[x_1, \ldots, x_s]$ of the affine toric variety $Y_\A$. A polynomial $f \in I(Y_\A)$ vanishes on the image of the monomial map $\Phi_\A$. A key observation is that, if $\sum_{i=1}^s a_i m_i = \sum_{i=1}^s b_i m_i$, then $x^a - x^b$ has this property. Indeed, plugging in $x = (\chi^{m_1}(t), \ldots, \chi^{m_s}(t)) \in \im \Phi_\A$, we get $\prod_{i=1}^s \chi^{m_i}(t)^{a_i} - \prod_{i=1}^s \chi^{m_i}(t)^{b_i} = \chi^{\sum_{i=1}^s a_i m_i}(t) - \chi^{\sum_{i=1}^s b_i m_i}(t) = 0$. By Exercise \ref{ex:phihat}, this shows that for every $a, b \in \N^s$ such that $a - b \in \ker \widehat{\Phi}_\A$, we find a binomial $x^a - x^b \in I(Y_\A)$. 

If $\ell \in \ker \widehat{\Phi}_\A \subset \Z^s$, we can write $\ell = \blue{\ell_+} - \red{\ell_-}$ where $\blue{\ell_+} = \sum_{\ell_i > 0} \ell_i e_i$ and $\red{\ell_-} = - \sum_{\ell_i < 0} \ell_i e_i$, with $e_i$ the standard basis vectors of $\Z^s$. By construction, we have that $x^{\blue{\ell_+}} - x^{\red{\ell_-}} \in I(Y_\A)$.

\begin{example}
Let $n =1, \A = \{1,2,3\}, \widehat{\Phi}_\A = A = [1~2~3] : \Z^3 \rightarrow \Z^1$. The vectors 
\[ (\blue{1}, \red{-2}, \blue{1}), \quad (\blue{1},\blue{1}, \red{-1}), \quad (\blue{2}, \red{-1}, 0)\]
are in $\ker \widehat{\Phi}_\A $. This gives binomials $x^{\blue{1}} z^{\blue{1}} - y^{\red{2}}, x^{\blue{1}} y^{\blue{1}} - z^{\red{1}}, x^{\blue{2}} - y^{\red{1}} \in I(Y_\A)$. Note that $Y_\A$ is the twisted cubic curve.
\end{example}
It turns out that $I(Y_\A)$ is generated by binomials that are found in this way. 

\begin{theorem} \label{thm:toricideal}
The vanishing ideal $I(Y_\A)$ of the affine toric variety $Y_\A \subset \C^s$ is given by 
\begin{align*} 
 I(Y_\A) &= \langle x^{\blue{\ell_+}} - x^{\red{\ell_-}} ~|~ \ell \in \ker \widehat{\Phi}_\A \rangle  \\
 &=  \langle x^a - x^b ~|~ a - b \in \ker \widehat{\Phi}_\A \rangle \subset \C[x_1, \ldots, x_s].
 \end{align*}
\end{theorem} 

We omit the proof of Theorem \ref{thm:toricideal} and refer to \cite[p.~15]{cox2011toric}. 

\begin{exercise}
Let $n = 1, T = \C^*, \A = \{2, 3\}$. Show that $Y_\A$ is the cuspidal cubic, i.e. $I(Y_\A) = \langle x^3 - y^2 \rangle$. 
\end{exercise}

\begin{exercise} \label{ex:C2inC3}
Let $n = 2, T =( \C^*)^2, \A = \{(1,0), (0,1),(1,1)\}$. Show that $Y_\A$ is given by $I(Y_\A) = \langle xy - z \rangle$. 
\end{exercise}

\begin{exercise} \label{ex:contains1}
Show that $Y_\A$ always contains $(1,\ldots,1) \in \C^s$.
\end{exercise}

\begin{definition}[Toric ideal]
A \emph{toric ideal} in $\C[x_1, \ldots, x_s]$ is a prime ideal that is generated by binomials. 
\end{definition}

As a corollary of Theorem \ref{thm:toricideal}, any affine toric variety coming from a monomial map $\Phi_\A$ is defined by a toric ideal. Computing toric ideals is an important task from \emph{combinatorial commutative algebra}. See \cite[Chapter 12]{sturmfels1996grobner} for algorithms. These algorithms are implemented in software systems such as \texttt{Macaulay2}, \texttt{Singular} and \texttt{Oscar.jl}.

\begin{examplestar} \label{ex:chiara1}
We illustrate how to compute $I(Y_\A)$ in \texttt{Oscar.jl}.  Consider the matrix
\begin{equation}
    A = 
    \begin{bmatrix}
    2 &2 &1 &0 &0 &1 &1 \\
    1 &0 &0 &1 &2 &2 &1 \\
    0 &1 &2 &2 &1 &0 &1
    \end{bmatrix}
\end{equation}
from \cite[Example 8.9]{michalek2021invitation}. The associated toric variety $Y_{\A}$ has dimension $3$ and is embedded in $\C^7$. To compute the toric ideal $I(Y_{\A})$, we use the function \texttt{toric\_ideal}. 
An example of in- and output in a Jupyter Notebook is shown in Figure \ref{fig:output1}. Note that, according to \texttt{Oscar} conventions, \texttt{toric\_ideal} requires the \emph{transpose} of $A$ as an input. The result, as in \cite[Example 8.9]{michalek2021invitation}, is an ideal generated by 9 binomials. We reiterate that the code of this example, and all other examples marked by an asterisk, is available at \url{https://mathrepo.mis.mpg.de/ToricGeometry}.
\begin{figure}
\centering
\includegraphics[scale=0.3]{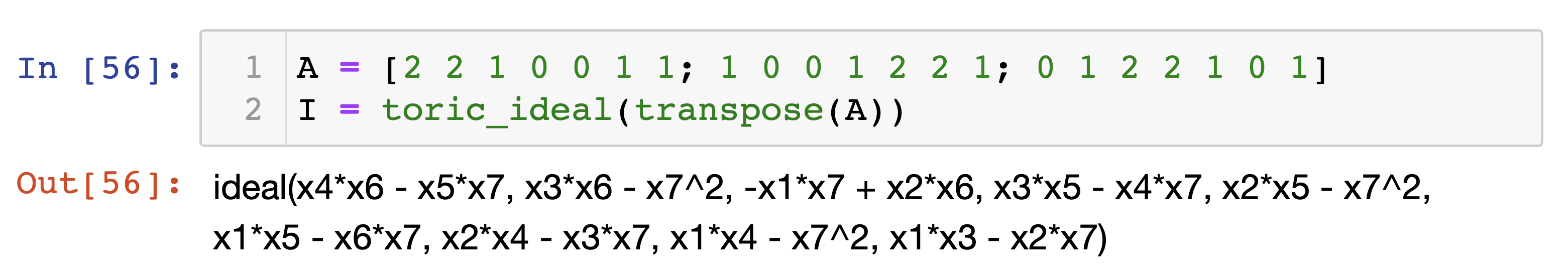}
\caption{Computing toric ideals in \texttt{Oscar.jl}.}
\label{fig:output1}
\end{figure}
\end{examplestar}

\subsection{Affine toric varieties from affine semigroups}
In the previous sections, we have described $Y_\A \subset \C^s$ as the closure of the image of a monomial map, and as the zero set of a toric ideal. In this section, we will describe its coordinate ring in a more intrinsic way. In particular, we will see that a choice of embedding corresponds to a choice of generators of an \emph{affine semigroup}.
\begin{definition}[Affine semigroup] \label{def:affsemgroup}
An \emph{affine semigroup} is a set $\S$ with a binary operation $+$ such that 
\begin{enumerate}
\item $+$ is associative, 
\item $\S$ has an identity element $0 \in \S$, 
\item $+$ is commutative, 
\item $\S$ is finitely generated: there is a finite subset $\A \subset \S$ such that 
\[ \N \A = \left \{ \sum_{m \in \A} a_m m ~|~ a_m \in \N \right \} = \S, \]
\item $\S$ can be embedded in a lattice $M$.
\end{enumerate}
\end{definition}
Note that only associativity is usually required for a semigroup. The adjective \emph{affine} adds all four other requirements to the list. By definition, all affine semigroups (up to isomorphism) are of the form $\N \A$ for some finite subset $\A$ of a lattice $M$. 

\begin{definition}[Semigroup algebra]
Let $M$ be the character lattice of a torus $T$. The \emph{semigroup algebra} associated to an affine semigroup $\S \subset M$ is the $\C$-algebra
\[ \C[\S] = \left \{ \sum_{m \in \S} c_m \chi^m ~|~ c_m \in \C, \text{finitely many $c_m$ nonzero} \right \},\]
with multiplication induced by the group operation in $M$: $\chi^m \chi^{m'} = \chi^{m + m'}$. 
\end{definition}

\begin{example} \label{ex:fingen}
If $\A = \{ m_1, \ldots, m_s \} \subset M$, then $\C[\N \A] = \C[\chi^{m_1}, \ldots, \chi^{m_s}]$. In particular, the Laurent polynomial ring $\C[t_1^{\pm 1}, \ldots, t_n^{\pm 1}]$ is the semigroup algebra $\C[\Z^n]$. 
\end{example}

Here is how affine semigroups give rise to affine toric varieties. 
\begin{proposition} \label{prop:semigroup}
Let $\S \subset M$ be an affine semigroup. Then 
\begin{enumerate}
\item $\C[\S]$ is an integral domain and finitely generated as a $\C$-algebra, 
\item $\Specm(\C[\S])$ is an affine toric variety whose dense torus has character lattice $\Z \S$. 
\end{enumerate}

\begin{proof}
For point 1, note that $\C[M]$ is an integral domain. Indeed, it is isomorphic to the Laurent polynomial ring (Example \ref{ex:fingen}), which is the coordinate ring of an irreducible affine variety. Since $\C[\S] \subset \C[M]$, $\C[\S]$ is an integral domain as well. It is finitely generated as a $\C$-algebra by definition (Example \ref{ex:fingen}). 

For point 2, let $\S = \N \A$ for a finite set $\A \subset M$ and let $\Phi_\A^* : \C[x_1, \ldots, x_s] \rightarrow \C[M]$ be the pullback map of $\Phi_\A$, given by $f \mapsto f \circ \Phi_\A$. There is a short exact sequence 
\[ 0 \rightarrow I(Y_\A) \hookrightarrow \C[x_1 ,\ldots, x_s] \overset{\Phi_\A^*}{\longrightarrow} \C[\S] \rightarrow 0, \]
which shows $\C[\S] \simeq \C[x_1, \ldots, x_s]/I(Y_\A)$. This proves that $\Specm(\C[\S])$ is isomorphic to $Y_\A$. Since $\Z \A = \Z \S$, the torus of $\Specm(\C[\S])$ has character lattice $\Z \S$ by Proposition \ref{prop:charLattice}.
\end{proof}
\end{proposition}

The following statement affirms that all affine toric varieties come from monomial maps, toric ideals and affine semigroups. 

\begin{theorem} \label{thm:classaff}
Let $V$ be an affine variety. The following are equivalent: 
\begin{enumerate}
\item $V$ is an affine toric variety,
\item $V \simeq Y_\A$ for a finite subset $\A$ of a lattice $M$,
\item $V$ is defined by a toric ideal,
\item $V = \Specm(\C[\S])$ for some affine semigroup $\S$.
\end{enumerate}
\end{theorem}
\begin{proof}
This is Theorem 1.1.17 in \cite{cox2011toric}. The equivalence of 2 - 4 can be seen as follows. If $V \simeq Y_\A$, then $V \simeq \Specm(\C[x_1,\ldots,x_s]/I(Y_\A))$, where $I(Y_\A)$ is the toric ideal from Theorem \ref{thm:toricideal}. This implies $V \simeq \Specm(\C[\N \A])$ by the short exact sequence in the proof of Proposition \ref{prop:semigroup}. The same sequence implies that $V \simeq Y_\A$, which gives $2 \Rightarrow 3 \Rightarrow 4 \Rightarrow 2$.
\end{proof}

\subsection{Rational convex polyhedral cones} \label{subsec:cones}
We have seen that all affine toric varieties arise as $\Specm(\C[\S])$ for some affine semigroup $\S$. We will see that the nicest affine toric varieties are those for which the semigroup algebra $\C[\S]$ is a \emph{normal} domain. These are affine toric varieties coming from polyhedral cones. Before introducing them, we present some notation and concepts related to such cones. \\

To a lattice $N \simeq \Z^n$ we associate the real vector space $N_\R = N \otimes_\Z \R \simeq \R^n$. The dual vector space $M_\R$ comes from $M = \Hom_\Z(N,\Z)$: $M_\R = \Hom_\R(N_\R,\R) =  M \otimes_\Z \R \simeq \R^n$. 

\begin{definition}[Rational convex polyhedral cone]
A \emph{rational convex polyhedral cone} in $N_\R$ is a set of the form
\[ \sigma = {\rm Cone}(S) = \left \{ \sum_{u \in S} \lambda_u \, u ~|~ \lambda_u \in \R_{\geq 0}  \right \}, \]
for some finite set $S \subset N$. The set $S$ is called a \emph{set of generators} of $\sigma$. 
\end{definition}
More generally, a convex cone is a set of the form ${\rm Cone}(S)$ for $S \subset N_\R$. Requiring $S$ to be finite makes the cone polyhedral, and requiring $S \subset N$ makes it rational. The cones related to toric varieties are rational, convex and polyhedral. In what follows, we will omit the adjectives `polyhedral' and `convex' and simply write `cone' in the statements that hold for any convex polyhedral cone. When we write `rational cone', we mean `rational convex polyhedral cone'. Given a cone $\sigma \subset N_\R$, its \emph{dual cone} $\sigma^\vee$ is 
\[ \sigma^\vee = \{ m \in M_\R ~|~ \langle u, m \rangle \geq 0 \text{ for all } u \in \sigma \}. \]
We have $(\sigma^\vee)^\vee = \sigma$, and if $\sigma$ is a rational cone then so is $\sigma^\vee$. For a point $m \in M_\R$, we set
\[ H_m = \{ u \in N_\R ~|~ \langle u, m \rangle = 0 \} \quad \text{and} \quad H_m^+ = \{ u \in N_\R ~|~ \langle u, m \rangle \geq 0 \}.\]
\begin{definition}[Face] A \emph{face} $\tau$ of a cone $\sigma$ is a subset of the form $\tau = \sigma \cap H_m$ for $m \in \sigma^\vee$. 
\end{definition}
Faces of (rational) cones are again (rational) cones.
Note that $\sigma$ is considered to be face of itself, as $0 \in \sigma^\vee$. If $\tau$ is a face of $\sigma$, we write $\tau \preceq \sigma$ and $\tau \prec \sigma$ if $\tau \neq \sigma$. 
\begin{example} \label{ex:coneex}
Let $N = \Z^2$. The cone $\sigma = {\rm Cone}((0,1),(1,2),(2,1))$ and its dual are shown in Figure \ref{fig:coneex}. The point $m = (-1,2) \in \sigma^\vee$ gives the face $\tau = \sigma \cap H_m = {\rm Cone}((2,1))$.
\end{example}
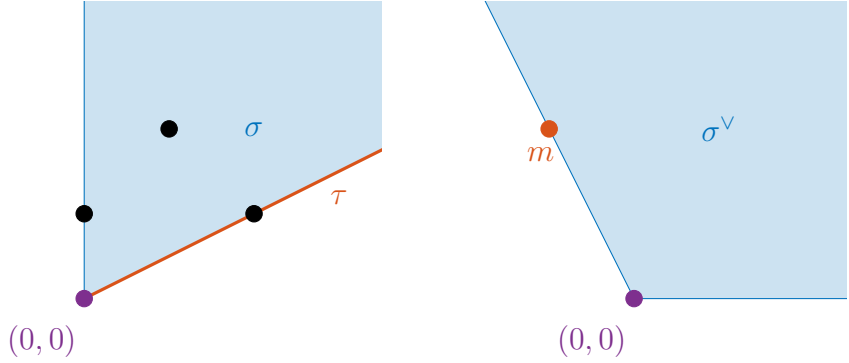
\begin{figure}
\centering
\begin{tikzpicture}[scale=1]
\begin{axis}[%
width=2in,
height=2in,
scale only axis,
xmin=-1.0,
xmax=3.5,
ymin=-1.0,
ymax=3.5,
ticks = none, 
ticks = none,
axis background/.style={fill=white},
axis line style={draw=none} 
]


\addplot [color=mycolor1,solid,fill opacity=0.2,fill = mycolor1,forget plot]
  table[row sep=crcr]{%
 0 5\\
10	 5\\	
0 0\\
0 5 \\
};

\addplot [very thick, color=mycolor2,solid,fill opacity=0.2,fill = mycolor1,forget plot]
  table[row sep=crcr]{%
0 0 \\
 10 5 \\
};

\addplot[only marks,mark=*,mark size=3.1pt,black
        ]  coordinates {
  (0,1) (1,2) (2,1)
};

\addplot[only marks,mark=*,mark size=3.1pt,mycolor4
        ]  coordinates {
  (0,0)
};

\node (P) at (axis cs:2,2) {$\textcolor{mycolor1}{\sigma}$};
\node (P) at (axis cs:3,1.2) {$\textcolor{mycolor2}{\tau}$};
\node (P) at (axis cs:-0.5,-0.5) {$\textcolor{mycolor4}{(0,0)}$};

\end{axis}
\end{tikzpicture} 
\qquad 
\begin{tikzpicture}[scale=1]
\begin{axis}[%
width=2in,
height=2in,
scale only axis,
xmin=-2.0,
xmax=2.5,
ymin=-1.0,
ymax=3.5,
ticks = none, 
ticks = none,
axis background/.style={fill=white},
axis line style={draw=none} 
]


\addplot [color=mycolor1,solid,fill opacity=0.2,fill = mycolor1,forget plot]
  table[row sep=crcr]{%
 5 0\\
5	 5\\	
-2.5 5\\
0 0 \\
5 0 \\
};

\addplot[only marks,mark=*,mark size=3.1pt,mycolor2
        ]  coordinates {
  (-1,2) 
};

\addplot[only marks,mark=*,mark size=3.1pt,mycolor4
        ]  coordinates {
  (0,0)
};

\node (P) at (axis cs:1,2) {$\textcolor{mycolor1}{\sigma^\vee}$};
\node (P) at (axis cs:-1.1,1.7) {$\textcolor{mycolor2}{m}$};
\node (P) at (axis cs:-0.5,-0.5) {$\textcolor{mycolor4}{(0,0)}$};

\end{axis}
\end{tikzpicture} 
\caption{A rational cone $\sigma$ and its dual. }
\label{fig:coneex}
\end{figure}

A useful way of representing cones, dual to a representation in terms of generators, is by a finite set of linear inequalities.
\begin{proposition}
The set $\{m_i \}_{i = 1}^s \subset M$ generates $\sigma^\vee$ if and only if $\sigma = H_{m_1}^+ \cap \cdots \cap H_{m_s}^+$. 
\end{proposition}
In the following, we summarize a list of properties of cones and their faces. For proofs, we refer to \cite[Sec.~1.2]{cox2011toric} and references therein.
\begin{proposition} \label{prop:bigconesprop}
Let $\sigma, \tau, \tau'$ be cones. Then we have 
\begin{enumerate}
\item $\tau \preceq \sigma$ and $\tau' \preceq \sigma$ implies $\tau \cap \tau' \preceq \sigma$,
\item $\tau \preceq \sigma$ and $\tau' \preceq \tau$ implies $\tau' \preceq \sigma$, 
\item if $\tau \preceq \sigma$ and $v, w \in \sigma$, then $v + w \in \tau$ implies $v \in \tau$ and $w \in \tau$, 
\item faces of $\sigma$ and faces of $\sigma^\vee$ are in bijective, inclusion reversing correspondence. 
\end{enumerate}
\end{proposition}

\begin{example} 
The bijection in item 4 of Proposition \ref{prop:bigconesprop} works as follows. Let $\tau \preceq \sigma^\vee$ be a face and let $m \in \tau$ be a generic point of $\tau$. More precisely, we pick $m$ such that $m$ is on none of the faces of $\tau$ except for $\tau$ itself. Then $H_m \cap \sigma$ is the corresponding face of $\sigma$. 
\end{example}
The dimension $\dim \sigma$ of a cone $\sigma \subset N_\R$ is the dimension of the smallest linear subspace of $N_\R$ containing $\sigma$. Here are some further properties of cones we will need. 
\begin{definition}[Strong convexity] \label{def:strongconvex}
A cone $\sigma \subset N_\R$ is called \emph{pointed} or \emph{strongly convex} if $\sigma \cap (-\sigma) = \{0\}$. This is equivalent to $\{0\} \preceq \sigma$, and to $\dim \sigma^\vee = n$. 
\end{definition}
The one-dimensional faces $\rho \preceq \sigma$ are called the \emph{rays} of $\sigma$. If $\sigma$ is strongly convex, for each ray $\rho$ there is a unique lattice point $u_\rho$ such that $\rho \, \cap \, N = \N \cdot u_\rho$. The set $\{ u_\rho ~|~ \rho~\preceq~\sigma, \dim \rho = 1 \}$ generates $\sigma$. It is called the set of \emph{minimal generators} of $\sigma$. 
\begin{definition}[Smoothness]
A strongly convex rational cone $\sigma \subset N_\R$ is \emph{smooth} if its minimal generators are a subset of a $\Z$-basis of $N$. 
\end{definition}
\begin{definition}[Simplicial]
A strongly convex cone $\sigma$ is \emph{simplicial} if its minimal generators are linearly independent over $\R$. 
\end{definition}

\begin{example}
All two-dimensional cones are simplicial. The cone $\sigma$ from Example \ref{ex:coneex} is not smooth. Its minimal generators $(0,1)$ and $(2,1)$ span a sublattice of $\Z^2$. The cone over $\{(1,0,0),(0,1,0),(1,0,1),(0,1,1)\}$ has 4 minimal generators. Hence it is not simplicial.
\end{example}

\begin{examplestar} \label{ex:claudia1}
Here is how to work with rational convex polyhedral cones in \texttt{Oscar.jl}. We define the cone $\sigma$ from Example \ref{ex:coneex} and its dual as follows:
\vspace{-0.2cm}
\begin{minted}{julia}
σ = positive_hull([0 1;1 2;2 1]); σ_dual = polarize(σ);
\end{minted}
The dimension of $\sigma$ and its facet inequalities are obtained via
\vspace{-0.2cm}
\begin{minted}{julia}
dimσ = dim(σ); facets_σ = facets(σ);
\end{minted}
This shows that $\sigma$ has dimension two and it is given by $x_1\geq 0$ and $x_1-2x_2\leq 0 $. Computations involving polyhedral geometry in \texttt{Oscar.jl}, and thus many computations related to toric varieties, rely on the software \texttt{Polymake} \cite{gawrilow2000polymake}.
\end{examplestar}

\subsection{Affine toric varieties from cones}

Let $\sigma$ be a rational cone (that is, a rational convex polyhedral cone) in $N_\R$. We associate an affine toric variety to $\sigma$, by considering the semigroup
\[ \S_\sigma = \sigma^\vee \cap M \subset M. \]
\begin{lemma}[Gordan's lemma]
The semigroup $\S_\sigma$ is affine, in the sense of Definition \ref{def:affsemgroup}.
\end{lemma}
We omit the proof of Gordan's lemma and refer to \cite[Prop.~1.2.17]{cox2011toric}. It entails that $\S_\sigma$ is finitely generated, so that $\S_\sigma = \N \A$ for a finite subset $\A \subset M$. Here is an obvious corollary.
\begin{corollary}
Let $\sigma \subset N_\R$ be a rational cone with associated semigroup $\S_\sigma = \sigma^\vee \cap M$. Then $U_\sigma = \Specm(\C[\S_\sigma])$ is an affine toric variety. 
\end{corollary}
\begin{proposition} \label{prop:7}
Let $U_\sigma = \Specm(\C[\S_\sigma])$ be the affine toric variety corresponding to $\sigma \subset N_\R$. We have that $\dim U_\sigma = n$ if and only if $\sigma$ is pointed. In this case, the torus of $U_\sigma$ has character lattice $M = \Hom_\Z(N,\Z)$.
\end{proposition}
\begin{proof}
By Proposition \ref{prop:charLattice}, $\dim U_\sigma = {\rm rank} \, \Z S_\sigma$. By the proof of \cite[Thm.~1.2.18]{cox2011toric}, we have \[{\rm rank} \, \Z \S_\sigma = n \Longleftrightarrow \Z \S_\sigma = M \Longleftrightarrow \sigma \text{ is strongly convex}. \qedhere \]
\end{proof}
\begin{example} \label{ex:smooth}
Let $\sigma = {\rm Cone}(e_1, \ldots, e_r) \subset \R^n = (\Z^n)_\R$, where $r \leq n$ and $e_i$ is the $i$-th standard basis vector. Note that $\sigma$ is pointed. One checks that $\sigma^\vee = {\rm conv}(e_1', \ldots, e_r', \pm e_{r+1}', \ldots, \pm e_n')$, where $e_i'$ is the dual basis vector of $(\R^n)^\vee$. The associated affine toric variety $U_\sigma$ is 
\[ \Specm(\C[x_1,\dots,x_r, x_{r+1}^{\pm 1}, \ldots, x_n^{\pm 1}]) \simeq \C^r \times (\C^*)^{n-r}. \]
By changing bases in the lattice $\Z^n$, we see that the affine toric variety corresponding to any smooth cone $\sigma$ is the product of an affine space with a torus. 
\end{example}
\begin{example}
Let $\A= \{(1,0),(0,1),(1,1)\} \subset M = \Z^2$. We've seen that $I(Y_\A) = \langle xy-z \rangle$ (Exercise \ref{ex:C2inC3}). The associated semigroup is $\S = \N \A = \R^2_{\geq 0} \cap \Z^2 = S_\sigma$, where $\sigma = \R^2_{\geq 0}$ (check that this is self-dual: $\sigma = \sigma^\vee$). Therefore $Y_\A \simeq U_\sigma = \Specm(\C[x,y]) = \C^2$. 
\end{example}
We have seen that if an affine semigroup $\S \subset M$ is generated by a finite set $\A = \{m_1, \ldots, m_s \} \subset \S$, then $Y_\A \subset \C^s$ is an embedding of $\Specm(\C[\S])$. To embed $U_\sigma$, we need to find a set of generators of $\S_\sigma = \sigma^\vee \cap M$. Such a basis can be described in terms of \emph{irreducible elements} of $\S_\sigma$. These are the elements $m \in \S_\sigma$ such that if $m = m' + m''$ for $m', m'' \in \S_\sigma$, then $m' = 0$ or $m'' = 0$. The following is shown in \cite[Prop.~1.2.23]{cox2011toric}.
\begin{proposition} \label{prop:hilbertbasis}
Let $\mathscr{H} = \{ m \in S_\sigma ~|~ m \text{ is irreducible } \}$. Then 
\begin{enumerate}
\item $\mathscr{H}$ is finite and generates $\S_\sigma$, 
\item $\mathscr{H}$ contains the minimal generators of $\sigma^\vee$,
\item $\mathscr{H}$ is the minimal set of generators with respect to inclusion.
\end{enumerate}
\end{proposition}
The set $\mathscr{H}$ from Proposition \ref{prop:hilbertbasis} is called the \emph{Hilbert basis} of $\S_\sigma$, or of $\sigma^\vee$. A consequence of this proposition is that $U_\sigma$ can be embedded via a monomial map in an affine space of dimension $|\mathscr{H}|$. We will see below that $|\mathscr{H}|$ is the smallest such dimension. 

\begin{examplestar} \label{ex:claudia2}
We will now use \texttt{Oscar.jl} to compute the embedding of $U_\sigma$, where $\sigma$ is the cone from Example \ref{ex:coneex}. We generated this cone and its dual in Example \ref{ex:claudia1}. The first step is to compute the Hilbert basis of $\sigma^\vee$: 
\begin{minted}{julia}
H = hilbert_basis(σ_dual)
\end{minted}
This gives $\mathscr{H} = \{ (1,0), (-1,2), (0,1) \}$. From it, we can compute the toric ideal $I(Y_{\mathscr{H}})$ via
\begin{minted}{julia}
I = toric_ideal(H)
\end{minted}
It is generated by $x_1 x_2 - x_3^2$, which is the equation for $U_\sigma \simeq Y_{\mathscr{H}}$ embedded in $\C^3$. 
\end{examplestar}

\subsection{Points on affine toric varieties}

In this section, we characterize points on an affine toric variety $V$ in terms of semigroup homomorphisms. This will help to give an intrinsic, coordinate-free description of the torus action $T \times V \rightarrow V$. By a \emph{semigroup homomorphism} we mean a map between commutative monoids which preserves the binary operation: $\gamma: \S_1 \rightarrow \S_2$ is such that $\gamma(m + cm') = \gamma(m) + c \gamma(m')$. The image of the identity element in $\S_1$ is the identity element in $\S_2$.

\begin{proposition} \label{prop:pointsarehomom}
Let $V = \Specm(\C[\S])$ for an affine semigroup $\S \subset M$. There are one-to-one correspondences between 
\begin{enumerate}
\item points $p \in V$, 
\item maximal ideals $\mathfrak{m}_p \subset \C[\S]$,
\item semigroup homomorphisms $\S \rightarrow \C$. 
\end{enumerate}
\end{proposition}
\begin{proof}
The correspondence between 1 and 2 is standard. To a point $p \in V$, we associate the semigroup homomorphism $m \mapsto \chi^m(p)$. Note that here $\C$ is considered as a multiplicative monoid. We now associate a maximal ideal $\mathfrak{m}_p$ to a semigroup homomorphism $\gamma: \S \rightarrow \C$. Note that $\gamma$ induces a surjective map of $\C$-algebras $\gamma' : \C[\S] \rightarrow \C$. Indeed, the image of $1 \in \C[\S]$ is $1 \in \C$, as $\gamma(0) = 1$. The kernel of $\gamma'$ is a maximal ideal, since $0 \rightarrow \ker \gamma' \rightarrow \C[\S] \rightarrow \C \rightarrow 0$ is exact and hence $\C[\S]/(\ker \gamma') \simeq \C$. We set $\mathfrak{m}_p = \ker \gamma'$. A more concrete way of associating a point $p \in V$ to $\gamma$ is by working with an embedding of $V$. Let $\A = \{ m_1, \ldots, m_s \}$ be such that $\N \A = \S$ and let $p = (\gamma(m_1), \ldots, \gamma(m_s)) \in \C^s$. We show that $p \in Y_\A \simeq V$. If $a, b \in \N^s$ are such that $a - b \in \ker \hat{\Phi}_\A$, then $\sum_{i=1}^s a_i m_i = \sum_{i=1}^s b_i m_i$. Hence 
\[ p^a - p^b = \prod_{i=1}^s \gamma(m_i)^{a_i} - \prod_{i=1}^s \gamma(m_i)^{b_i} = \gamma(\sum_{i=1}^s a_im_i) - \gamma(\sum_{i=1}^s b_im_i) = 0,\]
so that indeed $p \in Y_\A$. Moreover, one checks that the semigroup homomorphism associated to this point in the first step of the proof is $\gamma$, so the correspondence is indeed one-to-one. 
\end{proof}
This way of thinking about points on $V = \Specm(\C[\S])$ gives the following nice description of the torus action $T \times V \rightarrow V$. 
\begin{proposition} \label{prop:torusaction}
Let $V = \Specm(\C[\S])$ be an affine toric variety with dense torus $T \subset V$. An element $t \in T$ acts on $V$ by $(m \mapsto \gamma(m)) \mapsto (m \mapsto \chi^m(t)\gamma(m))$.
\end{proposition}
\begin{proof}
The action $T \times V \rightarrow V$ is algebraic, so it comes from a $\C$-algebra homomorphism $\C[V] = \C[\S] \rightarrow \C[T \times V] = \C[M] \otimes_\C \C[\S]$. Moreover, it is an extension of the action $T \times T \rightarrow T$ of $T$ on itself, which comes from $\C[M] \rightarrow \C[M] \otimes_\C \C[M]$, $\chi^m \mapsto \chi^m \otimes \chi^m$. Commutativity of the following dual diagrams
\begin{center}
\begin{tikzcd}
T \times T \ar[r] \ar[d,hookrightarrow]
    & T \ar[d,hookrightarrow]  \\
    T \times V \ar[r] & V
\end{tikzcd}
\hspace{2cm}
\begin{tikzcd}
\C[M] \otimes_\C \C[M] 
    & \C[M] \ar[l]  \\
    \C[M] \otimes_\C \C[\S] \ar[u,hookrightarrow]  & \C[\S] \ar[l] \ar[u,hookrightarrow]
\end{tikzcd}
\end{center}
implies that $\C[\S] \rightarrow \C[M] \otimes_\C \C[\S]$ is given by $\chi^m \mapsto \chi^m \otimes \chi^m$, which implies the statement. We leave the details as an exercise to the reader.
\end{proof}
The following statement uses \emph{pointed} semigroups, which are those affine semigroups $\S$ for which $\S \cap (-\S) = \{0\}$ (it is the discrete version of strong convexity). 
\begin{proposition} \label{prop:fixedpoint}
Let $V = \Specm(\C[\S]) \simeq Y_\A$ be an affine toric variety with dense torus $T$, for $\S = \N \A \subset M$ an affine semigroup and $\A = \{ m_1, \ldots, m_s \} \subset \S$ a finite generating set. We have that the torus action $T \times V \rightarrow V$ has a fixed point if and only if $\S$ is pointed. In this case, the fixed point is unique and it is given by 
\[ \gamma : m \longmapsto \begin{cases} 1 & m = 0 \\ 0 & \text{otherwise}
\end{cases}.\]
\end{proposition}
\begin{proof}
Note that the map $\gamma$ in the proposition is a semigroup homomorphism if and only if $\S$ is pointed. Moreover, if $\gamma: \S \rightarrow \C$ is a fixed point of the torus action, we must have $\chi^m(t) \gamma(m) = \gamma(m)$ for all $t \in T$, which implies that $\gamma(0) = 1$ and $\gamma(m) = 0$ for $m \neq 0$. 
\end{proof}
Here are two immediate consequences.
\begin{corollary}
Let $Y_\A \subset \C^s$ be an affine toric variety corresponding to $\A = \{m_1, \ldots, m_s \} \subset M \setminus \{0\}$. The torus action $T \times Y_\A \rightarrow Y_\A$ has a fixed point if and only if $0 \in Y_\A$, in which case $0$ is the unique fixed point. 
\end{corollary}
\begin{exercise}
Verify that your affine toric from Exercise \ref{ex:closed} does not contain $0$, and explain.
\end{exercise}
\begin{corollary} \label{cor:smoothptsigma}
Let $U_\sigma$ be the affine toric variety corresponding to a strongly convex rational cone $\sigma \subset N_\R \simeq \R^n$. The torus action on $U_\sigma$ has a fixed point if and only if $\sigma$ has dimension $n$, in which case it is given by the maximal ideal $\langle \chi^m ~|~ m \in \S_\sigma \setminus \{0\} \rangle$.
\end{corollary}
\begin{proof}
The cone $\sigma^\vee$ has dimension $n$ since $\sigma$ is strongly convex. It is strongly convex itself if and only if $\sigma$ has dimension $n$. This is sufficient and necessary for $S_\sigma = \sigma^\vee \cap M$ to be pointed, and hence for there to be a unique fixed point of the torus action (Proposition \ref{prop:fixedpoint}). The maximal ideal $\langle \chi^m ~|~ m \in \S_\sigma \setminus \{0\} \rangle$ is the kernel of the $\C$-algebra homomorphism $\gamma': \C[\S_\sigma] \rightarrow \C$ induced by $\gamma$ from Proposition \ref{prop:fixedpoint}. The statement follows from the proof of Proposition \ref{prop:pointsarehomom}.
\end{proof}

\subsection{Normality and smoothness}

Recall that an integral domain $R$ is \emph{normal} if it is integrally closed in its field of fractions $K$. This means that if $x \in K$ satisfies a monic polynomial relation $x^d + r_{d-1} x^{d-1} + \cdots + r_1 x + r_0$ with coefficients $r_i \in R$, then we must have $x \in R$. An irreducible affine variety $V$ is called \emph{normal} if its coordinate ring $\C[V]$ is normal. Normality of toric varieties will be important in our discussion on divisors later. We will see in this section that normal affine toric varieties are exactly those affine toric varieties corresponding to cones. 

\begin{definition}[Saturated semigroup]
An affine semigroup $\S \subset M$ is \emph{saturated} if $k m \in \S$ for $k \in \N_{>0}, m \in M$ implies that $m \in \S$. 
\end{definition}

\begin{theorem} \label{thm:normal}
Let $V$ be an affine toric variety with dense torus $T \simeq (\C^*)^n$, whose (co)character lattice is $M$ ($N$). The following are equivalent. 
\begin{enumerate}
\item $V$ is normal, 
\item $V = \Specm(\C[\S])$ for a saturated affine semigroup $\S \subset M$, 
\item $V \simeq U_\sigma = \Specm(\C[\S_\sigma])$ for some strongly convex rational cone $\sigma \subset N_\R$.  
\end{enumerate}
\end{theorem}
\begin{proof}
We prove $ 1 \Rightarrow 2 \Rightarrow 3$ and refer to \cite[Thm.~1.3.5]{cox2011toric} for the remaining implication $3 \Rightarrow 1$. For $1 \Rightarrow 2$, suppose that $\C[\S]$ is normal and $km \in \S$. Then $\chi^{km} \in \C[\S]$. Since the character lattice of $V$ is $M$, $\chi^m$ is a regular function on $T \subset V$, and thus a rational function on $V$. That is, $\chi^m$ belongs to the field of fractions of $\C[V] = \C[\S]$ and it satisfies the monic relation $(\chi^m)^k - \chi^{km} = 0$, with coefficients in $\C[\S]$. By normality, $m \in \S$ and $\S$ is indeed saturated. For $2 \Rightarrow 3$, suppose that $\S$ is saturated and that $\S = \N \A$ for a finite set $\A = \{m_1, \ldots, m_s\}$. Let $\sigma^\vee = {\rm Cone}(\A)$. It is clear that $\S \subset \sigma^\vee \cap M$. For the other inclusion, suppose $m \in \sigma^\vee \cap M$. Since 
\begin{equation}
\sigma^\vee \cap (M \otimes_\Z \Q) = \left \{ \sum_{i=1}^s \lambda_i m_i ~|~ \lambda_i \in \Q_{\geq 0} \right \},
\end{equation}
we have that $m = \sum_{i=1}^s \lambda_i m_i$ for some nonnegative rational coefficients $\lambda_i$. Clearing denominators, we can find $k$ such that $km \in \N \A = \S$. Since $\S$ is saturated, we conclude $m \in \S$ and thus $\S = \sigma^\vee \cap M$. Note that $\sigma$ is strongly convex since $V$ has dimension $n$.
\end{proof}
\begin{example}
The affine toric variety $Y_\A$ with $\A = \{2,3\} \subset \Z$ is a standard example of a non-normal variety. Its semigroup $\N \A = \{0,2,3,4,\ldots\}$ is strictly contained in ${\rm Conv}(\A) \cap \Z = \N$. The affine toric variety $Y_\A$ with $\A = \{2,4\} \subset \Z$ from Example \ref{ex:parabola} is normal, although $\N \A \subsetneq \Z$. The reason is that $\A$ is saturated in the lattice it generates, i.e., in the character lattice $2 \Z \subset \Z$ of the dense torus of $Y_\A$. 
\end{example}

Next, we discuss which affine toric varieties are smooth. Smoothness implies normality, so by Theorem \ref{thm:normal}, the only candidates are the affine toric varieties coming from cones. 

\begin{theorem} \label{thm:smooth}
An affine toric variety $V$ is smooth if and only if $V = \Specm(\C[\S_\sigma])$ for some smooth rational cone $\sigma$. 
\end{theorem}
\begin{proof}
If $\sigma$ is smooth, $\Specm(\C[\S_\sigma])$ is smooth by Example \ref{ex:smooth}. Conversely, suppose $V = U_\sigma$ is smooth and $\dim \sigma = n$. Then in particular, the fixed point of the torus action corresponding to the maximal ideal $\mathfrak{m} = \langle \chi^m ~|~ m \in \S_\sigma \setminus \{0\} \rangle$ (see Corollary \ref{cor:smoothptsigma}) is a smooth point $p$ of $V$. Therefore, the tangent space at $p$ has dimension $n$: $\dim_\C \mathfrak{m}/\mathfrak{m}^2 = n$. Observe that
\begin{align*}
\mathfrak{m} = \bigoplus_{m \in \S_\sigma \setminus \{0\}} \C \cdot \chi^m &= \bigoplus_{m \text{ irreducible}} \C \cdot \chi^m \oplus \bigoplus_{m \text{ reducible}} \C \cdot \chi^m \\
&= \bigoplus_{m \in \mathscr{H}} \C \cdot \chi^m \oplus \mathfrak{m}^2,
\end{align*}
where $\mathscr{H}$ is the Hilbert basis of $\S_\sigma$ (Proposition \ref{prop:hilbertbasis}). Hence $|\mathscr{H}| = n$, which implies that $\sigma^\vee \cap M$ is generated by $n$ elements. The Hilbert basis contains the minimal generators, and $\sigma$ has dimension $n$ by assumption, so that $\sigma$ has $n$ rays. Moreover, $\Z \S_\sigma = \Z \mathscr{H} = M$, which means that $\mathscr{H}$ is a $\Z$-basis for $M$. We conclude that $\sigma^\vee$ is smooth, and smoothness is preserved by duality (exercise). For the case $\dim \sigma < n$ we refer to \cite[Thm.~1.3.12]{cox2011toric}.
\end{proof}

The proof of Theorem \ref{thm:smooth} implies that in general, when $\sigma$ is a strongly convex rational cone of dimension $n$, the tangent space at the torus fixed point of $U_\sigma$ has dimension $|\mathscr{H}|$. For any affine embedding of $U_\sigma$, the tangent space is at most the dimension of the ambient affine space. Therefore $|\mathscr{H}|$ is the smallest $s$ for which $U_\sigma$ can be embedded in $\C^{s}$.

\begin{examplestar} \label{ex:chiara3} Consider the rational cone 
\[\sigma = {\rm Cone}((1,2,3), (2,1,3), (1,3,2), (3,1,2), (2,3,1), (3,2,1)) \subset \R^3.\] 
This is the cone over the two-dimensional \emph{permutohedron}, see Example* \ref{ex:chiara7}. Since $\sigma \subset \R^3_{\geq 0}$, $\sigma$ is strongly convex and hence $\dim U_\sigma = 3$ (Proposition \ref{prop:7}). We check this in Julia by constructing an affine normal toric variety from $\sigma$: 
\begin{minted}{julia}
U_σ = AffineNormalToricVariety(σ); dim(U_σ)
\end{minted}
For any embedding $U_\sigma \simeq Y_\A \subset \C^s$, $s$ is at least $15$, since this is the number of elements in the Hilbert basis of $\sigma^\vee$, computed as in Example \ref{ex:claudia1}. The function \texttt{toric\_ideal} can be used on the object returned by \texttt{AffineNormalToricVariety}: 
\begin{minted}{julia}
I = toric_ideal(U_σ)
\end{minted}
This computes $77$ binomial generators for $I(Y_\A)$ within $0.0028$ seconds. 
\end{examplestar}

\subsection{Toric morphisms} \label{subsec:toricmorphismsaff}
Let $V_i = \Specm \C[\S_i]$, $i = 1, 2$ be affine toric varieties. In this section we study morphisms $\phi : V_1 \rightarrow V_2$ which preserve the toric structure. 
\begin{definition}[Toric morphism]
A morphism $\phi: V_1 \rightarrow V_2$ is \emph{toric} if the pullback $\phi^*: \C[\S_2] \rightarrow \C[\S_1]$ is induced by a semigroup homomorphism $\widehat{\phi}: \S_2 \rightarrow \S_1$, i.e.~$\phi^*(\chi^m) = \chi^{\widehat{\phi}(m)}$. 
\end{definition}
\begin{example}
The map $\Phi_\A : T \rightarrow (\C^*)^s$ of tori from previous sections is a toric morphism. The pullback map is induced by our map $\widehat{\Phi}_\A: \Z^s \rightarrow M$. 
\end{example}
The following result tells us that to check whether a morphism $\phi$ is toric, it suffices to know its restriction to the torus of $V_1$. Let $T_i$ be the dense torus of the affine toric variety $V_i$ for $i = 1,2$. The (co)character lattice of $T_i$ is denoted by $M_i$ ($N_i$).
\begin{proposition}
 A morphism $\phi: V_1 \rightarrow V_2$ is toric if and only if $\phi(T_1) \subset T_2$ and the restriction $\phi_{|T_1}: T_1 \rightarrow T_2$ is a group homomorphism. 
\end{proposition}
\begin{proof}
By point 2 in Proposition \ref{prop:semigroup}, we have $M_i = \Z \S_i$. If $\phi$ is toric, $\widehat{\phi}:\S_2 \rightarrow \S_1$ extends to a group homomorphism $\widehat{\phi}:M_2 \rightarrow M_1$. We obtain the diagrams
\begin{center}
\begin{tikzcd}
\S_2 \ar[r] \ar[d,hookrightarrow] & \S_1 \ar[d,hookrightarrow] \\
M_2 \ar[r] & M_1
\end{tikzcd}
$\longrightarrow$
\begin{tikzcd}
\C[\S_2] \ar[r] \ar[d,hookrightarrow] & \C[\S_1] \ar[d,hookrightarrow] \\
\C[M_2] \ar[r] & \C[M_1]
\end{tikzcd}
$\overset{\Specm}{\longrightarrow}$
\begin{tikzcd}
V_2 & V_1 \ar[l] \\
T_2 \ar[u,hookrightarrow] & T_1 \ar[u,hookrightarrow] \ar[l]
\end{tikzcd}
\end{center}
which shows that $\phi(T_1) \subset T_2$. The map $V_1 \rightarrow V_2$ is $\phi$ and $T_1 \rightarrow T_2$ is $\phi_{|T_1}$. This is indeed a group homomorphism, obtained by taking $\Hom_\Z(-,\C^*)$ of $\widehat{\phi}: M_2 \rightarrow M_1$. The reader is encouraged to verify this last statement. Conversely, if $\phi(T_1) \subset T_2$ and $\phi_{|T_1}$ is a group homomorphism, the diagram on the right of the above display induces the diagram in the middle, in which $\C[M_2] \rightarrow \C[M_1]$ comes from a group homomorphism $M_2 \rightarrow M_1$ which restricts to a semigroup homomorphism $\S_2 \rightarrow \S_1$, hence $\phi$ is toric. 
\end{proof}

\begin{proposition} \label{prop:equivariantaff}
A toric morphism $\phi: V_1 \rightarrow V_2$ is \emph{equivariant}, meaning that $\phi(t \cdot p) = \phi(t) \cdot \phi(p)$ for $t \in T_1$, $p \in V_1$. 
\end{proposition}
\begin{proof}
The statement comes down to the fact that the diagram 
\begin{center}
\begin{tikzcd}
T_1 \times V_1 \ar[r] \ar[d,"\phi_{|T_1} \times \phi"] & V_1 \ar[d,"\phi"] \\
T_2 \times V_2 \ar[r] & V_2
\end{tikzcd}
, which restricts to \quad
\begin{tikzcd}
T_1 \times T_1 \ar[r] \ar[d,"\phi_{|T_1} \times \phi_{|T_1}"] & T_1 \ar[d,"\phi_{|T_1}"] \\
T_2 \times T_2 \ar[r] & T_2
\end{tikzcd},
\end{center}
commutes. The commutation of the right diagram follows from the fact that $\phi_{|T_1}$ is a group homomorphism. Then the left diagram gives two maps $T_1 \times V_1 \rightarrow V_2$. Since they agree on a nonempty Zariski open subset, they agree on $T_1 \times V_1$, so also the left diagram commutes. 
\end{proof}

We now show how to construct toric morphisms $\phi: V_1 \rightarrow V_2$ between normal affine toric varieties from maps of lattices. Let $\overline{\phi}:N_1 \rightarrow N_2$ be a group homomorphism, with $N_1 \simeq \Z^{n_1}, N_2 \simeq \Z^{n_2}$. The $N_i$ are the cocharacter lattices of the tori $T_i = N_i \otimes_\Z \C^* \simeq (\C^*)^{n_i}$. The map $\overline{\phi}$ naturally induces a group homomorphism $\phi_{|T_1}: T_1 \rightarrow T_2$, which takes $u \otimes t  \mapsto \overline{\phi}(u) \otimes t$.

\begin{example} \label{ex:F}
Here's how this works in coordinates. Let $N_i = \Z^{n_i}, T_i = (\C^*)^{n_i}$. Then $\overline{\phi}: \Z^{n_1} \rightarrow \Z^{n_2}$ is given by a matrix $F \in \Z^{n_2 \times n_1}$. We write $F_{i,:} \in \Z^{n_1}$ for the $i$-th row of $F$. The map $\phi_{|(\C^*)^{n_1}}: (\C^*)^{n_1} \rightarrow (\C^*)^{n_2}$ is given by $\phi_{|(\C^*)^{n_1}}(t) = (t^{F_{1,:}}, \ldots, t^{F_{n_2,:}})$. 
\end{example}

Let $\sigma_i \subset (N_i)_\R, i = 1, 2$ be strongly convex rational cones. Our next goal is to characterize when $\overline{\phi}$ induces a map of tori $\phi_{|T_1}$ which extends to a toric morphism $\phi: U_{\sigma_1} \rightarrow U_{\sigma_2}$. Note that this makes sense since $T_i$ is the dense torus of $U_{\sigma_i}$ by Proposition \ref{prop:semigroup}. We denote $\overline{\phi}_\R = \overline{\phi} \otimes_\Z \R : (N_1)_\R \rightarrow (N_2)_\R$. In the notation of Example \ref{ex:F}, this is the linear map $\R^{n_1} \rightarrow \R^{n_2}$ represented by $F$. 

\begin{proposition} \label{prop:compcone}
Let $\overline{\phi}:N_1 \rightarrow N_2$ be a group homomorphism. The induced map of tori $\phi_{|T_1}: T_1 \rightarrow T_2$ extends to a toric morphism $\phi: U_{\sigma_1} \rightarrow U_{\sigma_2}$ if and only if $\overline{\phi}_\R(\sigma_1) \subset \sigma_2$. 
\end{proposition}
\begin{proof}
We identify $N_i = \Z^{n_i}$ and $T_i = (\C^*)^{n_i}$ as in Example \ref{ex:F}, so that $\overline{\phi}$ is represented by $F \in \Z^{n_2 \times n_1}$. The map of tori $\phi_{|T_1}: T_1 \rightarrow T_2$ gives a map of character lattices $\widehat{\phi}: M_2 = \Z^{n_2} \rightarrow M_1 = \Z^{n_1}$ which is given by the transpose $F^\top$. For this to induce a map of $\C$-algebras $\phi^* : \C[\S_{\sigma_2}] \rightarrow \C[\S_{\sigma_1}]$ we need that $\widehat{\phi}(m_2) \in \S_{\sigma_1} = \sigma_1^\vee \cap \Z^{n_1}$ for all $m_2 \in \S_{\sigma_2} = \sigma_2^\vee \cap \Z^{n_2}$, or 
\begin{align*}
 \langle F^\top m_2, u_1 \rangle \geq 0, \quad &\text{ for all $u_1 \in \sigma_1, m_2 \in \sigma_2^\vee$} \\
\Longleftrightarrow ~  \langle m_2, Fu_1 \rangle \geq 0, \quad &\text{ for all $u_1 \in \sigma_1, m_2 \in \sigma_2^\vee$} \\
\Longleftrightarrow Fu_1 \in (\sigma_2^\vee)^\vee = \sigma_2, \quad &\text{ for all $u_1 \in \sigma_1$}. \qedhere
\end{align*}
\end{proof}
If $\overline{\phi}_\R: (N_1)_\R \rightarrow (N_2)_\R$ satisfies $\overline{\phi}(\sigma_1) \subset \sigma_2$, we say that $\overline{
\phi}$ is \emph{compatible} with $\sigma_1$ and $\sigma_2$. 
\begin{example}[Affine open subsets from faces] \label{ex:affineopensubsets}
Let $\sigma \subset N_\R$ be a strongly convex rational cone and let $\tau \preceq \sigma$ be one of its faces. The identity map $\overline{\phi}: N \rightarrow N$ is compatible with $\tau$ and $\sigma$. It induces the inclusion $\phi: U_\tau \hookrightarrow U_\sigma$. The pullback map is the inclusion of $\C$-algebras $\phi^*: \C[\S_\sigma] \hookrightarrow \C[\S_\tau]$ induced by $\sigma^\vee \cap M \subset \tau^\vee \cap M$. It is shown in \cite[Prop.~1.3.16]{cox2011toric} that $\C[\S_\tau] = \C[\S_\sigma]_{\chi^m} = \C[\S_\sigma][\chi^{-m}]$ is the localization of $\C[\S_\sigma]$ at $\chi^m$. That is, $U_\tau$ is the affine open subset $\{ p \in U_\sigma ~|~ \chi^m(p) \neq 0 \}$ of $U_\sigma$. This is illustrated in Figure \ref{fig:localize}.
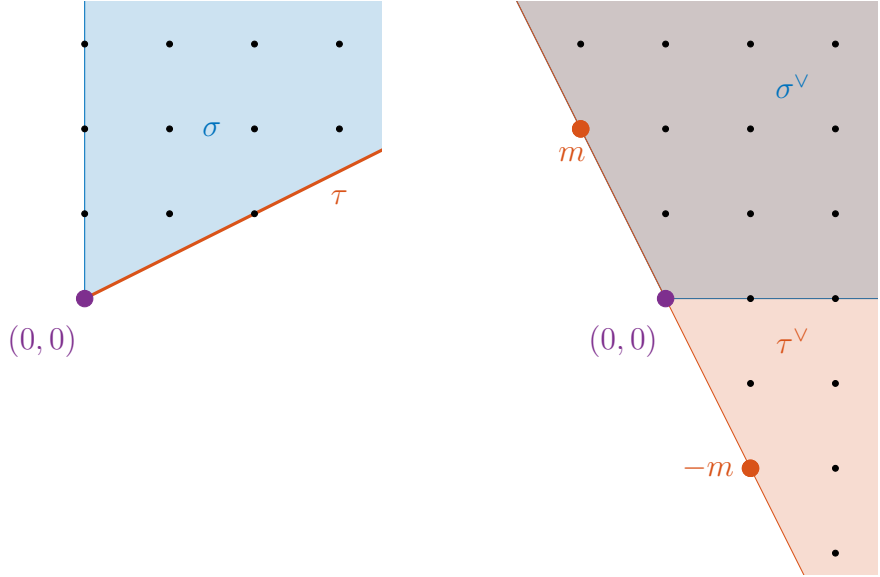
\begin{figure}
\centering
\begin{tikzpicture}[scale=1]
\begin{axis}[%
width=2in,
height=3in,
scale only axis,
xmin=-1.0,
xmax=3.5,
ymin=-3.25,
ymax=3.5,
ticks = none, 
ticks = none,
axis background/.style={fill=white},
axis line style={draw=none} 
]


\addplot [color=mycolor1,solid,fill opacity=0.2,fill = mycolor1,forget plot]
  table[row sep=crcr]{%
 0 5\\
10	 5\\	
0 0\\
0 5 \\
};

\addplot [very thick, color=mycolor2,solid,fill opacity=0.2,fill = mycolor1,forget plot]
  table[row sep=crcr]{%
0 0 \\
 10 5 \\
};

\addplot[only marks,mark=*,mark size=3.1pt,mycolor4
        ]  coordinates {
  (0,0)
};

\addplot[only marks,mark=*,mark size=1.1pt,black
        ]  coordinates {
  (0,1) (1,1) (2,1)
  (0,2) (1,2) (2,2) (3,2) 
  (0,3) (1,3) (2,3) (3,3) 
};

\node (P) at (axis cs:1.5,2) {$\textcolor{mycolor1}{\sigma}$};
\node (P) at (axis cs:3,1.2) {$\textcolor{mycolor2}{\tau}$};
\node (P) at (axis cs:-0.5,-0.5) {$\textcolor{mycolor4}{(0,0)}$};

\end{axis}
\end{tikzpicture} 
\qquad \quad
\begin{tikzpicture}[scale=1]
\begin{axis}[%
width=2in,
height=3in,
scale only axis,
xmin=-2.0,
xmax=2.5,
ymin=-3.25,
ymax=3.5,
ticks = none, 
ticks = none,
axis background/.style={fill=white},
axis line style={draw=none} 
]


\addplot [color=mycolor1,solid,fill opacity=0.2,fill = mycolor1,forget plot]
  table[row sep=crcr]{%
 5 0\\
5	 5\\	
-2.5 5\\
0 0 \\
5 0 \\
};

\addplot [color=mycolor2,solid,fill opacity=0.2,fill = mycolor2,forget plot]
  table[row sep=crcr]{%
-2.5 5\\
3 -6\\
3 6 \\
-2.5 5\\
};

\addplot[only marks,mark=*,mark size=3.1pt,mycolor2
        ]  coordinates {
  (-1,2) (1,-2)
};

\addplot[only marks,mark=*,mark size=3.1pt,mycolor4
        ]  coordinates {
  (0,0)
};


\addplot[only marks,mark=*,mark size=1.1pt,black
        ]  coordinates {
    (-1,3) (0,3) (1,3) (2,3)
    (0,2) (1,2) (2,2)
    (0,1) (1,1) (2,1)
    (1,0) (2,0)
    (1,-1) (2,-1)
    (1,-2) (2,-2)
    (2,-3)
};

\addplot[only marks,mark=*,mark size=3.1pt,mycolor2
        ]  coordinates {
  (-1,2) (1,-2)
};

\node (P) at (axis cs:1.5,2.5) {$\textcolor{mycolor1}{\sigma^\vee}$};
\node (P) at (axis cs:1.5,-0.5) {$\textcolor{mycolor2}{\tau^\vee}$};
\node (P) at (axis cs:-1.1,1.7) {$\textcolor{mycolor2}{m}$};
\node (P) at (axis cs:0.5,-2) {$\textcolor{mycolor2}{-m}$};
\node (P) at (axis cs:-0.5,-0.5) {$\textcolor{mycolor4}{(0,0)}$};

\end{axis}
\end{tikzpicture} 
\caption{The semigroup algebra $\C[\S_\tau] = \C[\tau^\vee \cap \Z^2]$ is the localization $\C[\sigma]_{\chi^m}$.}
\label{fig:localize}
\end{figure}
\end{example}

\begin{example}
Let $\sigma$ be the cone $\R^2_{\geq 0} \subset \R^2$ with affine toric variety $U_\sigma = \C^2$. The dual cone of the face $\tau = \R_{\geq} \cdot e_1$ is $\tau^\vee = \R_{\geq 0} \times \R \subset (\R^2)^\vee$, which contains $\sigma^\vee = \R^2_{\geq 0} \subset (\R^2)^\vee$. The coordinate ring of $U_\tau$ is given by $\C[U_\tau] = \C[\S_\tau] = \C[x_1, x_2, x_2^{-1}] = \C[U_\sigma]_{x_2}$. Here the monomial $x_2$ corresponds to the character $\chi^m$ with $m = (0,1)$, for which $H_m \cap \sigma = \tau$. 
\end{example}
\begin{example}[Sublattices of finite index] \label{ex:sublatticesfiniteindex}
Let $N_1 = \Z^2$ and $N_2 = \{(a/2,b/2) \in \Z^2 ~|~ a + b = 0 \mod 2 \}$ . We consider the cone $\sigma = \R^2_{\geq 0} \subset \R^2 = (N_1)_\R = (N_2)_\R$. The identity map $\overline{\phi}:\R^2 \rightarrow \R^2$ is trivially compatible with $\sigma$ and $\sigma$. The dual lattices are $M_1 = \{(a,b) \in \Z^2 ~|~ a - b = 0 \mod 2 \}$ and $M_2 = \Z^2$. Considering $\sigma$ as a rational cone in $(N_2)_\R$, we obtain the affine toric variety $U_{\sigma,N_1} = \C^2$. On the other hand, $U_{\sigma,N_2}$ is not smooth, as the Hilbert basis of $\sigma^\vee \subset (M_2)_\R$ has three elements $(2,0),(1,1)$ and $(0,2)$. The map $\overline{\phi}$ induces a toric morphism $\phi: U_{\sigma,N_1} \rightarrow U_{\sigma,N_2}$ given by 
\[ \phi^*: \C[x,y,z]/\langle x z - y^2 \rangle \hookrightarrow \C[s,t], \quad \text{where } x \mapsto s^2, y \mapsto st, z \mapsto t^2.\]
The map $\phi$ is generically 2-to-1, and the multiplicative group $\{-1,1\}$ acts on its fibers. The \emph{ring of invariants} of this action is $\im \phi^*$. See \cite[Ex.~1.3.19]{cox2011toric} for a nice discussion. 
\end{example}

\begin{exercise}
Consider the map $\overline{\phi} : \Z^4 \rightarrow \Z^2$ given by the matrix 
\[ F = \begin{bmatrix}
1 &0 & -1  & 0 \\ 0 & 1 & 2 &-1
\end{bmatrix}.\]
This is compatible with the cones $\sigma_1 = {\rm Cone}(e_3,e_4) \subset \R^4$ and $\sigma_2 = {\rm Cone}( (-1,2),(0,-1) ) \subset \R^2$. Show that the corresponding toric morphism from $U_{\sigma_1} \simeq \C^2 \times (\C^*)^2 = \Specm(\C[X_1^{\pm 1}, X_2^{\pm 1}, X_3, X_4])$ to $U_{\sigma_2} \simeq \C^2 = \Specm(\C[t_1^{-1}, t_1^{-2}t_2^{-1}]) = \Specm(\C[x,y])$ is 
\[ \C[x,y] \longrightarrow \C[X_1^{\pm 1}, X_2^{\pm 1}, X_3, X_4]   , \quad \text{where} \quad x \mapsto \frac{X_3}{X_1} \quad \text{and} \quad y \mapsto \frac{X_4}{X_1^2X_2}. \qedhere \]
\end{exercise}

\section{Projective toric varieties} \label{sec:projectiveTV}

We proceed by studying projective toric varieties. These often arise as the closure of the image of a monomial map to projective space, in which case they are covered by affine toric varieties. The polyhedral cones in this setting come from a convex lattice polytope, which combinatorially encodes how the affine pieces are glued together. We start with identifying $\PP^n$ as a projective toric variety. Next, we define projective toric varieties embedded via monomial maps and derive their dimension and degree. Next, we discuss affine covers, connections with polytopes, normality and smoothness.

\subsection{Definition of a projective toric variety}
\begin{definition}[Projective toric variety]
A \emph{projective toric variety} is an irreducible projective variety $X$ containing a torus $T \simeq (\C^*)^n$ as a Zariski open subset, such that the action of $T$ on itself extends to an algebraic action $T \times X \rightarrow X$ of $T$ on $X$.
\end{definition}
\begin{example}
$\PP^n$ is a projective toric variety with torus 
\begin{align*}
T_{\PP^n} &= \{ (x_0: \cdots: x_n) \in \PP^n ~|~ x_i \neq 0, \text{ for all } i \} \\
&= \{ (1:t_1:\cdots:t_n) \in \PP^n ~|~ t \in (\C^*)^n \} \simeq (\C^*)^n,
\end{align*}
where $t \in (\C^*)^n$ acts by $t \cdot (x_0:x_1:\cdots:x_n) = (x_0:t_1x_1:\cdots:t_nx_n)$. Recall that $T_{\PP^n} \simeq (\C^*)^{n+1}/\C^*$, where the quotient is realized by the surjective map 
\[ \pi: (\C^*)^{n+1} \rightarrow T_{\PP^n}, \quad (t_0, \ldots, t_n) \mapsto (t_0:\cdots:t_n).\]
This is a group homomorphism, and it induces an inclusion of character lattices ${\cal M}_n \hookrightarrow \Z^{n+1}$, where ${\cal M}_n$ is the character lattice of $T_{\PP^n}$. This is made explicit by reversing the arrows in the short exact sequence 
\[ 1 \longrightarrow \C^* \longrightarrow (\C^*)^{n+1} \overset{\pi}{\longrightarrow} T_{\PP^n} \longrightarrow 1 \]
which shows that ${\cal M}_n = \ker( (a_0, \ldots, a_n) \mapsto a_0 + \cdots + a_n ) = \{ a \in \Z^{n+1} ~|~ \sum_{i=0}^n a_i = 0 \}$.
\end{example}

\subsection{Projective toric varieties from monomial maps} \label{sec:projmonom}

Recall that a finite subset $\A = \{m_1, \ldots, m_s \} \subset M$ gives a map $\Phi_\A : T \rightarrow (\C^*)^s$ given by $t \mapsto( \chi^{m_1}(t), \ldots, \chi^{m_s}(t))$ 
. In Section \ref{subsec:affmonomial} we used this map to construct affine toric varieties. In this section, we compose this map with the quotient map $\pi: (\C^*)^s \rightarrow T_{\PP^{s-1}}$ introduced above in order to obtain a projective toric variety. We set 
\[ X_\A = \overline{(\pi \circ \Phi_\A)(T)} \subset \PP^{s-1}. \]
\begin{example}
When $T = (\C^*)^n$ and $M = \Z^n$, $X_\A$ is the Zariski closure of the image of the monomial map $t \mapsto (t^{m_1}: \cdots: t^{m_s})$.
\end{example}
The proof of the following proposition is similar to that of Proposition \ref{prop:charLattice}.
\begin{proposition}
With the notation introduced above, $X_\A \subset \PP^{s-1}$ is a projective toric variety. Its dense torus is $(\pi \circ \Phi_\A)(T) = X_\A \cap T_{\PP^{s-1}}$.
\end{proposition}
In the case where the image of $\Phi_\A$ is invariant under coordinate-wise scaling $(x_1,\ldots,x_s) \mapsto (\lambda x_1, \ldots, \lambda x_s)$, the affine toric variety $Y_\A$ is clearly the affine cone over the projective toric variety $X_\A$. This happens in the following example. 
\begin{example}
Let $M = \Z^2$ and $\A = \{(1,0),(1,1),(1,2)\}$. This gives the affine toric variety $Y_\A = \{ xz - y^2 = 0 \}$, which is the affine cone over the projective toric curve $X_\A$, given by the same equation in the homogeneous coordinates $(x:y:z)$ on $\PP^2$.
\end{example}
The following proposition characterizes the subsets $\A \subset M$ for which $Y_\A$ is the affine cone over $X_\A$.
\begin{proposition} \label{prop:affinehyper}
The toric ideal $I(Y_\A)$ is homogeneous if and only if there exists $u \in N$ and $c \in \Z \setminus \{0\}$ such that $\langle u, m_i \rangle = c$ for all $i$. 
\end{proposition}
If $M = \Z^n$, we can collect the $m_i$ in the columns of a matrix $A \in \Z^{n \times s}$. The condition in Proposition \ref{prop:affinehyper} is equivalent to $(1,\ldots, 1)$ being in the row space of $A$, viewed as a matrix over $\Q$, and to $m_1, \ldots, m_s$ lying on an affine hyperplane in $M_\R$.
\begin{proof}
Without loss of generality, we assume $M = \Z^n$. Note that $I(Y_\A)$ cannot contain a monomial, as we have $(1,\ldots,1) \in  Y_\A \subset \C^s$ (Exercise \ref{ex:contains1}). 
Therefore, if $I(Y_\A)$ is homogeneous, each binomial generator in Theorem \ref{thm:toricideal} must be homogeneous:
\begin{equation} \label{eq:ellsumeq}
\sum_{i = 1}^s (\ell_+)_i - \sum_{i = 1}^s (\ell_-)_i = \sum_{i=1}^s \ell_i = 0, \quad \text{for all $\ell \in \ker \widehat{\Phi}_\A$}.
\end{equation}
This is equivalent to $(1,\ldots,1) \in \Q^s$ being in the row space of $A$, so that $\langle \hat{u}, m_i \rangle = 1$ for some $\hat{u} \in \Q^n$ and for all $m_i \in \A$. Clearing denominators gives $\langle u, m_i \rangle = c$. Conversely, if $(1,\ldots,1)$ is in the row space of $A$, then \eqref{eq:ellsumeq} holds. By Theorem \ref{thm:toricideal}, $I(Y_\A)$ is generated by homogeneous polynomials. 
\end{proof}
We use Proposition \ref{prop:affinehyper} to show that for any finite subset $\A = \{m_1, \ldots, m_s\} \subset M = \Z^n$, there is $\widehat{\A} \subset \Z^n \times \Z$ such that $X_\A = X_{\widehat{\A}}$ and $Y_{\widehat{\A}}$ is the affine cone over $X_{\widehat{\A}}$. This set $\widehat{\A} \subset \Z^n \oplus \Z$ can be obtained by appending a``1"-entry to each element of $\A$. This way 
\[ (\pi \circ \Phi_\A): t \mapsto (t^{m_1}:\cdots:t^{m_s}) \quad \text{and} \quad (\pi \circ \Phi_{\widehat{\A}}): (t, t_{n+1}) \mapsto (t_{n+1} t^{m_1}:\cdots:t_{n+1}t^{m_s}) \]
have the same image in $T_{\PP^{s-1}}$. Moreover, setting $u = e_{n+1}$, Proposition \ref{prop:affinehyper} guarantees that $I(Y_{\widehat{\A}})$ is homogeneous, so we arrive at the following result.
\begin{proposition} \label{prop:Ahat}
With the notation above, $X_\A = X_{\widehat{\A}}$ and $Y_{\widehat{\A}}$ is the affine cone over $X_{\widehat{\A}}$.
\end{proposition}

We are now ready to discuss the character lattice of the dense torus of $X_\A$. 
\begin{definition}[Affine lattice]
The \emph{lattice affinely generated by $\A = \{m_1, \ldots, m_s\} \subset M$} is 
\[ \Z' \A = \left \{ \sum_{i=1}^s a_i m_i ~|~ a \in \Z^s, \sum_{i=1}^s a_i = 0 \right \} \subset M.\]
\end{definition}
\begin{exercise} \label{ex:afflattice} Writing $\A - m_i = \{ m_j - m_i, \ldots, m_s - m_i \}$, show that $\Z' \A = \Z (\A - m_i)$, for $i = 1, \ldots, s$.
\end{exercise}
\begin{proposition} \label{prop:charlatticeproj}
The character lattice of the dense torus $T_{X_\A} \subset X_\A$ is $\Z' \A$. 
\end{proposition}
\begin{proof}
Writing $M'$ for the character lattice of $T_{X_\A}$, the statement follows as in the proof of Proposition \ref{prop:charLattice} from
\begin{center}
\begin{tikzcd}
T \ar[r] \ar[dr,->>]
    & (\PP^*)^{s-1}  \\
    & T_{X_\A} \ar[u,hookrightarrow]
\end{tikzcd}
\hspace{2cm}
\begin{tikzcd}
M 
    & {\cal M}_{s-1} \ar[l] \ar[d,->>] \\
    & M' \ar[ul,hookrightarrow] 
\end{tikzcd}. 
\end{center}
\end{proof}
\begin{corollary}
The dimension of $X_\A$ is the rank of $\Z' \A$, which is ${\rm rank} \, \Z \A - 1$ if the condition of Proposition \ref{prop:affinehyper} is satisfied, and ${\rm rank} \, \Z \A$ otherwise.
\end{corollary}
\begin{proof}
For details, see \cite[Prop.~2.1.6]{cox2011toric}.
\end{proof}
\begin{exercise} \label{ex:twotoone}
Let $M = \Z^2$ and $\A = \{(1,0),(0,1),(1,2),(2,1)\}$. Show that $\Z' \A$ has the same rank as $\Z \A$, but it has index 2 as a sublattice. Geometrically, this translates into $(\pi \circ \Phi_\A)$ being two-to-one (see Example \ref{ex:sublatticesfiniteindex}).
\end{exercise}

\subsection{Convex lattice polytopes} \label{subsec:polytopes}
We will see that the projective toric varieties from the previous section are in many ways related to \emph{convex polytopes}. In this section, we introduce some definitions and basic facts related to these objects. Much of this is taken from \cite[App.~D]{telen2020thesis}.
\begin{definition}[Polytope]
A \emph{polytope} $P$ in $M_\R$ is the convex hull of a finite set of points $\A = \{m_1,\ldots,m_k\}$ in $M_\R$:
$$P =  {\rm Conv}(\A) = \left \{\sum_{i =1}^{k} c_i m_i \in M_\R ~|~  c_i \in \R ,\sum_{i=1}^k c_i = 1, c_i \geq 0 \right \} \subset M_\R.$$
If $\A \subset M$, $P$ is called a \emph{lattice polytope}.
\end{definition}
Note that we define a polytope to be convex. The reason is that we will not encounter non-convex polytopes in this text. 
The dimension $\dim P$ of a polytope $P \subset M_\R$ is defined as the dimension of the smallest affine subspace of $M_\R$ containing $P$. A polytope in $M_\R$ is said to be \emph{full-dimensional} if $\dim P = n$. 
A point $u \in N_\R \backslash \{0\}$ and a scalar $a \in \R$ give \[ H_{u,a} = \{ m \in M_\R~|~ \langle u,m \rangle + a = 0 \}\quad \text{and} \quad H_{u,a}^+ = \{ m \in M_\R ~|~ \langle u,m \rangle +a \geq 0 \}.\]
\begin{definition}[Faces of a polytope]
Take $ u \in N_\R \backslash \{0\}$, $a \in \R$ and let $P \subset M_\R$ be a convex polytope. The set $H_{u,a} \cap P$ is a \emph{face} of $P$ if $P \subset  H_{u,a}^+$ and $a = - \min_{m \in P} \langle u,m \rangle$. We say that $P$ is a face of $P$ by convention. We write $Q \preceq P$ if $Q$ is a face of $P$.
\end{definition} 
A hyperplane $H_{u,a}$ for which $H_{u,a} \cap P$ is a face of $P$ is called a \emph{supporting hyperplane}. A face $Q$ of a polytope is again a polytope. The \emph{codimension} of a face $Q \subset P$ is $\dim P - \dim Q$. A face of codimension 1 in $P$ is called a \emph{facet}, a face of dimension 1 is an \emph{edge} and a face of dimension 0 is a \emph{vertex}. 
One way of representing a polytope $P$ in terms of faces uses vertices: \[P = {\rm Conv}( v \in M_\R ~|~ v \text{ is a vertex of } P ).\] 
This is called a \emph{vertex representation} or \emph{V-representation}.
Alternatively, any polytope can be expressed as the intersection of finitely many closed half-spaces $H_{u,a}^+$ associated to supporting hyperplanes. That is, any polytope $P \subset M_\R$ can be written as 
\begin{equation} \label{eq:Hrep1}
P = H_{u_1,a_1}^+ \cap \cdots \cap H_{u_k,a_k}^+ = \{ m \in M_\R ~|~ \langle u_i,m \rangle + a_i \geq 0, i = 1, \ldots, k \}
\end{equation}
for some $u_1, \ldots, u_k \in N_\R$, $a_1, \ldots, a_k \in \R$. 
The representation in \eqref{eq:Hrep1} 
is called a \emph{half-space representation} or \emph{H-representation} of the polytope $P$. There exist infinitely many different H-representations for any polytope. However, if $P$ is full-dimensional, there exists an \emph{essentially unique}, \emph{minimal} H-representation of $P$, in the sense that it consists of a minimal number $k$ of inequalities where the inequalities are uniquely defined up to multiplication with a nonzero scalar. Suppose that $P$ is full-dimensional. For a supporting hyperplane $H_{u,a}$ corresponding to a facet $Q$ of $P$, the vector $u$ is uniquely determined up to a nonzero scalar factor. For every facet $Q$, let $u_Q, a_Q$ be such that $P \subset H_{u_Q,a_Q}^+,  H_{u_Q,a_Q} \cap P = Q$. The minimal H-representation of $P$ is given by 
$$ P = \bigcap_{Q \text{ facet of } P} H_{u_Q,a_Q}^+.$$
If $P$ is a full-dimensional lattice polytope, then for any facet $Q \subset P$, $u_Q$ can be chosen in a unique way as the generator of the sublattice 
$$ \{ u \in N ~|~ \langle u,m \rangle = 0 \text{ for all } m \in Q \},$$ 
for which $P \in H_{u_Q,a_Q}^+$. This is called the \emph{primitive, inward pointing facet normal} of $Q$. 
It is the inward pointing integer vector perpendicular to $Q$ of the smallest length. Below, by `the facet normal' associated to $Q$ we mean the primitive, inward pointing facet normal. 
\begin{example} \label{ex:polygon1}
Figure \ref{fig:polygon1} shows a full-dimensional polytope in $\R^2$ (a 2-dimensional polytope is also called a \emph{polygon}) together with its interior lattice points and primitive inward pointing facet normals. 
The minimal H-representation is given by 
\[ u_1 = (-2,-1), \quad u_2 = (1,2) \quad \text{and} \quad u_3 = (1,-1). \]
The supporting hyperplane $H_{u_2,a_2}$ is also shown in Figure \ref{fig:polygon1}, and its corresponding half-space $H_{u_2,a_2}^+$ (shaded in green) contains the polytope. 
\begin{figure}
\centering
\begin{tikzpicture}[scale=0.8]
\begin{axis}[%
width=2.5in,
height=2.19in,
scale only axis,
xmin=-9,
xmax=7,
ymin=-7,
ymax=7,
ticks = none, 
ticks = none,
axis background/.style={fill=white},
axis line style={draw=none} 
]

\addplot [color=mycolor5,solid,thick, fill = mycolor5!20!white,forget plot]
  table[row sep=crcr]{%
8	-7\\
8	8\\
-22	8\\	
8	-7\\
};

\addplot [color=mycolor1,solid,thick, fill = mycolor1!20!white,forget plot]
  table[row sep=crcr]{%
6	-6\\
0	6\\
-6	0\\	
6	-6\\
};

\addplot[only marks,mark=*,mark size=1.0pt,mycolor1!50!white
        ]  coordinates {
    (0,6) (-1,5) (0,5) (-2,4) (-1,4) (0,4) (1,4) (-3,3) (-2,3) (-1,3) (0,3) (1,3) 
};
\addplot[only marks,mark=*,mark size=1.0pt,mycolor1!50!white
        ]  table[row sep=crcr]{%
-4	2\\
-3	2\\
-2	2\\
-1	2\\
0	2\\
1	2\\
2	2\\
-5	1\\
-4	1\\
-3	1\\
-2	1\\
-1	1\\
0	1\\
1	1\\
2	1\\
-6	0\\	
-5	0\\	
-4	0\\
-3	0\\
-2	0\\
-1	0\\
0	0\\	
1	0\\
2	0\\
3	0\\
-4	-1\\
-3	-1\\
-2	-1\\
-1	-1\\
0	-1\\	
1	-1\\
2	-1\\
3	-1\\
-2	-2\\
-1	-2\\
0	-2\\	
1	-2\\
2	-2\\
3	-2\\
4	-2\\
0	-3\\	
1	-3\\
2	-3\\
3	-3\\
4	-3\\
2	-4\\
3	-4\\
4	-4\\
5	-4\\
4	-5\\
5	-5\\
6	-6\\
};

\node (A1) at (axis cs:2, 2) {};
\node (A2) at (axis cs:0, 1) {};
\node (B1) at (axis cs:-2,-2) {};
\node (B2) at (axis cs:-1, 0) {};
\node (C1) at (axis cs:-3,3) {};
\node (C2) at (axis cs:-2,2) {};
\draw [->,>=stealth,mycolor2, thick]  (axis cs:2, 2) -- (axis cs:0, 1);
\draw [->,>=stealth,mycolor2, thick] (axis cs:-2, -2) -- (axis cs:-1, 0);
\draw [->,>=stealth,mycolor2, thick] (axis cs:-3,3) -- (axis cs:-2,2);

\node (u1) at (axis cs:1,0.5) {$u_1$};
\node (u2) at (axis cs:-1.8,0) {$u_2$};
\node (u3) at (axis cs:-1.5,2.5) {$u_3$};
\node (H2) at (axis cs:-7,1.5) {$H_{u_2,a_2}$};
\node (H2+) at (axis cs:5.5,5) {$H_{u_2,a_2}^+$};

\end{axis}
\end{tikzpicture}%
\caption{Illustration of a lattice polytope of dimension 2 and its primitive inward pointing facet normals.}
\label{fig:polygon1}
\end{figure}
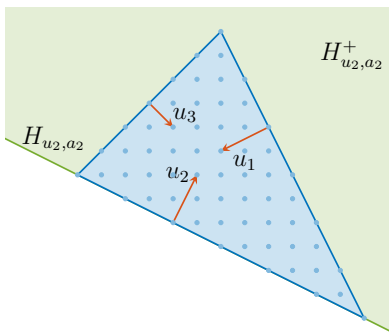
We note that, strictly speaking, the orange arrows do not belong in the same picture: they live in the dual plane $(\R^2)^\vee$. 
\end{example}
For any polytope $P \subset M_\R$ and any $\lambda \in \R, \lambda \geq 0$, we define the polytope $\lambda \cdot  P$ as $\lambda \cdot  P = \{\lambda p ~|~ p \in P \}$. This is called a \emph{dilation} of the polytope $P$ and all dilations are obtained by restricting scalar multiplication in $M_\R$ to $P$. 
\subsection{Kushnirenko's theorem} \label{sec:kush}
In this section we relate the degree of $X_\A \subset \PP^{s-1}$ to the volume of a convex lattice polytope. When $M = \Z^n$ and $\A \subset \Z^n$ is finite, $P = {\rm Conv}(\A)$ is a polytope in $\R^n$ with volume ${\rm Vol}(P) = \int_P  dx_1 \cdots dx_n$. To simplify some of the statements in this section, we will often assume that 
\begin{equation} \label{eq:assumptionkush}
\A = \{m_1, \ldots, m_s\} \subset \Z^n \text{ is such that } \Z' \A = \Z^n.
\end{equation}
Here is this section's main result, due to Kushnirenko \cite{kouchnirenko1976polyedres}. 
\begin{theorem}[Kushnirenko's theorem] \label{thm:kush}
Consider $\A \subset \Z^n$ satisfying \eqref{eq:assumptionkush}. The degree of $X_\A \subset \PP^{s-1}$ is the normalized volume $n! \,  {\rm Vol}(P)$ of the convex lattice polytope $P = {\rm Conv}(\A)$.
\end{theorem}
\begin{remark} The \emph{degree} of a projective variety $X \subset \PP^{s-1}$ can be thought of as the generic number of intersection points of $X$ with a generic linear space of dimension $s-1-\dim(X)$. A more precise definition can be given via the \emph{Hilbert polynomial} of $X$, see below. 
\end{remark}
\begin{remark}
The term \emph{normalized volume} in Theorem \ref{thm:kush} indicates that the volume measure is normalized such that a standard simplex $\Delta_n = {\rm Conv}(0,e_1,\ldots,e_n)$ in the lattice $\Z^n$ has volume 1: $n! \, {\rm Vol} (\Delta_n) = 1$. If $\A$ does not satisfy \eqref{eq:assumptionkush}, one needs to replace the number $n! \, {\rm Vol}(P)$ in Theorem \ref{thm:kush} with the normalized volume of ${\rm Conv}(\A)$ in the lattice $\Z' \A$. 
\end{remark}
Before proving Theorem \ref{thm:kush}, we state an important consequence for the study of systems of polynomial equations, alluded to in Example \ref{ex:example}.
\begin{corollary} \label{cor:unmixed}
For $\A \subset \Z^n$ satisfying \eqref{eq:assumptionkush}, consider the system of Laurent polynomial equations $f_1(t) = \cdots = f_n(t) = 0$ where $f_i(t) = \sum_{j=1}^s c_{ij} \, t^{m_j} \in \C[t_1^{\pm 1}, \ldots, t_n^{\pm 1}], i = 1, \ldots, n$. This system has at most $n! \, {\rm Vol}(P)$ isolated solutions in $(\C^*)^n$, with $P = {\rm Conv}(\A)$. Moreover, for generic coefficients $c_{ij}$, the number of solutions is precisely  $n! \, {\rm Vol}(P)$.
\end{corollary}
\begin{proof}
Let $L = \{C \cdot x = 0 \}$ be the linear space in $\PP^{s-1}$ defined by the matrix of coefficients $C = (c_{ij}) \in \C^{n \times s}$.  By Proposition \ref{prop:charlatticeproj} and the assumption \eqref{eq:assumptionkush}, the solutions of $f_1(t) = \cdots = f_n(t)= 0$ are in one to one correspondence with the points in $ L \cap X_\A \cap T_{X_\A}$. The set $L \cap X_\A$ contains at most $\deg(X_\A)$ isolated points, and for generic $C$, $|L \cap X_{\A}| = \deg(X_\A)$ and $L \cap X_\A \subset T_{X_\A}$, see \cite[Thm.~III' (ii)]{kouchnirenko1976polyedres}. The statement follows from Theorem \ref{thm:kush}.
\end{proof}
We present a proof of Theorem \ref{thm:kush} based on the lecture notes \cite{sottile2017ibadan} by Frank Sottile. It uses the following result by Hilbert, see \cite[Thm.~2.2.5]{telen2020thesis} and references therein. 
\begin{theorem} \label{thm:hilbert}
Let $X \subset \PP^{s-1}$ be a projective variety with coordinate ring $\C[X] = \C[x_1, \ldots, x_s]/I(X)$. For $d \gg 0$, the \emph{Hilbert function}
\[ {\rm HF}_X : \Z \rightarrow \N, \quad d \mapsto \dim_\C \C[X]_d \]
is given by a polynomial ${\rm HP}_X(d)$, called the \emph{Hilbert polynomial} of $X$, with leading term 
\[\frac{\deg(X)}{\dim(X)!} d^{\,\dim(X)}. \]
\end{theorem}
In order to use Theorem \ref{thm:hilbert} to prove Theorem \ref{thm:kush}, we must understand the coordinate ring $\C[X_\A]$ of $X_\A$ and its graded pieces. We use the notation $\widehat{\A}$ from Proposition \ref{prop:Ahat}.
\begin{lemma}
Let $\A = \{m_1, \ldots, m_s \} \subset M$ be a finite subset. The coordinate ring $\C[X_\A]$ of $X_\A \subset \PP^{s-1}$ is the semigroup algebra $\C[\N \widehat{\A}] \subset \C[M \oplus \Z]$, with grading 
\[ \C[ \N \widehat{\A} ] = \bigoplus_{d \in \Z} \C[\N \widehat{\A}]_d, \quad \C[\N \widehat{\A}]_d = \bigoplus_{m \in d \A} \C \cdot \chi^{(m,d)}. \]
Here $d \A = \{ \sum_{1\leq i_1 \leq \cdots \leq i_d \leq s} m_{i_d} \}$ and $\C[\N \widehat{\A}]_d = 0$ for $d < 0$. 
\end{lemma}
\begin{proof}
By Propositions \ref{prop:semigroup} and \ref{prop:Ahat}, $\C[X_\A] = \C[Y_{\widehat{\A}}] = \C[x_1, \ldots, x_s]/I(Y_{\widehat{\A}}) \simeq \C[\N \widehat{\A}]$. The last isomorphism is given by $[x_i] \mapsto \chi^{(m_i,1)}$, which induces the grading.
\end{proof}
\begin{example} \label{ex:kush}
Let $\A = \{0,2,3\} \subset \Z$, such that $\widehat{\A} = \{(0,1),(2,1),(3,1) \}$. The semigroup $\N \widehat{\A}$ is shown in \cite[Fig.~5]{sottile2017ibadan}. The number of characters $\chi^{(m,d)} \in \C[\N \widehat{\A}]_d$ is $1, 3, 6, 9, 12, \ldots$ for $d = 0,1,2,3,4,\ldots$. These are the values of the Hilbert function ${\rm HF}_{X_\A}$ of $X_\A \subset \PP^2$. The Hilbert polynomial is ${\rm HP}_{X_\A} = 3d$, and it agrees with ${\rm HF}_{X_\A}$ for $d \geq 1$. The vanishing ideal $I(X_\A) = I(Y_{\widehat{\A}})$ is $\langle xz^2-y^3 \rangle$. 
\end{example}
The previous example illustrates the following corollary. 
\begin{corollary} \label{cor:HF}
The Hilbert function of $X_\A$ is given by ${\rm HF}_{X_\A}(d) = |d \A|$. 
\end{corollary}
To link the leading coefficient of the Hilbert function to the volume of ${\rm Conv}(\A)$, we need some results from polyhedral geometry. Here is a famous result by Ehrhart \cite{Ehrhart}. 
\begin{theorem} \label{thm:ehrhart}
Let $P \subset \R^n$ be a convex lattice polytope of dimension $n$. The function $E_P : \N \rightarrow \N$ given by $d \mapsto |d \cdot P \cap \Z^n|$ is a polynomial with leading term ${\rm Vol}(P) \, d^{n}$. 
\end{theorem}
\begin{examplestar} \label{ex:chiara3}
The Ehrhart polynomial of a polytope $P$ can be computed using \texttt{Oscar.jl}. We illustrate this and some more functionalities for the polytope $P = {\rm Conv} ((0,0),(1,0),(0,1),(2,1),(1,2))$ in Figure \ref{fig:output2}.
\begin{figure}
\centering
\includegraphics[scale=0.3]{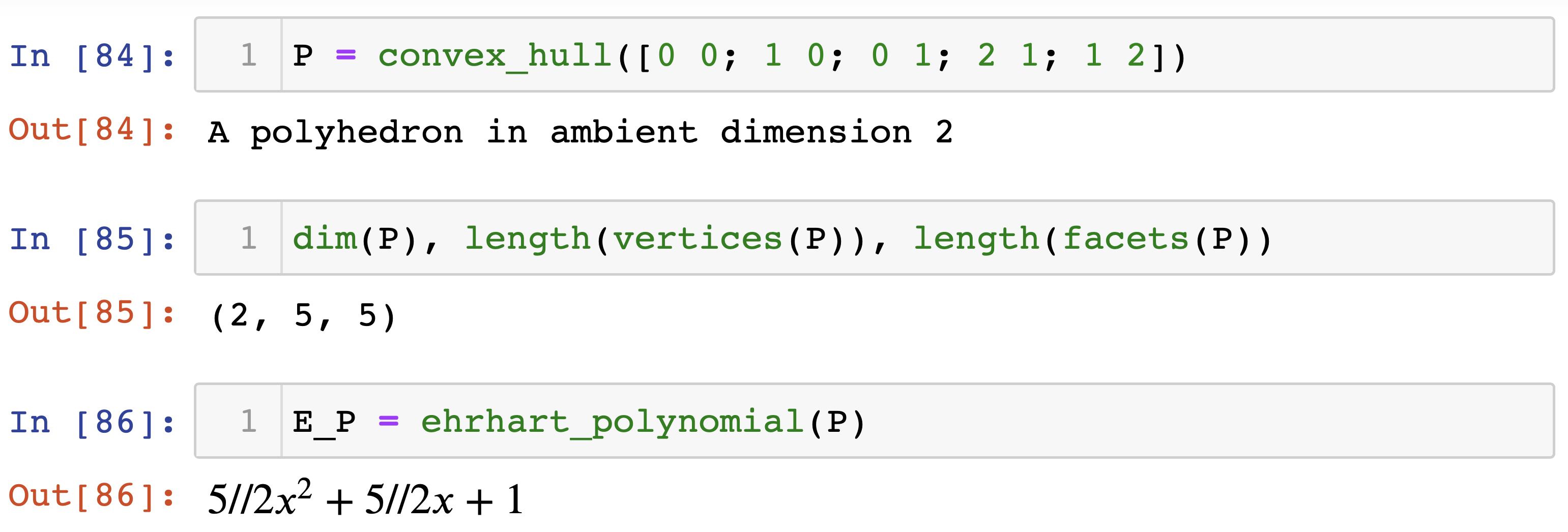}
\caption{The polygon $P$ has dimension 2. It has 5 vertices and 5 edges, and its Ehrhart polynomial is $\frac{5}{2} x^2 + \frac{5}{2} x + 1$.}
\label{fig:output2}
\end{figure}
\end{examplestar}
Note that we have the obvious inequality ${\rm HF}_{X_\A}(d) \leq E_P(d)$, for $d \geq 0$ and $P = {\rm Conv}(\A)$. If equality holds, we are done, since combining Theorems \ref{thm:hilbert} and \ref{thm:ehrhart} would give $\deg(X_\A) = n! {\rm Vol}(P)$. In general, the inequality may be strict for all $d$, see Example \ref{ex:kush}. However, the Hilbert and Ehrhart polynomials do have the same leading term, which is our strategy for proving Theorem \ref{thm:kush}.

Let $\A$ be as in \eqref{eq:assumptionkush} and $\sigma^\vee = {\rm Cone}(\widehat{\A}) \subset \R^{n+1}$. This gives a semigroup $\S_\sigma = \sigma^\vee \cap \Z^{n+1}$. 
\begin{remark} \label{rem:EhrhartVSHilbert}
Note that ${\rm HF}_{X_\A}(d) = E_P(d), d \geq 0$ if and only if $\S_\sigma$ is saturated in $M \oplus \Z$. In this case $X_\A$ is \emph{projectively normal}, see Exercise \ref{ex:projnormal}.
\end{remark}

\begin{lemma} \label{lem:trick}
If $\A$ satisfies \eqref{eq:assumptionkush}, there exists $m \in \N \widehat{\A}$ such that $m + \S_\sigma \subset \N \widehat{\A}$.
\end{lemma}
\begin{proof}
Let ${\cal B} = \{ \sum_{i=1}^s \lambda_i (m_i,1) \in \Z^{n+1} ~|~ \lambda_i \in [0,1) \cap \Q \}$. For a geometric interpretation, see Example \ref{ex:zonotope}. For each $b \in {\cal B}$, fix $a_i(b) \in \Z$ such that $b = \sum_{i =1}^s a_i(b) (m_i,1)$. Note that this is possible by the assumption $\Z' \A = \Z^n$. Fix $\nu \in \N$ such that $-\nu \leq a_i(b)$ for all $b \in {\cal B}$ and $i = 1, \ldots, s$. We define $m = \nu \, \sum_{i =1}^s (m_i, 1)$ and show that $m + \S_\sigma \subset \N \widehat{\A}$. 

If $m' \in m + \S_\sigma$, there exist $\alpha_i \in \Q_{\geq 0}$ such that $m' - m = \sum_{i=1}^s \alpha_i \, (m_i, 1)$. Set $\alpha_i = \lambda_i +  \gamma_i $ where $\lambda_i \in [0,1) \cap \Q$ and $\gamma_i \in \N$. Then $m'-m = \sum_{i=1}^s \lambda_i (m_i,1) + \sum_{i=1}^s \gamma_i (m_i,1)$, so that 
\begin{align*}
m' &= m + \sum_{i=1}^s \lambda_i (m_i,1) + \sum_{i=1}^s \gamma_i (m_i,1) = \sum_{i = 1}^s (\nu + a_i(b)) (m_i,1) + \sum_{i=1}^s \gamma_i (m_i,1),
\end{align*}
where $b = \sum_{i=1}^s \lambda_i (m_i,1) \in {\cal B}$, and both sums are in $\N \widehat{\A}$ by construction. 
\end{proof}
\begin{proof}[Proof of Theorem \ref{thm:kush}]
By the proof of Lemma \ref{lem:trick}, there exist $\tilde{m} \in \Z^n$, $e \in \N$ such that $m =  (\tilde{m},e) \in \S_\sigma$ satisfies $m + \S_\sigma \subset \N \widehat{\A}$. Therefore, for $d \geq e$ we have
\[ |(d - e) \cdot P \cap \Z^n| \leq |d \A| \leq |d \cdot P \cap \Z^n|. \]
In terms of Ehrhart polynomials and Hilbert functions, by Corollary \ref{cor:HF} this reads
\[ E_{P}(d-e) \leq {\rm HF}_{X_\A}(d) \leq E_P(d) \]
and for $d \gg 0$ we may replace $ {\rm HF}_{X_\A}(d)$ by  ${\rm HP}_{X_\A}(d)$. This implies that $ {\rm HP}_{X_\A}(d)$ and $E_P(d)$ have the same leading term, which implies Theorem \ref{thm:kush} by Theorems \ref{thm:hilbert} and \ref{thm:ehrhart}.
\end{proof}

\begin{example} \label{ex:zonotope}
The set ${\cal B}$ from the proof of Lemma \ref{lem:trick} contains the origin in $\Z^{n+1}$, and the interior lattice points of a \emph{zonotope} obtained by taking the Minkowski sum (see Definition \ref{def:minkowski}) of the line segments connecting $0$ with $(m_i,1)$. This is shown in \cite[Fig.~5]{sottile2017ibadan}.
\end{example}

\begin{examplestar} \label{ex:chiara4}
Let $\A = \{(0,0),(1,0),(0,1),(2,1),(1,2)\}$. The polytope $P = {\rm Conv}(\A)$ was computed in \texttt{Oscar.jl} in Example* \ref{ex:chiara3}. It is the pentagon shown in Figure \ref{fig:ex_7}. By Kushnirenko's theorem, the surface $X_\A \subset \PP^5$ has degree 5. This can be computed in \texttt{Oscar.jl} using the command \texttt{2*volume(P)}.
\begin{figure}[h]
    \centering
 \begin{tikzpicture}[scale=1]
\begin{axis}[%
width=2.0in,
height=2.0in,
scale only axis,
xmin=-1.2,
xmax=3.2,
ymin=-1.2,
ymax=3.2,
ticks = none, 
ticks = none,
axis background/.style={fill=white},
axis line style={draw=none} 
]

\draw[->,>=stealth] (axis cs:-1.1,0) -- (axis cs:3.2,0);
\draw[->,>=stealth] (axis cs:0,-1.1) -- (axis cs:0,3.2);

\addplot [color=mycolor1,solid,fill opacity=0.2,fill = mycolor1,forget plot]
  table[row sep=crcr]{%
0   0\\
0	1\\
1	2\\	
2   1\\
1   0\\
0	0\\
};

\addplot [color=mycolor5,solid,fill opacity=0.2,fill = mycolor5,forget plot]
  table[row sep=crcr]{%
0.03	1\\
1	1.97\\	
1.97   1\\
1   0.03\\
0.03	1\\
};

\addplot[only marks,mark=*,mark size=1.5pt,mycolor1,
        ]  coordinates {
    (0,0) (0,1) (1,0) (2,1) (1,2) (1,1)
};

\addplot[only marks,mark=o,mark size=2.3pt,thick,mycolor5,
        ]  coordinates {
    (0,1) (1,0) (2,1) (1,2) (1,1)
};

\addplot[only marks,mark=*,mark size=1.5pt,black,
        ]  coordinates {
    (-1,-1) (-1,0) (-1,1) (-1,2) (-1,3)
    (0,-1) (0,2) (0,3)
    (1,-1) (1,3)
    (2,-1) (2,0) (2,2) (2,3)
    (3,-1) (3,0) (3,1) (3,2) (3,3)
};

\end{axis}
\end{tikzpicture}
    \caption{The polygons $P = {\rm Conv}(\A)$ and $P' = {\rm Conv}(\A')$ from Example* \ref{ex:chiara4}.}
    \label{fig:ex_7}
\end{figure}
We verify this using Corollary \ref{cor:unmixed} and \texttt{HomotopyContinuation.jl} (v2.6.4) \cite{breiding2018homotopycontinuation}. We construct random polynomials $f_1, f_2$ as follows: 
\begin{minted}{julia}
using HomotopyContinuation, LinearAlgebra
@var t[1:2]
A = [0 0; 1 0; 0 1; 2 1; 1 2]
monomials = [prod(t.^A[:,i]) for i = 1:5]; coeffs = [randn(5) for i = 1:2];
f = [dot(coeffs[i], monomials) for i = 1:2]
\end{minted}
The command \texttt{solve(f)} returns five solutions. 
Alternatively, we can count the number of points in $L \cap X_\A \cap T_\A$ as in the proof of Corollary \ref{cor:unmixed}. We compute the \texttt{toric\_ideal} associated to $\widehat{\A}$ using \texttt{Oscar.jl}; we get the ideal of $\C[x_1,x_2,x_3,x_4,x_5]$ generated by
\[ x_3x_4-x_2x_5, \, x_2x_3^2-x_1^2x_5, \, x_2^2x_3-x_1^2x_4, \, x_1^2x_4^2 - x_2^3x_5.\]
These are the equations of $X_\A \subset \PP^4$. The degree of $X_\A$ counts the number of intersection points with a generic linear space of codimension $2$. This is obtained as follows:
\begin{minted}{julia}
@var x[1:5]
systemA = [x[3]*x[4]-x[2]*x[5]; x[2]*x[3]^2-x[1]^2*x[5]; 
          x[2]^2*x[3]-x[1]^2*x[4];  x[1]^2*x[4]^2 - x[2]^3*x[5];
          dot(coeffs[1],x); dot(coeffs[2],x); x[1]-1]
\end{minted}
Note the dehomogenization $x_1 = 1$. Entering \texttt{solve(systemA)} gives $5$ solutions. They are the images of the previously computed solutions in $(\C^*)^2$ under $\pi \circ \Phi_{\A}$ from Section \ref{sec:projmonom}.

We now remove the point $(0,0)$ from $\A$ to obtain $\A'$, so that the assumption \eqref{eq:assumptionkush} is no longer satisfied. The degree of $X_{\A'} \subset \PP^4$ is 2, while the normalized volume of $P' = {\rm Conv}(\A')$ in $\R^2$ is 4. The number of solutions to a new set of equations with support in $\A'$ is 4. This discrepancy is due to the fact that the map $(\pi \circ \Phi_{\A'}): (\C^*)^2 \rightarrow T_{X_{\A'}}$ is two-to-one, since $\Z' \A$ has index 2 in $\Z^2$. See Exercise \ref{ex:twotoone}.
\end{examplestar}
\begin{remark}
Another geometric invariant of $X_\A$ that is encoded by $P = {\rm Conv}(\A)$ is its dimension. By Proposition \ref{prop:charlatticeproj} and the fact that the rank of $\Z' \A$ is equal to the dimension of $P$, we have $\dim X_\A = \dim P$. This holds regardless of the assumption \eqref{eq:assumptionkush}.
\end{remark}

\subsection{Affine pieces of a projective toric variety}
Let $U_i = \PP^{s-1} \setminus V(x_i)$ for $i = 1, \ldots, s$ be the standard affine charts of $\PP^{s-1}$. Recall that $U_i \simeq \C^{s-1}$ via 
\begin{equation} \label{eq:dehom}
(x_1: \cdots: x_s) \mapsto \left ( \frac{x_1}{x_i}, \ldots, \frac{x_{i-1}}{x_i}, \frac{x_{i+1}}{x_i}, \ldots, \frac{x_s}{x_i} \right ).
\end{equation}
Note that $T_{\PP^{s-1}}  = \bigcap_{i=1}^s U_i$. Since $T_{X_\A} = T_{\PP^{s-1}} \cap X_\A \subset U_i$ and $X_\A$ is the Zariski closure in $\PP^{s-1}$ of $T_{X_\A}$, the affine piece $X_\A \cap U_i$ is the Zariski closure in $U_i$ of $T_{X_\A}$. Composing $T \overset{\Phi_\A}{\longrightarrow} U_i$ with \eqref{eq:dehom} we obtain the map 
\[ t \mapsto (\chi^{m_1-m_i}(t), \ldots, \chi^{m_{i-1}-m_i}(t), \chi^{m_{i+1}-m_i}(t), \ldots, \chi^{m_s-m_i}(t))\]
from which we immediately see the following. 
\begin{proposition} \label{prop:affcharts}
The $i$-th affine piece $X_\A \cap U_i$ of $X_\A$ is isomorphic to the affine toric variety $Y_{\A_i} \simeq \Specm(\C[\S_i])$ where $\A_i = \A - m_i = \{m_1-m_i, \ldots, m_{i-1}-m_i, m_{i+1}-m_i,\ldots,m_s\}$ and $\S_i = \N \A_i$. 
\end{proposition}
\begin{remark} \label{rem:sametorus}
Note that the character lattice $\Z' \A$ of $T_{X_\A}$ is the same as the character lattice $\Z \A_i$ of the dense torus of $X_A \cap U_i \simeq \C[\S_i]$, as expected. See Exercise \ref{ex:afflattice}.
\end{remark}
For an affine variety $V$ with coordinate ring $\C[V]$ we write $V_f$ for the affine variety $\Specm(\C[V]_f)$, $f \in \C[V]$. The inclusions
\[ X_\A \cap U_i \supset X_\A \cap U_i \cap U_j \subset X_\A \cap U_j \]
 can be described nicely in terms of the semigroup algebras $\C[\S_i], \C[\S_j]$ as follows. For the inclusion $X_\A \cap U_i \supset X_\A \cap U_i \cap U_j$, note that 
 \[ X_\A \cap U_i \cap U_j = \{ x \in U_i ~|~ x_j/x_i \neq 0 \} = (X_\A \cap U_i)_{[x_j/x_i]} \simeq \Specm(\C[\S_i]_{\chi^{m_j-m_i}}). \]
 Hence $X_\A \cap U_i \supset X_\A \cap U_i \cap U_j$ is given by the inclusion of $\C[\S_i]$ in its localization $\C[\S_i]_{\chi^{m_j-m_i}}$. Arguing analogously for $X_\A \cap U_j$, we obtain 
 \[ \C[\S_i] \subset \C[\S_i]_{\chi^{m_j-m_i}} \simeq \C[\S_j]_{\chi^{m_i-m_j}} \supset \C[\S_j]. \]
 This will help us prove another nice connection between $P = {\rm Conv}(\A) \subset M_\R$ and our projective toric variety $X_\A$.
 
 \begin{proposition} \label{prop:projcover}
 Let $\A = \{m_1, \ldots, m_s\} \subset M$, $P = {\rm Conv}(\A) \subset M_\R$ and ${\cal V} = \{ i ~|~ m_i \in \A \text{ is a vertex of $P$} \}$.  We have 
 \[ X_\A = \bigcup_{j=1}^s X_\A \cap U_j = \bigcup_{i \in {\cal V}} X_\A \cap U_i.\] 
 \end{proposition}
 \begin{proof}
 The first equality is the standard covering of $X_\A$ induced by $\PP^{s-1} = \bigcup_{i=1}^s U_i$. For the second equality, we show that $X_\A \cap U_j \subset U_{i^*}$ for some $i^* \in {\cal V}$. Let $M_\Q = M \otimes_\Z \Q$. Since $m_j \in P \cap M_\Q$, there exist $r_i \in \Q_{\geq 0}, i \in {\cal V}$ such that $m_j =  \sum_{i\in {\cal V}} r_i m_i$ and $\sum_{i \in {\cal V}} r_i = 1$. Clearing denominators gives 
 \[ k m_j = \sum_{i \in {\cal V}} k_i m_i, \quad\text{ with } k, k_i \in \N, \,  \sum_{i \in {\cal V}} k_i = k. \]
 This gives $\sum_{i \in {\cal V}} k_i (m_j - m_i) = 0$. We now choose $i^*$ such that $k_{i^*} > 0$ and rewrite
 \[ m_j - m_{i^*} = \sum_{i \in {\cal V} \setminus \{i^*\}} k_i (m_i - m_j) + (k_{i^*} -1)(m_{i^*}-m_j), \]
 which shows that $m_j-m_{i^*} \in \S_j$. Therefore $\chi^{m_{i^*}-m_j}$ is invertible in $\C[\S_j]$, and $\C[\S_j]_{\chi^{m_{i^*}-m_j}} = \C[\S_j]$. The statement follows from Proposition \ref{prop:affcharts} and
 \[ X_\A \cap U_{i^*} \cap U_j \simeq \Specm(\C[\S_j]_{\chi^{m_{i^*}-m_j}}) = \Specm(\C[\S_j]) \simeq X_\A \cap U_j. \qedhere \]
 \end{proof}
 
 \subsection{Very ample and normal polytopes}
 In previous sections we constructed a projective toric variety $X_\A$ from a set of lattice points $\A$, and showed that the polytope ${\rm Conv}(\A)$ encodes many of its properties. In this section, we start from a convex lattice polytope $P$ and associate a projective toric variety to it by using all lattice points contained in it: $\A = P \cap M$. We will see that the nicest such varieties come from \emph{very ample} convex lattice polytopes.
 
 \begin{definition}[Very ample] \label{def:veryample}
 A convex lattice polytope $P \subset M_\R$ is \emph{very ample} if for every vertex $v \in P \cap M$, the semigroup $\N (P \cap M - v)$ is saturated in the lattice $M$. 
 \end{definition}
 
To each vertex $v$ of a polytope $P$ we associate the cone $\sigma_v^\vee = {\rm Cone}(P - v) \subset M_\R$. By Theorem \ref{thm:normal}, a very ample polytope has the property that for each vertex $v \in P \cap M$, the lattice points in $(P\cap M - v)$ suffice to generate the semigroup $\sigma_v^\vee \cap M$. 

Consider a full-dimensional, very ample lattice polytope $P \subset M_\R$. Let $P \cap M = \{ m_1, \ldots, m_s\}$ and let ${\cal V} = \{ i ~|~ m_i \in P \cap M \text{ is a vertex of } P \}$. By Proposition \ref{prop:projcover}, the toric variety $X_{P \cap M} \subset \PP^{s-1}$ is covered by $X_{P \cap M} \cap U_i, i \in {\cal V}$. We will write $\sigma_i = \sigma_{m_i} = (\sigma_{m_i}^\vee)^\vee$ for $i \in {\cal V}$ and $\S_i = \N(P \cap M - m_i)$. 

\begin{theorem} \label{thm:normalPTV}
With the above set-up, for every vertex $m_i, i \in {\cal V}$, $X_{P \cap M} \cap U_i $ is isomorphic to the normal affine toric variety $U_{\sigma_i}= \Specm(\C[\sigma_i^\vee \cap M])$ corresponding to the $n$-dimensional, strongly convex rational cone $\sigma_i$. The dense torus of $X_{P \cap M}$ has character lattice $M$. 
\end{theorem}
\begin{proof}
The cone $\sigma_i^\vee = {\rm Cone}(P \cap M - m_i)$ is strongly convex of dimension $\dim P = n$. Hence, by Definition \ref{def:strongconvex}, $\sigma_i$ is strongly convex of dimension $n$ too. By very ampleness, $\S_i$ is saturated in $M$. Therefore $X_{P\cap M} \cap U_i \simeq \Specm(\C[\S_i]) = \Specm(\C[\sigma_i^\vee \cap M]) = U_{\sigma_i}$. By Proposition \ref{prop:7}, the torus of $U_{\sigma_i}$ has character lattice $M$. Hence, so does $T_{X_{P\cap M}}$, see Remark \ref{rem:sametorus}.
\end{proof}

A projective variety is called \emph{normal} if it is covered by normal affine varieties. Theorem \ref{thm:normalPTV} ensures that when $P$ is very ample, the projective toric variety $X_{P \cap M}$ is normal. Here is another property of convex lattice polytopes which is stronger than very ampleness.

\begin{definition}[Normal polytopes]
A convex lattice polytope $P$ is \emph{normal} if $(k\cdot P) \cap M + (\ell \cdot P) \cap M = ((k+\ell)\cdot P) \cap M$ for all $k, \ell \in \N$. 
\end{definition} 

\begin{exercise} \label{ex:projnormal}
Show that if $P$ is normal, the affine cone $Y_{\widehat{P \cap M}}$ over $X_{P \cap M}$ is a normal affine toric variety. Projective varieties whose affine cone is normal are called \emph{projectively normal}. 
\end{exercise}
The following properties will be useful, see \cite[Thm.~2.2.12, Prop.~2.2.18]{cox2011toric}. 

\begin{proposition} \label{prop:dilatenormal}
Let $P \subset M_\R$ be a full dimensional convex lattice polytope. Then $k \cdot P$ is normal for all $k \geq n-1$. 
\end{proposition}
\begin{proposition} \label{prop:normalveryample}
Every normal polytope $P$ is very ample. 
\end{proposition}

\begin{remark}
The converse of Proposition \ref{prop:normalveryample} is not true, see \cite[Ex.~2.2.20]{cox2011toric}.
\end{remark}

\subsection{Normal fans} \label{sec:normalfans}
In this section, we explain the main reason for defining $U_{\sigma_i}$ from the cone $\sigma_i$ \emph{dual} to the cone $\sigma_i^\vee$, even though the latter seems more closely related to its coordinate ring. The cones $\sigma_v$, where $v$ is a vertex of a full dimensional lattice polytope $P$, fit together very nicely.

\begin{example} \label{ex:fan1}
Consider the standard simplex ${\rm Conv}((0,0),(1,0),(0,1)) \subset \R^2$. The cones $\sigma_1, \sigma_2, \sigma_3$ and their duals are shown in Figure \ref{fig:illustratefan}. 
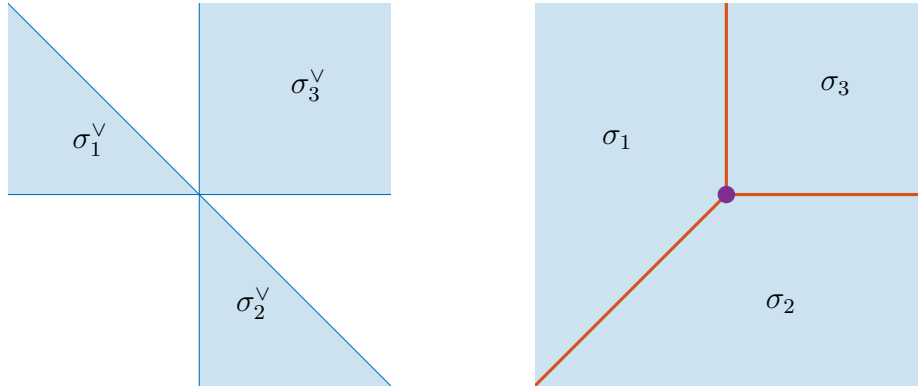
\begin{figure}
\centering
\begin{tikzpicture}[scale=1]
\begin{axis}[%
width=2in,
height=2in,
scale only axis,
xmin=-3.5,
xmax=3.5,
ymin=-3.5,
ymax=3.5,
ticks = none, 
ticks = none,
axis background/.style={fill=white},
axis line style={draw=none} 
]


\addplot [color=mycolor1,solid,fill opacity=0.2,fill = mycolor1,forget plot]
  table[row sep=crcr]{%
 5 0\\
5 5\\	
0 5\\
0 0 \\
5 0 \\
};

\addplot [color=mycolor1,solid,fill opacity=0.2,fill = mycolor1,forget plot]
  table[row sep=crcr]{%
 0 0 \\
0 -5 \\
 5 -5\\
 0 0\\
};

\addplot [color=mycolor1,solid,fill opacity=0.2,fill = mycolor1,forget plot]
  table[row sep=crcr]{%
 0 0 \\
 -5 0\\
 -5 5\\
 0 0\\
};

\node (P) at (axis cs:2,2) {$\sigma_3^\vee$};
\node (P) at (axis cs:1,-2) {$\sigma_2^\vee$};
\node (P) at (axis cs:-2,1) {$\sigma_1^\vee$};

\end{axis}
\end{tikzpicture} 
\qquad \qquad
\begin{tikzpicture}[scale=1]
\begin{axis}[%
width=2in,
height=2in,
scale only axis,
xmin=-3.5,
xmax=3.5,
ymin=-3.5,
ymax=3.5,
ticks = none, 
ticks = none,
axis background/.style={fill=white},
axis line style={draw=none} 
]


\addplot [color=mycolor1,solid,fill opacity=0.2,fill = mycolor1,forget plot]
  table[row sep=crcr]{%
 5 0\\
5 5\\	
0 5\\
0 0 \\
5 0 \\
};

\addplot [color=mycolor1,solid,fill opacity=0.2,fill = mycolor1,forget plot]
  table[row sep=crcr]{%
 0 0 \\
 -5 -5 \\
 5 -5\\
 5 0\\
 0 0\\
};

\addplot [color=mycolor1,solid,fill opacity=0.2,fill = mycolor1,forget plot]
  table[row sep=crcr]{%
 0 0 \\
 -5 -5 \\
 -5 5\\
 0 5\\
 0 0\\
};

\addplot [very thick, color=mycolor2,solid,fill opacity=0.2,fill = mycolor1,forget plot]
  table[row sep=crcr]{%
 0 0 \\
 5 0 \\
};

\addplot [very thick, color=mycolor2,solid,fill opacity=0.2,fill = mycolor1,forget plot]
  table[row sep=crcr]{%
 0 0 \\
 -5 -5 \\
};

\addplot [very thick, color=mycolor2,solid,fill opacity=0.2,fill = mycolor1,forget plot]
  table[row sep=crcr]{%
 0 0 \\
 0 5 \\
};

\addplot[only marks,mark=*,mark size=3.1pt,mycolor4
        ]  coordinates {
  (0,0)
};

\node (P) at (axis cs:2,2) {$\sigma_3$};
\node (P) at (axis cs:1,-2) {$\sigma_2$};
\node (P) at (axis cs:-2,1) {$\sigma_1$};

\end{axis}
\end{tikzpicture} 
\caption{A fan from a triangle.}
\label{fig:illustratefan}
\end{figure}
\end{example}

\begin{definition}[Fan] \label{def:fan}
A \emph{fan} in $N_\R$ is a finite collection $\Sigma$ of strongly convex rational polyhedral cones such that 
\begin{itemize}
\item if $\sigma \in \Sigma$ and $\tau \preceq \sigma$ is a face, then $\tau \in \Sigma$, 
\item if $\tau = \sigma \cap \sigma'$ for $\sigma, \sigma' \in \Sigma$, then $\tau \preceq \sigma$ and $\tau \preceq \sigma'$.
\end{itemize}
\end{definition}

\begin{exercise}
Show that the right part of Figure \ref{fig:illustratefan} represents a fan with three two-dimensional cones, three one-dimensional cones and one zero-dimensional cone. Do you see a correspondence with the faces of the standard simplex from Example \ref{ex:fan1}? 
\end{exercise}



\begin{definition}[Normal fan] \label{def:normalfan}
Let $P \subset M_\R$ be full dimensional and let  $P = \bigcap_{F \text{ facet of } P} H_{u_F,a_F}^+$ be the minimal H-representation of $P$. The \emph{normal fan} of $P$ is 
\[ \Sigma_P = \{ \sigma_{Q} ~|~ Q \preceq P \}, \quad \text{where } \sigma_{Q} = {\rm Cone}(\{ u_F~|~ Q \subset F \} ). \]
\end{definition}
The fact that $\Sigma_P$ is a fan is \cite[Thm.~2.3.2]{cox2011toric}. Here are some properties of normal fans.
\begin{proposition} \label{prop:fanproperties}
Let $P \subset M_\R$ be a full-dimensional lattice polytope and for each face $Q \preceq P$, let $\sigma_Q$ be as in Definition \ref{def:normalfan}. We have 
\begin{enumerate}
\item $\dim Q + \dim \sigma_Q = n$ for all $Q \preceq P$, 
\item $N_\R = \bigcup_{v \text{ vertex of } P} \sigma_v = \bigcup_{Q \preceq P} \sigma_Q$, 
\item for any $m \in M$, $k \in \N \setminus \{0\}$, $\Sigma_{k \cdot P + m} = \Sigma_P$. 
\end{enumerate}
\end{proposition}
For a proof of Proposition \ref{prop:fanproperties}, see \cite[Prop.~2.3.8 and 2.3.9]{cox2011toric}.
\begin{examplestar} \label{ex:chiara5}
Here is how normal fans can be computed in \texttt{Oscar.jl}. We consider the pentagon $P$ from Example* \ref{ex:chiara3} and execute
\texttt{normal\_fan(P)}. The result is Figure \ref{fig:ex_8}.
\begin{figure}[h]
    \centering
    \begin{tikzpicture}[scale=1]
\begin{axis}[%
width=2.0in,
height=2.0in,
scale only axis,
xmin=-2.15,
xmax=2.15,
ymin=-2.15,
ymax=2.15,
ticks = none, 
ticks = none,
axis background/.style={fill=white},
axis x line = middle, 
axis y line = middle, 
axis line style={draw=none} 
]

\addplot [color=black,solid,fill opacity=0.2,fill = mycolor1,forget plot]
  table[row sep=crcr]{%
5   0\\
0   0\\
0   5\\
};

\addplot [color=black,solid,fill opacity=0.2,fill = mycolor2,forget plot]
  table[row sep=crcr]{%
0   3.15\\
0   0\\
-2.2   2.2\\
};

\addplot [color=black,solid,fill opacity=0.2,fill = mycolor3,forget plot]
  table[row sep=crcr]{%
-2.2   2.2\\
0   0\\
-2.2   -2.2\\
};

\addplot [color=black,solid,fill opacity=0.2,fill = mycolor4,forget plot]
  table[row sep=crcr]{%
-2.2   -2.2\\
0   0\\
2.2   -2.2\\
};

\addplot [color=black,solid,fill opacity=0.2,fill = mycolor5,forget plot]
  table[row sep=crcr]{%
2.2   -2.2\\
0   0\\
3.15   0\\
};

\addplot[only marks,mark=*,mark size=1.5pt,black!80,
        ]  coordinates {
    (-2,-2) (-2,2)
    (-1,-1) (-1,1)
    (0,1) (0,2) (0,3)
    (1,-1) (1,0)
    (2,-2) (2,0)
    (3,0)
};

\addplot[only marks,mark=*,mark size=1.5pt,mycolor1,
        ]  coordinates {
    (1,1) (1,2) (1,3)
    (2,1) (2,2) (2,3)
    (3,1) (3,2) (3,3)
};

\addplot[only marks,mark=*,mark size=1.5pt,mycolor2,
        ]  coordinates {
    (-1,2) (-1,3)
    (-2,3)
};

\addplot[only marks,mark=*,mark size=1.5pt,mycolor3,
        ]  coordinates {
    (-2,-1) (-2,0) (-2,1)
    (-1,0)
};

\addplot[only marks,mark=*,mark size=1.5pt,mycolor4,
        ]  coordinates {
    (-1,-2)
    (0,-1) (0,-2)
    (1,-2)
};

\addplot[only marks,mark=*,mark size=1.5pt,mycolor5,
        ]  coordinates {
    (2,-1)
    (3,-1) (3,-2)
};

\end{axis}
\end{tikzpicture}
    \caption{The normal fan from Example* \ref{ex:chiara5}.}
    \label{fig:ex_8}
\end{figure}
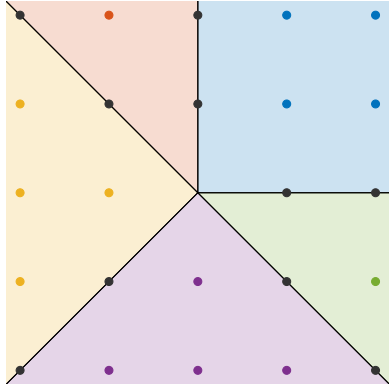
\end{examplestar}
\begin{proposition} \label{prop:normalfanglue}
Let $P \subset M_\R$ be a full-dimensional, very ample convex lattice polytope. Let $P \cap M = \{m_1, \ldots, m_s\}$ and ${\cal V} = \{ i ~|~ m_i \text{ is a vertex of } P \}$. Let $i, j \in {\cal V}, i \neq j$ and let $Q$ be the smallest face of $P$ containing $m_i$ and $m_j$. Then 
\[ X_{P \cap M} \cap U_i \cap U_j \simeq U_{\sigma_Q} = \Specm( \C[\sigma_Q^\vee \cap M]) \]
and the inclusions
\[ X_{P \cap M} \cap U_i \supset X_{P\cap M} \cap U_i \cap U_j \subset X_{P \cap M} \cap U_j \]
are given by
\[ \C[\S_i] \subset \C[\S_i]_{\chi^{m_j-m_i}} \simeq \C[\S_j]_{\chi^{m_i-m_j}} \supset \C[\S_j]. \]
\end{proposition}
\begin{proof}
The statement about the inclusions follows from the discussion preceding Proposition \ref{prop:projcover}. It suffices to show that $\C[\S_i]_{\chi^{m_j-m_i}} = \C[\sigma_Q^\vee \cap M]$. We refer to \cite[Prop.~2.3.13]{cox2011toric}.
\end{proof}

\begin{example} \label{ex:fromthesis}
Proposition \ref{prop:normalfanglue} is illustrated in Figure \ref{fig:localizcone}, taken from \cite[App.~E]{telen2020thesis}.
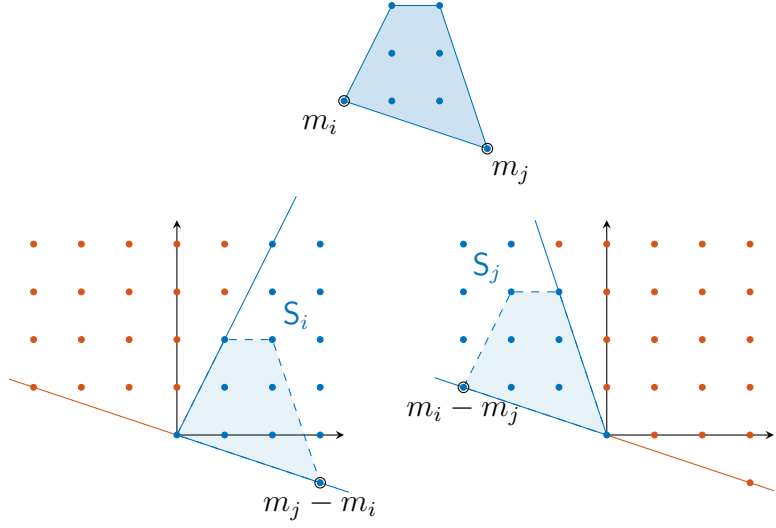
\begin{figure}
\centering
\begin{tikzpicture}[scale=1]
\begin{axis}[%
width=4in,
height=3in,
scale only axis,
xmin=-8.5,
xmax=7.5,
ymin=-2.5,
ymax=9.5,
ticks = none, 
ticks = none,
axis background/.style={fill=white},
axis line style={draw=none} 
]

\draw[->,>=stealth] (axis cs:-5,0) -- (axis cs:-5,4.5);
\draw[->,>=stealth] (axis cs:-5,0) -- (axis cs:-1.5,0);
\draw[->,>=stealth] (axis cs:4,0) -- (axis cs:7.5,0);
\draw[->,>=stealth] (axis cs:4,0) -- (axis cs:4,4.5);

\addplot [color=mycolor1,solid,fill opacity=0.2,fill = mycolor1,forget plot]
  table[row sep=crcr]{%
-1.5	7\\
-0.5	9\\	
0.5	9\\
1.5	6\\
-1.5	7\\
};

\addplot[only marks,mark=*,mark size=1.1pt,mycolor1
        ]  coordinates {
    (-1.5,7) (-0.5,7) (0.5,7) (1.5,6) (-0.5,8) (0.5,8) (-0.5,9) (0.5,9)
};

\addplot [color=mycolor1,dashed,fill opacity=0.1,fill = mycolor1,forget plot]
  table[row sep=crcr]{%
-5	0\\
-4	2\\	
-3	2\\
-2	-1\\
-5	0\\
};

\addplot[only marks,mark=*,mark size=1.1pt,mycolor1
        ]  coordinates {
    (-5,0) (-4,0) (-4,1) (-4,2) (-3,0) (-3,1) (-3,2) (-3,3) (-3,4) (-2,-1) (-2,0) (-2,1) (-2,2) (-2,3) (-2,4)
};

\addplot[only marks,mark=*,mark size=1.1pt,mycolor2
        ]  coordinates {
    (-8,1) (-8,2) (-8,3) (-8,4)
    (-7,1) (-7,2) (-7,3) (-7,4)
    (-6,1) (-6,2) (-6,3) (-6,4)
    (-5,1) (-5,2) (-5,3) (-5,4)
    (-4,3) (-4,4)
};

\addplot [color=mycolor2,solid,fill opacity=0.2,fill = mycolor2,forget plot]
  table[row sep=crcr]{%
-5	0\\
-8.6	1.2\\
};

\addplot [color=mycolor1,solid,fill opacity=0.2,fill = mycolor1,forget plot]
  table[row sep=crcr]{%
-5	0\\
-2.5	5\\
};

\addplot [color=mycolor1,solid,fill opacity=0.2,fill = mycolor1,forget plot]
  table[row sep=crcr]{%
-5	0\\
-1.4 -1.2\\
};

\addplot [color=mycolor1,dashed,fill opacity=0.1,fill = mycolor1,forget plot]
  table[row sep=crcr]{%
1	1\\
2	3\\	
3	3\\
4	0\\
1	1\\
};

\addplot [color=mycolor1,solid,fill opacity=0.2,fill = mycolor1,forget plot]
  table[row sep=crcr]{%
4	0\\
0.4	1.2\\
};

\addplot [color=mycolor1,solid,fill opacity=0.2,fill = mycolor1,forget plot]
  table[row sep=crcr]{%
4	0\\
2.5 4.5\\
};

\addplot[only marks,mark=*,mark size=1.1pt,mycolor1
        ]  coordinates {
    (1,1) (1,2) (1,3) (1,4)
    (2,1) (2,2) (2,3) (2,4)
    (3,1) (3,2) (3,3)
    (4,0)
};

\addplot[only marks,mark=*,mark size=1.1pt,mycolor2
        ]  coordinates {
    (3,4)
    (4,1) (4,2) (4,3) (4,4)
    (5,0) (5,1) (5,2) (5,3) (5,4)
    (6,0) (6,1) (6,2) (6,3) (6,4)
    (7,-1) (7,0) (7,1) (7,2) (7,3) (7,4)
};

\addplot[only marks,mark=o,mark size=2pt,
        ]  coordinates {
    (-2,-1) (1,1) (-1.5,7) (1.5,6)
};

\addplot [color=mycolor2,solid,fill opacity=0.2,fill = mycolor2,forget plot]
  table[row sep=crcr]{%
4	0\\
7.6	-1.2\\
};

\node (P) at (axis cs:-2,6.5) {$m_i$};
\node (P) at (axis cs:2,5.5) {$m_j$};
\node (P) at (axis cs:-2,-1.5) {$m_j - m_i$};
\node (P) at (axis cs:1,0.5) {$m_i - m_j$};
\node[mycolor1] (P) at (axis cs:-2.5,2.5) {$\S_i$};
\node[mycolor1] (P) at (axis cs:1.5,3.5) {$\S_j$};

\end{axis}
\end{tikzpicture} 
\caption{An illustration of Proposition \ref{prop:normalfanglue}.}
\label{fig:localizcone}
\end{figure}
\end{example}

It follows from Proposition \ref{prop:normalfanglue} that when $P$ is full-dimensional and very ample, the affine pieces of $X_{P \cap M}$ and the way they patch together to form $X_{P \cap M}$ are completely determined by the normal fan $\Sigma_P$. Consequently, if $P, P'$ are both very ample with the same normal fan, we have $X_{P \cap M} \simeq X_{P' \cap M}$. That is, $X_{P \cap M}$ and $X_{P' \cap M}$ are the \emph{same} projective toric variety, embedded in a different way. 

\begin{example}
The standard simplex in Example \ref{ex:fan1} gives the toric variety $X_{P \cap M} = \PP^2$. Its dilation $k \cdot P$ is very ample for all $k$, and has the same normal fan by Proposition \ref{prop:fanproperties}. The toric surface $X_{(k \cdot P) \cap M} \subset \PP^{ \left ( \begin{smallmatrix}
k+2 \\ 2
\end{smallmatrix} \right ) - 1}$ is the $k$-th Veronese embedding of $\PP^2$. 
\end{example}

\begin{examplestar} \label{ex:chiara6} 
The projective toric variety $X_{P\cap M}$ where $P$ is the polygon from Example \ref{ex:fromthesis} can be defined in \texttt{Oscar.jl} as follows:
\begin{minted}{julia}
P = convex_hull([0 0; 3 -1; 2 2; 1 2])
X = NormalToricVariety(P)
\end{minted}
Recall that all polygons are normal by Proposition \ref{prop:dilatenormal}. Therefore, we can embed $X_{P\cap M} \subset \PP^7$ via the toric morphism induced by the map of tori $(\C^*)^3 \to (\C^*)^8$ associated to the matrix
\begin{equation*}
    A = 
    \begin{bmatrix}
    0 & 1 & 1& 1 & 2 & 2 & 2 & 3\\
    0 & 0 & 1 & 2 & 0 & 1 & 2 & -1 \\
    1 & 1 & 1 & 1 & 1 & 1 & 1 & 1
    \end{bmatrix}.
\end{equation*}
Its columns are the points in $\widehat{\A}$ obtained from \texttt{lattice\_points(P)}. The toric ideal $I_{\widehat{\A}}$ is 
\[
\begin{matrix} \langle -x_5 x_7 + x_6^2, \quad
 -x_4 x_8 + x_5 x_6, \quad
 -x_3 x_7 + x_4 x_6, \quad
 -x_2 x_7 + x_3 x_6,\quad
 -x_3 x_8 + x_5^2, \\
 -x_2 x_7 + x_4 x_5, \quad
 -x_2 x_6 + x_3 x_5, \quad
 -x_1 x_7 + x_2 x_4,\quad
 -x_1 x_7 + x_3^2, \quad
 -x_1 x_6 + x_2 x_3, \\
 -x_1 x_5 + x_2^2, \quad
 x_2 x_6 x_7 - x_4^2 x_8, \quad
 x_2 x_5 x_7 - x_3 x_4 x_8, \quad
 x_1 x_6 x_7^2 - x_4^3 x_8,\\
 x_1 x_5 x_7^2 - x_3 x_4^2 x_8, \quad
 x_1 x_3 x_7^3 - x_4^4 x_8, \quad
 x_1 x_2 x_7^3 - x_3 x_4^3 x_8, \quad
 x_1^2 x_7^4 - x_3 x_4^4 x_8 \rangle
 \end{matrix}
\]
and it can be obtained using the command \texttt{toric\_ideal} which takes as input the matrix $A$.
The affine charts $X_{P \cap M} \cap U_i$ and $X_{P \cap M} \cap U_j$ corresponding to the cones in Figure \ref{fig:localizcone} are computed via \texttt{affine\_open\_covering(X)}. They can be obtained by dehomogenizing with respect to the variable associated to the vertices $m_i$ and $m_j$, which in our notation are $x_3$ and $x_4$ respectively: the vertices $m_i$, $m_j$ correspond to columns 1 and 8 of $A$.
Their overlaps $\left(X_{P \cap M} \cap U_i\right) \cap U_j$ and $\left(X_{P \cap M} \cap U_j\right) \cap U_i$ are isomorphic via the usual map
\begin{equation*}
    \left(1 , \frac{x_2}{x_1} , \frac{x_3}{x_1} ,\frac{x_4}{x_1} , \frac{x_5}{x_1} , \frac{x_6}{x_1} , \frac{x_7}{x_1}, \frac{x_8}{x_1} \right) \overset{\cdot \frac{x_1}{x_8}}{\longmapsto} \left(\frac{x_1}{x_8} , \frac{x_2}{x_8} , \frac{x_3}{x_8} , \frac{x_4}{x_8} , \frac{x_5}{x_8} , \frac{x_6}{x_8} , \frac{x_7}{x_8}, 1 \right).\qedhere
\end{equation*}
\end{examplestar}

\subsection{Projective toric varieties from polytopes}

To a (not necessarily very ample) full dimensional polytope $P \subset M_\R$, we would like to associate the normal projective toric variety coming from its normal fan. If $P$ is not very ample, the first guess $X_{P \cap M}$ for its projective toric variety is not normal. In order to fix this, by Proposition \ref{prop:dilatenormal} we can dilate $P$ such that it becomes normal. By Proposition \ref{prop:fanproperties}, this does not change the normal fan. 

\begin{definition}[The toric variety of a polytope] \label{def:torvarofpol}
To a full dimensional convex lattice polytope $P \subset M_\R$ we associate the projective toric variety $X_P \simeq X_{(k \cdot P) \cap M}$ for any $k$ such that $k \cdot P$ is very ample. 
\end{definition}

As pointed out in the previous section, different choices of $k$ in Definition \ref{def:torvarofpol} correspond to different embeddings of the same projective variety $X_P = X_{k \cdot P}$. 

We conclude with a result on normality and smoothness. The proof is easy and left as an exercise. A fan $\Sigma$ is called \emph{smooth} if all of its cones $\sigma \in \Sigma$ are smooth. A polytope $P$ is \emph{smooth} if its normal fan is smooth.

\begin{proposition}
Let $P$ be a full dimensional lattice polytope in $M_\R$. The projective toric variety $X_P$ is normal. It is smooth if and only if $P$ is smooth, and it is projectively normal if and only if $P$ is normal.
\end{proposition}

\begin{examplestar} \label{ex:chiara7}
The \emph{permutohedron} $\Pi_n$ is a lattice polytope of dimension $n-1$ in $\R^n$ whose vertices are given by the permutations of the first $n$ natural numbers. Note that it is contained in the affine hyperplane 
\begin{equation*}
    \left\{ (x_1,\ldots,x_n)\in\R^n \,|\, x_1 + \ldots + x_n = 1+2+\ldots+n = \frac{n(n+1)}{2}\right\}.
\end{equation*}
The toric variety $X_{\Pi_n}$ associated to $\Pi_n$ is sometimes called a \emph{permutohedral variety}. For instance, the permutohedral curve is $\PP^1$ and the permutohedral surface is a toric surface of degree $6$.
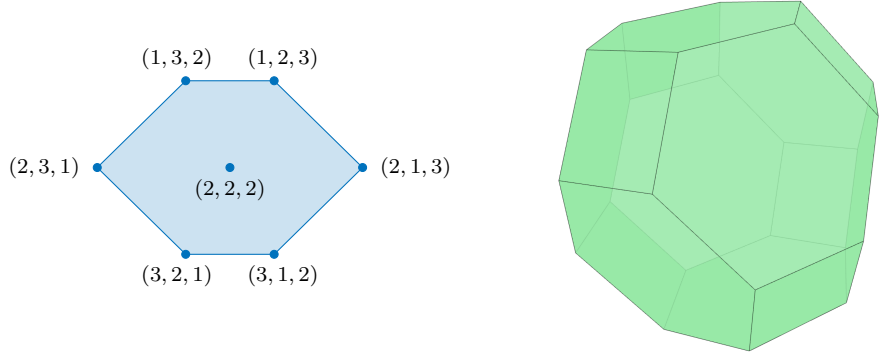
\begin{figure}
    \centering
    \begin{tikzpicture}[scale=1]
\scriptsize
\begin{axis}[%
width=2.5in,
height=2.0in,
scale only axis,
xmin=-1.2,
xmax=4.2,
ymin=-1.2,
ymax=3.2,
ticks = none, 
ticks = none,
axis background/.style={fill=white},
axis line style={draw=none} 
]


\addplot [color=mycolor1,solid,fill opacity=0.2,fill = mycolor1,forget plot]
  table[row sep=crcr]{%
0	1\\
1	0\\	
2   0\\
3   1\\
2	2\\
1   2\\
0	1\\
};

\addplot[only marks,mark=*,mark size=1.5pt,mycolor1,
        ]  coordinates {
    (0,1) (1,0) (2,0) (3,1) (2,2) (1,2) (1.5,1)
};

\node (P) at (axis cs:-0.6,1) {$(2,3,1)$};
\node (P) at (axis cs:0.9,-0.25) {$(3,2,1)$};
\node (P) at (axis cs:2.1,-0.25) {$(3,1,2)$};
\node (P) at (axis cs:3.6,1) {$(2,1,3)$};
\node (P) at (axis cs:2.1,2.25) {$(1,2,3)$};
\node (P) at (axis cs:0.9,2.25) {$(1,3,2)$};
\node (P) at (axis cs:1.5,0.75) {$(2,2,2)$};

\end{axis}
\end{tikzpicture}
    \quad \quad 
    \includegraphics[width=0.27\textwidth]{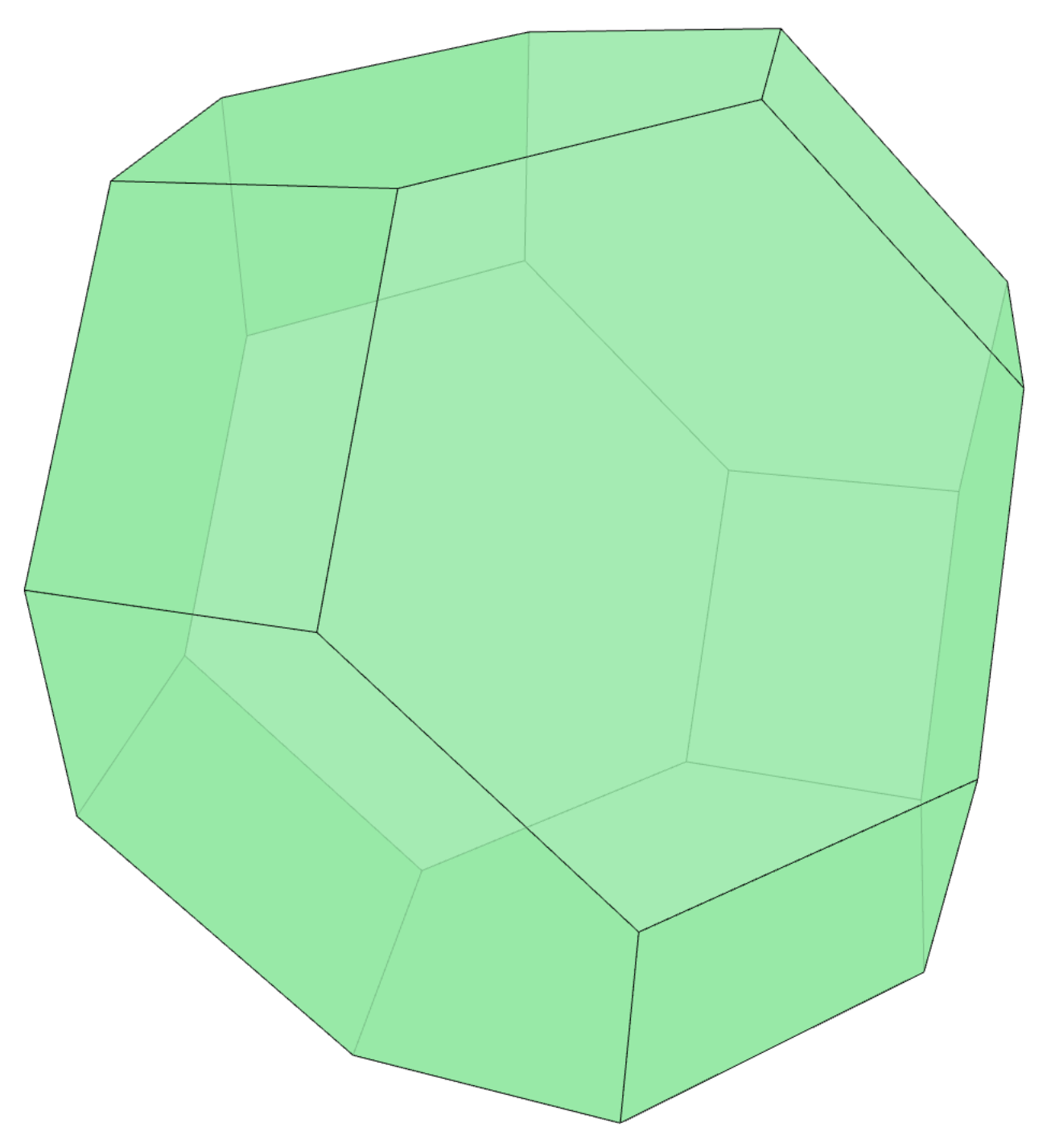}
    \caption{Left: the hexagon $\Pi_3 \subset \{(x_1,x_2,x_3)\in\R^3 \,|\, x_1+x_2+x_3 = 6\}$. The $7$ blue dots are the lattice points contained in $\Pi_3$. Right: the polytope $\Pi_4\subset \{(x_1,x_2,x_3,x_4)\in\R^4 \,|\, x_1+x_2+x_3+x_4 = 10\}$. It has $38$ lattice points, $24 = 4!$ of which are vertices.}
    \label{fig:permutohedral_surface}
\end{figure}
For all $n$, $\Pi_n$ is a smooth, normal polytope, so that $X_{\Pi_n}$ is smooth and projectively normal. We check these properties using \texttt{Oscar.jl} for $n = 4$:
\begin{minted}{julia}
PP = Polyhedron(polytope.permutahedron(3))
P = project_full(PP)
X = NormalToricVariety(P)
isnormal(P)
issmooth(P)
\end{minted}

We compute the Hilbert function of $X_{\Pi_n \cap \Z^n}$. Because the variety is projectively normal, the Hilbert polynomial coincides with the Ehrhart polynomial of $\Pi_n$. We get
\begin{align*}
    E_{\Pi_3}(x) &= 3x^2+3x+1, \\
    E_{\Pi_4}(x) &= 16 x^3 + 15 x^2 + 6 x + 1, \\
    E_{\Pi_5}(x) &= 125 x^4 + 110 x^3 + 45 x^2 + 10 x + 1, \\
    E_{\Pi_6}(x) &= 1296 x^5 + 1080 x^4 + 435 x^3 + 105 x^2 + 15 x + 1, \\
    E_{\Pi_7}(x) &= 16807 x^6 + 13377 x^5 + 5250 x^4 + 1295 x^3 + 210 x^2 + 21 x + 1.
\end{align*}
\texttt{Oscar.jl} can compute this on a MacBook Pro within one minute for $n \leq 7$. The leading coefficient of $E_{\Pi_n}$, which equals the Euclidean volume ${\rm Vol}(E_{\Pi_n})$ by Theorem \ref{thm:ehrhart}, is $n^{n-2}$. This is the number of trees on $n$ labeled nodes \cite[Prop.~2.4]{postnikov2009permutohedra}. 
\end{examplestar}

\begin{example}
Let $P = {\rm Conv}(\{ (0,0), (1,2), (2,1) \})$ be the triangle shown in Figure \ref{fig:singularsurface}. Since all polygons are very ample, its toric variety $X_P$ can be embedded in $\PP^3$. The defining equation, up to reordering the variables, is $x_1^3 - x_0x_2x_3 = 0$. This surface, in the chart $x_0 + x_1 + x_2 + x_3 \neq 0$ is shown on the left side of Figure \ref{fig:singularsurface}. Its three singular points correspond to the vertices of $P$. They are the torus fixed points from Proposition \ref{prop:fixedpoint}.
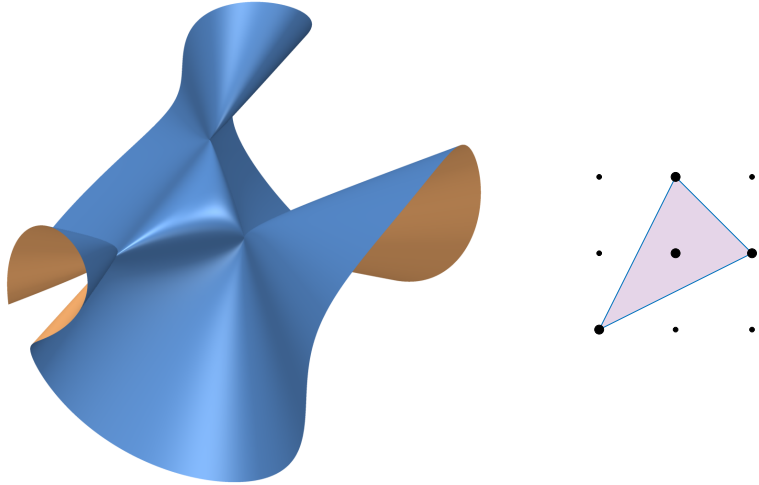
\begin{figure}
\centering
\includegraphics[scale=0.18]{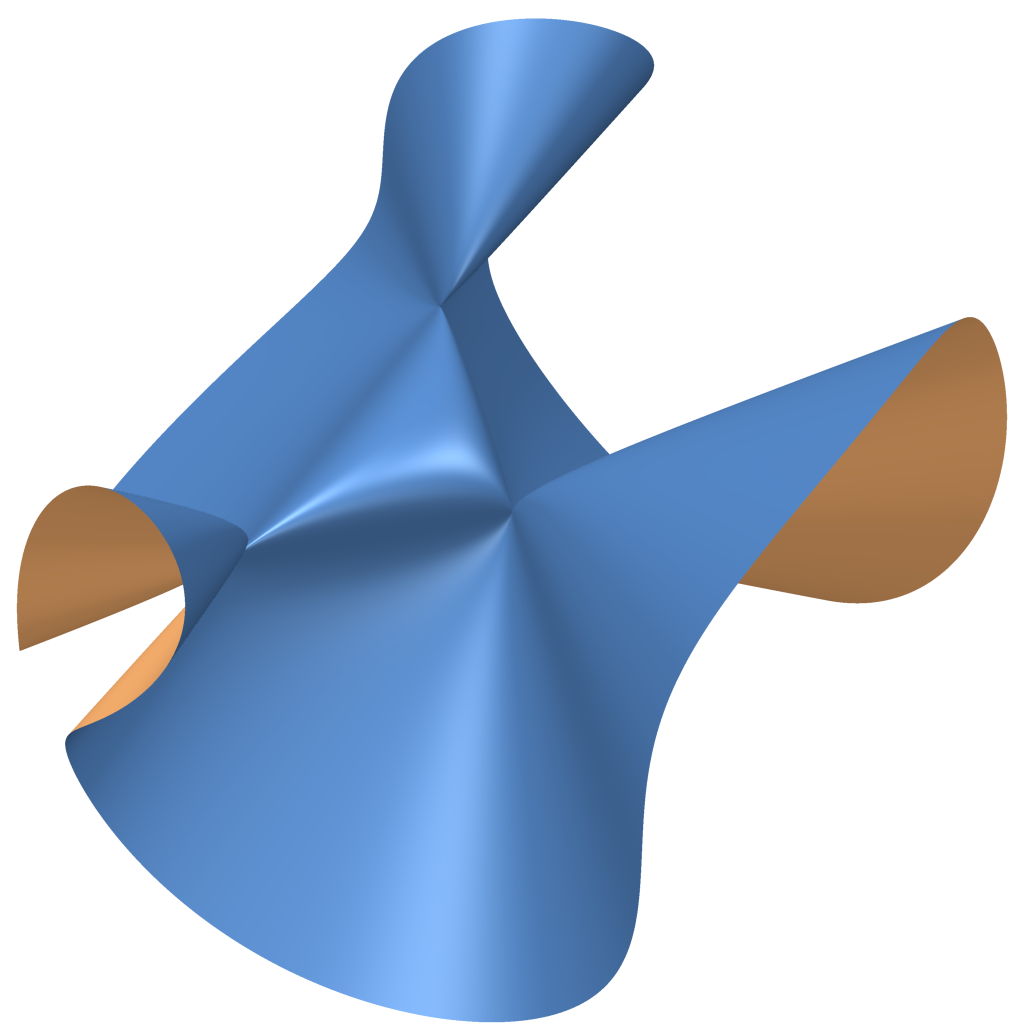}
\qquad 
\begin{tikzpicture}[scale=0.8]
\begin{axis}[%
width=1.5in,
height=3in,
scale only axis,
xmin=-0.5,
xmax=2.5,
ymin=-2,
ymax=4,
ticks = none, 
ticks = none,
axis background/.style={fill=white},
axis line style={draw=none} 
]


\addplot [color=mycolor1,solid,fill opacity=0.2,fill = mycolor4,forget plot]
  table[row sep=crcr]{%
 0 0\\
1 2\\	
2 1\\
0 0 \\
};

\addplot[only marks,mark=*,mark size=2.1pt,black
        ]  coordinates {
  (0,0) (2,1) (1,2) (1,1)
};


\addplot[only marks,mark=*,mark size=1.1pt,black
        ]  coordinates {
   (0,0) (1,0) (2,0) (0,1) (1,1) (2,1) (0,2) (1,2) (2,2)
};

\end{axis}
\end{tikzpicture} 
\caption{The surface $x_1^3-x_0x_2x_3 = 0$ is the toric variety of a singular triangle.}
\label{fig:singularsurface}
\end{figure}
\end{example}

\section{Abstract toric varieties} \label{sec:abstractTV}

In previous sections, we have defined normal toric varieties from cones and polytopes. A more general class of toric varieties comes from fans. To a fan $\Sigma$ we will associate an \emph{abstract toric variety $X_\Sigma$}. The construction generalizes $\sigma \mapsto U_\sigma$ and $P \mapsto \Sigma_P \mapsto X_P$, where the fans are $\sigma$ and all its faces, and the normal fan of $P$ respectively. We recall how abstract varieties are obtained from gluing affine varieties before explaining the construction. Next, we discuss the \emph{orbit-cone correspondence}, which tells us that $X_\Sigma$ is stratified by torus orbits in a way that is nicely encoded by $\Sigma$. Finally, we study toric morphisms in this general setting. 

\subsection{Gluing affine varieties} \label{subsec:gluing}

Much of the material in this section is taken from \cite[Sec.~2.3]{telen2020thesis}. Consider a set $\{Y_i\}_{i \in \J}$ of affine varieties for some index set $\J$, such that for all $i,j \in \J$, we have isomorphic Zariski open subsets $Y_{ij} \subset Y_i$, $Y_{ji} \subset Y_j$. Let $\{\phi_{ij} \}_{i,j \in \J}$ be isomorphisms such that for all $i,j,k \in \J$,
\begin{enumerate}
\item $\phi_{ij}: Y_{ij} \rightarrow Y_{ji}$ and $\phi_{ji}: Y_{ji} \rightarrow Y_{ij}$ satisfy $\phi_{ij} \circ \phi_{ji} = \id_{Y_{ji}}, \phi_{ji} \circ \phi_{ij} = \id_{Y_{ij}}$,
\item $\phi_{ij}(Y_{ij} \cap Y_{ik}) = Y_{ji} \cap Y_{jk}$,
\item $\phi_{ik} = \phi_{jk} \circ \phi_{ij}$ on $Y_{ik} \cap Y_{ij}$.
\end{enumerate} 
The disjoint union $\bigsqcup_{i \in \J} Y_i$ is the set 
$$ \hat{X} = \bigsqcup_{i \in \J} Y_i = \{ (x,Y_i) ~|~ i \in \J, x \in Y_i \}.$$
It is a topological space with the disjoint union topology, which is such that the open subsets of $\hat{X}$ are disjoint unions of open subsets in the $Y_i$. We define an equivalence relation $\sim$ on $\hat{X}$ by setting $(x,Y_i) \sim (y,Y_j)$ if $x \in Y_{ij}$, $y \in Y_{ji}$ and $\phi_{ij}(x) = y$. The first condition on the $\phi_{ij}$ makes $\sim$ reflexive and symmetric, the second and third conditions make it transitive. We consider the quotient space $X = \hat{X}/\sim$ with its quotient topology, called the \emph{Zariski topology} on $X$. In this topological space, 
$$U_i = \{[(x,Y_i)] ~|~ x \in Y_i \} \subset X$$ 
are open subsets isomorphic to $Y_i$ (here we write $[\cdot ]$ for an equivalence class in the quotient). The space $X$ is called an \emph{abstract variety}. The affine varieties $\{Y_i \}_{i \in \J}$ and the isomorphisms $\{\phi_{ij}\}_{i,j \in \J}$ are called the \emph{gluing data} for the construction of $X$. 

\begin{example}[Gluing of $\PP^1$] \label{ex:glueP1}
The projective line $\PP^1$ is covered by $\PP^1 = U_x \cup U_y$ where
$$ U_x = \{ (x:y) \in \PP^1 ~|~ x \neq 0 \}, \quad  U_y = \{ (x:y) \in \PP^1 ~|~ y \neq 0 \}.$$
Consider the isomorphisms 
$$ h_x : U_x \rightarrow \C_t \quad \textup{and} \quad h_y: U_y \rightarrow \C_u,$$ 
where $\C_t$ is $\C$ with coordinate $t$ and analogously for $u$, given by $h_x(x:y) = y/x$ and $h_y(x:y) = x/y$. For a point $(x:y) \in U_x \cap U_y$, we have $h_x(x:y) = h_y(x:y)^{-1}$. Let 
$$\C_{tu} = \C_t^* = \C_t \setminus \{0\}, \quad \C_{ut} = \C_u^* = \C_u \setminus \{0\}$$
and $\phi_{tu}: \C_{tu} \rightarrow \C_{ut}$ given by $\phi_{tu}(t) = t^{-1}$, $\phi_{ut} = \phi_{tu}^{-1}$. This gives a commutative diagram
\[
\begin{tikzcd}[row sep = 1.5cm, column sep = 1.5cm]
U_x \cap U_y \arrow{r}{h_x} \arrow{d}{h_y} & \C_{tu}  \arrow[rightharpoonup]{dl}{\phi_{tu}}\\
\C_{ut} \arrow[rightharpoonup, shift left = 1.5]{ur}{\phi_{ut}} & 
\end{tikzcd} 
\]
The projective line $\PP^1$ is a gluing of two copies of $\C$ with gluing data $\{\C_t,\C_u \}$ and $\{ \phi_{tu},\phi_{ut} \}$. The two affine lines $\C_t$ and $\C_u$ are glued together along the open subsets $\C_t^*$ and $\C_u^*$, which give the open subset $U_x \cap U_y \subset \PP^1$. The missing points $\PP^1 \setminus (U_x \cap U_y) = \{(1:0), (0:1) \}$ correspond to the origin in $\C_t$ and $\C_u$. If we consider $\PP^1$ as the compactification of $\C_t$, the \emph{point at infinity} corresponds to the origin in $\C_u$. This is illustrated in Figure \ref{fig:P1}.
\begin{figure}
\centering
\input{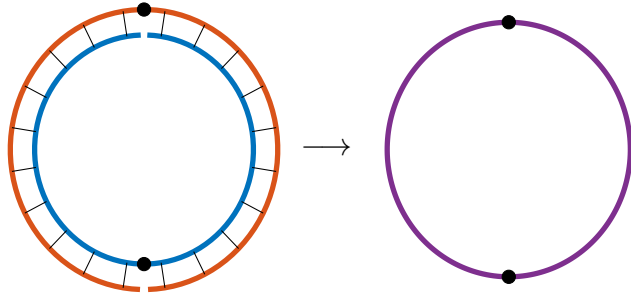}
\caption{Illustration of the construction of $\PP^1$ as the gluing of two affine lines. The affine lines are represented as circles with a missing point (`at infinity'). The origin in each line is indicated with a black dot and the gluing isomorphism is illustrated by black line segments.}
\label{fig:P1}
\end{figure}
\end{example}
\begin{example}[A non-separated variant] \label{ex:nonsep}
We replace the isomorphisms in the glueing data from Example \ref{ex:glueP1} by $\phi_{tu}(t) = t$ and $\phi_{ut} (u) = u$. This way, we obtain an abstract variety $X$ that, like $\PP^1$, is a union of $\C^*$ and two points. This is illustrated in Figure \ref{fig:nonsep}. However, this variety is not \emph{separated}, meaning that the \emph{classical} topology on $X$ is not Hausdorff. We will see that such situations never occur in our gluing construction for toric varieties. 
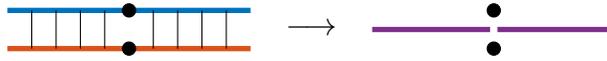
\begin{figure}
\centering
%
%
\definecolor{mycolor1}{rgb}{0.00000,0.44700,0.74100}%
\definecolor{mycolor2}{rgb}{0.85000,0.32500,0.09800}%
\begin{tikzpicture}

\begin{axis}[%
width=3.5in,
height=0.5in,
at={(0.49in,0.343in)},
scale only axis,
xmin=-1.5,
xmax=4.0,
xtick = \empty,
ymin=-0.5,
ymax=0.5,
ytick = \empty,
axis background/.style={fill=white},
axis line style={draw=none},
]
\addplot [color=mycolor1, line width=2.0pt, forget plot]
  table[row sep=crcr]{%
-1 0.2 \\
1 0.2 \\
};
\addplot [color=mycolor2, line width=2.0pt, forget plot]
  table[row sep=crcr]{%
-1 -0.2\\
1 -0.2 \\
};

\addplot [color=mycolor4, line width=2.0pt, forget plot]
  table[row sep=crcr]{%
2 0 0\\
2.97  0\\
};

\addplot [color=mycolor4, line width=2.0pt, forget plot]
  table[row sep=crcr]{%
3.03 0\\
4	0 \\
};

\addplot [color=black, forget plot]
  table[row sep=crcr]{%
-0.8 0.2\\
-0.8 -0.2\\
};
\addplot [color=black, forget plot]
  table[row sep=crcr]{%
-0.6 0.2\\
-0.6 -0.2\\
};
\addplot [color=black, forget plot]
  table[row sep=crcr]{%
-0.4 0.2\\
-0.4 -0.2\\
};
\addplot [color=black, forget plot]
  table[row sep=crcr]{%
-0.2 0.2\\
-0.2 -0.2\\
};
\addplot [color=black, forget plot]
  table[row sep=crcr]{%
0.2 0.2\\
0.2 -0.2\\
};
\addplot [color=black, forget plot]
  table[row sep=crcr]{%
0.4 0.2\\
0.4 -0.2\\
};
\addplot [color=black, forget plot]
  table[row sep=crcr]{%
0.6 0.2\\
0.6 -0.2\\
};
\addplot [color=black, forget plot]
  table[row sep=crcr]{%
0.8 0.2\\
0.8 -0.2\\
};

\addplot [color=black, draw=none, mark size=2.5pt, mark=*, mark options={solid, black}, forget plot]
  table[row sep=crcr]{%
0	0.2\\
};
\addplot [color=black, draw=none, mark size=2.5pt, mark=*, mark options={solid, black}, forget plot]
  table[row sep=crcr]{%
0	-0.2\\
};

\addplot [color=black, draw=none, mark size=2.5pt, mark=*, mark options={solid, black}, forget plot]
  table[row sep=crcr]{%
3	0.2\\
};

\addplot [color=black, draw=none, mark size=2.5pt, mark=*, mark options={solid, black}, forget plot]
  table[row sep=crcr]{%
3	-0.2\\
};

\node (P) at (axis cs:1.5,0) {$\longrightarrow$};

\end{axis}
\end{tikzpicture}%
\caption{A non-separated variety obtained from gluing two affine lines.}
\label{fig:nonsep}
\end{figure}
\end{example}
\begin{example}[Gluing of $\PP^2$] \label{ex:glueP2}
We glue $\PP^2$ from three copies of $\C^2$. Consider the isomorphisms
$$ h_x : U_x \rightarrow \C^2_t ,~ h_y: U_y \rightarrow \C^2_u,~ \textup{and} ~ h_z: U_z \rightarrow \C^2_v $$
where $\C^2_t$ is the affine plane with coordinates $t_1,t_2$ (analogously for $u,v$) and 
$$h_x(x:y:z) = (y/x,z/x),~ h_y(x:y:z) = (x/y, z/y),~ h_z(x:y:z) = (x/z,y/z).$$ 
The gluing morphisms $\phi_{tv} = \phi_{vt}^{-1}$ come from identifying the images of points in $U_x \cap U_z$ under $h_x$ and $h_z$, e.g.\ on $\C^2_{tv} = \C_t^2 \setminus V(t_2)$ 
$$\phi_{tv}(t_1,t_2)= (t_2^{-1}, t_1t_2^{-1}) \quad \textup{comes from} \quad \left ( \frac{x}{z}, \frac{y}{z} \right ) = \left ( \left ( \frac{z}{x} \right )^{-1}, \left ( \frac{y}{x} \right ) \left ( \frac{z}{x} \right )^{-1} \right ). \qedhere$$ 
\end{example}
Let $X_1 = \cup_{i \in \J} U_i$ and $X_2 = \cup_{j \in \J'} U_j'$ be abstract varieties. A \emph{morphism} $\Phi: X_1 \rightarrow X_2$ is a Zariski continuous map such that 
\[\Phi_{|U_i \cap \Phi^{-1}(U_j')}: U_i \cap \Phi^{-1}(U_j') \rightarrow U_j' \]
is a morphism of Zariski open subsets of affine varieties. We note that the product $X_1 \times X_2$ is an abstract variety as well, obtained by glueing $U_i \times U_j'$ in the appropriate way. An abstract variety $X$ is \emph{irreducible} if it cannot be written as a union of two proper closed subvarieties. It is called \emph{normal} if each of the affine varieties $Y_i$ is normal. 

\subsection{Toric varieties from fans}
We are now ready to state our third, last and most general definition of a toric variety.
\begin{definition}[Toric variety] \label{def:TV}
A \emph{toric variety} is an irreducible abstract variety $X$ containing a torus $T \simeq (\C^*)^n$ as a Zariski open subset, such that the action of $T$ on itself extends to an algebraic action $T \times X \rightarrow X$ on $X$.
\end{definition}
Here, by an \emph{algebraic} action we mean that $T \times X \rightarrow X$ is a morphism of abstract varieties. Note that affine toric varieties and projective toric varieties are toric varieties in this sense. 
 
 Recall that a \emph{fan} $\Sigma$ in $N_\R$ is a collection of cones that fit together in a nice way (Definition \ref{def:fan}). We fix such a fan $\Sigma$ in $N_\R$. We now explain how this encodes the gluing data for constructing a toric variety $X_\Sigma$. Recall that each $m \in M_\R$ gives a hyperplane $H_m = \{ u \in N_\R ~|~ \langle u, m \rangle = 0 \}$. Here are two facts we will need in the construction: 
 \begin{enumerate}
 \item[(i)] \emph{Faces of $\sigma$ give affine open subsets of $U_\sigma$.} If $\tau \preceq \sigma$, $U_\tau \simeq (U_\sigma)_{\chi^m} \subset U_\sigma$, where $m \in M$ is such that $\tau = H_m \cap \sigma$. See Example \ref{ex:affineopensubsets}.
 \item[(ii)] \emph{Common faces give common affine open subsets.} Let $\tau = \sigma_1 \cap \sigma_2$. There exists $m \in \sigma_1^\vee \cap (- \sigma_2^\vee) \cap M$ such that $\sigma_1 \cap H_m = \tau = \sigma_2 \cap H_m$. It follows that
 \[ U_{\sigma_1} \supseteq (U_{\sigma_1})_{\chi^m} \simeq U_\tau \simeq (U_{\sigma_2})_{\chi^{-m}} \subset U_{\sigma_2}. \]
 \end{enumerate}
 The second item is \cite[Lem.~1.2.13]{cox2011toric}. The gluing data for $X_\Sigma$ are
 \begin{enumerate}
 \item the collection of normal affine toric varieties $\{ U_\sigma \}_{\sigma \in \Sigma}$ and
 \item for each pair of cones $\sigma_1, \sigma_2 \in \Sigma$, an isomorphism 
 \[ \phi_{\sigma_1,\sigma_2} : U_{\sigma_1,\sigma_2} = (U_{\sigma_1})_{\chi^m} \rightarrow U_{\sigma_2,\sigma_1} = (U_{\sigma_2})_{\chi^{-m}} \quad \text{given by} \quad (U_{\sigma_1})_{\chi^m} \simeq U_\tau \simeq (U_{\sigma_2})_{\chi^{-m}}. \]
 Here $\tau = \sigma_1 \cap \sigma_2$ and $m$ is as in item (ii) above.
\end{enumerate}  

\begin{theorem}
Let $\Sigma$ be a fan in $N_\R$. The abstract variety $X_\Sigma$ obtained from the above gluing data is a normal toric variety. 
\end{theorem}
\begin{proof}
By Definition \ref{def:fan}, each $\sigma \in \Sigma$ is strongly convex. Therefore, $\{0\}$ is a face of each cone, and hence $T \simeq (\C^*)^n = \Specm(\C[M]) = U_{\{0\}} \subset U_\sigma$ for each $\sigma \in \Sigma$. These tori are identified by the gluing, so that $X_\Sigma$ is irreducible and $T \subset X_\Sigma$ a Zariski open subset. The actions $T \times U_\sigma \rightarrow U_\sigma$ are compatible on overlaps (this is easily seen from the intrinsic description of the torus action from Proposition \ref{prop:torusaction}), and therefore glue to a morphism $T \times X_\Sigma \rightarrow X_\Sigma$. Normality follows from Theorem \ref{thm:normal}.
\end{proof}

\begin{remark} \label{rem:maxcones}
Since $U_\tau \subset U_\sigma$ when $\tau \preceq \sigma$, we have that $X_\Sigma = \bigcup_{\sigma \in \Sigma} U_\sigma = \bigcup_{\sigma \text{ maximal}} U_\sigma$ is covered by the \emph{maximal} cones in $\Sigma$ (with respect to inclusion).
\end{remark}
We state, without proof, that this construction never leads to non-separated abstract varieties, such as $X$ from Example \ref{ex:nonsep}. See \cite[Exer.~3.1.2]{cox2011toric}. Moreover, it turns out every normal, separated toric variety arises in this way \cite[Cor.~3.1.8]{cox2011toric}. 
\begin{theorem} \label{thm:classification}
Let $X$ be a normal, separated toric variety with torus $T = N \otimes_\Z \C^*$. There exists a fan $\Sigma$ in $N_\R$ such that $X \simeq X_\Sigma$.
\end{theorem}
\begin{example}
If $\sigma \subset N_\R$ is a strongly convex rational polyhedral cone and $\Sigma$ is the fan consisting of $\sigma$ and all its faces, we have $X_\Sigma = U_\sigma$.
\end{example}
\begin{example}
Let $P$ be a full-dimensional polytope in $M_\R$ and let $\Sigma_P$ be its normal fan. Then $X_P \simeq X_{\Sigma_P}$. 
\end{example}
\begin{example} \label{ex:1D}
If $N = \Z$, there are only three cones in $N_\R$: $\tau = \{0\}, \sigma_1 = \R_{\leq 0}, \sigma_2= \R_{\geq 0}$. By Theorem \ref{thm:classification}, there are only three one-dimensional normal separated toric varieties: 
\begin{enumerate}
\item $\Sigma = \{\tau\}$, $X_\Sigma = \C^*$, 
\item $\Sigma = \{\sigma_i, \tau \}, i = 1 \text{ or } 2$, $X_\Sigma = \C$.
\item $\Sigma = \{\tau, \sigma_1, \sigma_2\}$, $X_\Sigma =\PP^1$. 
\end{enumerate}
In the latter case, $\Sigma$ is the normal fan of a line segment $P = {\rm Conv}(0,1)$, and $ \PP^1 \simeq X_P$. 
\end{example}
\begin{example} \label{ex:coveringP2}
We revisit the gluing of $\PP^2$ from a toric point of view. Consider the fan shown in Figure \ref{fig:fanP2}. 
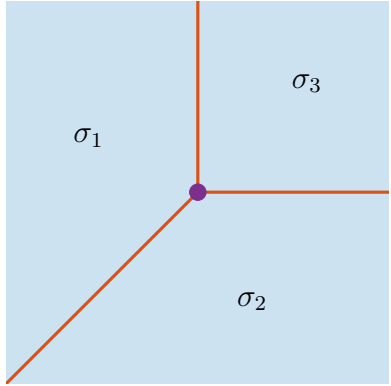
\begin{figure}[h!]
\centering
\begin{tikzpicture}[scale=1]
\begin{axis}[%
width=2in,
height=2in,
scale only axis,
xmin=-3.5,
xmax=3.5,
ymin=-3.5,
ymax=3.5,
ticks = none, 
ticks = none,
axis background/.style={fill=white},
axis line style={draw=none} 
]


\addplot [color=mycolor1,solid,fill opacity=0.2,fill = mycolor1,forget plot]
  table[row sep=crcr]{%
 5 0\\
5 5\\	
0 5\\
0 0 \\
5 0 \\
};

\addplot [color=mycolor1,solid,fill opacity=0.2,fill = mycolor1,forget plot]
  table[row sep=crcr]{%
 0 0 \\
 -5 -5 \\
 5 -5\\
 5 0\\
 0 0\\
};

\addplot [color=mycolor1,solid,fill opacity=0.2,fill = mycolor1,forget plot]
  table[row sep=crcr]{%
 0 0 \\
 -5 -5 \\
 -5 5\\
 0 5\\
 0 0\\
};

\addplot [very thick, color=mycolor2,solid,fill opacity=0.2,fill = mycolor1,forget plot]
  table[row sep=crcr]{%
 0 0 \\
 5 0 \\
};

\addplot [very thick, color=mycolor2,solid,fill opacity=0.2,fill = mycolor1,forget plot]
  table[row sep=crcr]{%
 0 0 \\
 -5 -5 \\
};

\addplot [very thick, color=mycolor2,solid,fill opacity=0.2,fill = mycolor1,forget plot]
  table[row sep=crcr]{%
 0 0 \\
 0 5 \\
};

\addplot[only marks,mark=*,mark size=3.1pt,mycolor4
        ]  coordinates {
  (0,0)
};

\node (P) at (axis cs:2,2) {$\sigma_3$};
\node (P) at (axis cs:1,-2) {$\sigma_2$};
\node (P) at (axis cs:-2,1) {$\sigma_1$};

\end{axis}
\end{tikzpicture} 
\caption{The fan of $\PP^2$.}
\label{fig:fanP2}
\end{figure}
Let $\sigma_1, \sigma_2, \sigma_3$ be the 2-dimensional cones of $\Sigma$. We have 
\[ U_{\sigma_3} = \Specm(\C[t_1,t_2]), \quad U_{\sigma_1} = \Specm(\C[t_1^{-1},t_1^{-1}t_2]), \quad U_{\sigma_2} = \Specm(\C[t_2^{-1},t_1t_2^{-1}]). \]
This is most easily seen from the dual cones in Figure \ref{fig:illustratefan}. Each of these is isomorphic to $\C^2$. Let $\tau = \sigma_3 \cap \sigma_2$. We have $U_\tau  = \Specm (\C[t_1,t_2,t_2^{-1}]) \simeq  \Specm(\C[t_1,t_2]_{t_2}) \simeq \Specm(\C[t_2^{-1},t_1t_2^{-1}]_{t_2^{-1}})$. Identifying  $\C[t_2^{-1},t_1t_2^{-1}]_{t_2^{-1}}$ with $\C[v_1,v_2]_{v_1}$, induces an isomorphism of $\C$-algebras $\C[v_1,v_2]_{v_1} \rightarrow \C[t_1,t_2]_{t_2} $ sending $v_1 \mapsto t_2^{-1}$ and $v_2 \mapsto t_1 t_2^{-1}$. This is the pullback of $\phi_{tv}$ in Example \ref{ex:glueP2}. The gluing of the other charts is similar. We conclude that $X_\Sigma = \PP^2$. Note that $\Sigma$ is the normal fan of the standard simplex $\Delta_2 = {\rm Conv}( 0, e_1, e_2 ) \subset \R^2$. Hence $X_{\Delta_2} \simeq \PP^2$, which is obvious from $X_{\Delta_2} \simeq X_{\Delta_2 \cap \Z^2}$ (Definition \ref{def:torvarofpol}).
\end{example}
\begin{example} \label{ex:blowupC2}
We consider the fan $\Sigma$ from Figure \ref{fig:fanBl0C2}. 
Note that it is not the normal fan of a polytope. By Remark \ref{rem:maxcones}, $X_\Sigma$ is covered by $U_{\sigma_1} \simeq \{ w - uv = 0 \} \subset \C^3, \quad U_{\sigma_2} \simeq \{ z - xy = 0\} \subset \C^3$
These overlap in $X_\Sigma$ on $(U_{\sigma_1})_v$  and $(U_{\sigma_2})_x$. The isomorphism comes from 
\[ \left( \C[x,y,z]/\langle z - xy \rangle \right)_{[x]} \rightarrow \left( \C[u,v,w]/\langle w - uv \rangle \right)_{[v]} \]
given by $x \mapsto v^{-1}, y \mapsto w, z \mapsto u$. To see to compute this morphism in coordinates, see Example \ref{ex:chiaramara} below. We conclude that $X_\Sigma = \{x_0 z - x_1 y  = 0 \}  = \text{Bl}_0(\C^2)\subset \PP^1 \times \C^2$. 
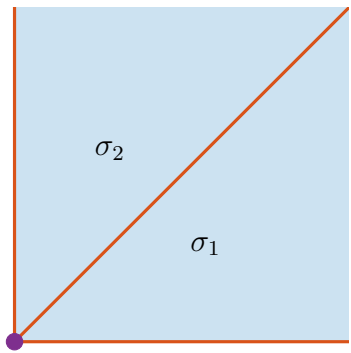
\begin{figure}[h!]
\centering
\begin{tikzpicture}[scale=1]
\begin{axis}[%
width=2in,
height=2in,
scale only axis,
xmin=-0.5,
xmax=3.5,
ymin=-0.5,
ymax=3.5,
ticks = none, 
ticks = none,
axis background/.style={fill=white},
axis line style={draw=none} 
]


\addplot [color=mycolor1,solid,fill opacity=0.2,fill = mycolor1,forget plot]
  table[row sep=crcr]{%
 5 0\\
5 5\\	
0 0 \\
5 0 \\
};

\addplot [color=mycolor1,solid,fill opacity=0.2,fill = mycolor1,forget plot]
  table[row sep=crcr]{%
0 5\\
5 5\\	
0 0 \\
0 5\\
};

\addplot [very thick, color=mycolor2,solid,fill opacity=0.2,fill = mycolor1,forget plot]
  table[row sep=crcr]{%
 0 0 \\
 5 0 \\
};

\addplot [very thick, color=mycolor2,solid,fill opacity=0.2,fill = mycolor1,forget plot]
  table[row sep=crcr]{%
 0 0 \\
 0 5 \\
};

\addplot [very thick, color=mycolor2,solid,fill opacity=0.2,fill = mycolor1,forget plot]
  table[row sep=crcr]{%
 0 0 \\
 5 5 \\
};

\addplot[only marks,mark=*,mark size=3.1pt,mycolor4
        ]  coordinates {
  (0,0)
};

\node (P) at (axis cs:2,1) {$\sigma_1$};
\node (P) at (axis cs:1,2) {$\sigma_2$};

\end{axis}
\end{tikzpicture} 
\caption{The fan of ${\rm Bl}_0(\C^2)$.}
\label{fig:fanBl0C2}
\end{figure}
\end{example}
\begin{exercise} \label{ex:P1P1}
Show that $\PP^1 \times \PP^1$ is the toric surface $X_\Sigma$ where $\Sigma$ is the normal fan of the square $[0,1] \times [0,1] \subset \R^2$, shown in Figure \ref{fig:fanP1P1}.. 
\end{exercise}
The following result says that the construction of a toric variety from a fan behaves well under taking products. Here is the precise statement \cite[Prop.~3.1.14]{cox2011toric}. 

\begin{proposition} \label{prop:product}
Let $\Sigma_i$ be a fan in $(N_i)_\R$, $i = 1, 2$. The set of cones
\[ \Sigma_1 \times \Sigma_2 = \{ \sigma_1 \times \sigma_2 ~|~ \sigma_i \in \Sigma_i \} \]
is a fan in $(N_1)_\R \times (N_2)_\R = (N_1 \times N_2)_\R$ and $X_{\Sigma_1 \times \Sigma_2} \simeq X_{\Sigma_1} \times X_{\Sigma_2}$. 
\end{proposition}
\begin{corollary}
Let $P_1 \subset (M_1)_\R$ and $P_2 \subset (M_2)_\R$ be full-dimensional polytopes and let $P_1 \times P_2$ be the product in $(M_1 \times M_2)_\R$. We have $X_{P_1 \times P_2} \simeq X_{P_1} \times X_{P_2}$. 
\end{corollary}
\begin{proof}
This follows from Proposition \ref{prop:product} and the observation that the normal fans satisfy $\Sigma_{P_1 \times P_2} = \Sigma_{P_1} \times \Sigma_{P_2}$.
\end{proof}
\begin{example}
We have seen that $\PP^1$ is the toric variety of a line segment (Example \ref{ex:1D}) and $\PP^1 \times \PP^1$ comes from a square (Exercise \ref{ex:P1P1}), i.e.~the product of two line segments. 
\end{example}
We now discuss some properties of $X_\Sigma$ that can be read off the fan $\Sigma$. We repeat that a fan $\Sigma$ in $N_\R$ is called \emph{smooth} if all of its cones are smooth. It is called \emph{simplicial} if all of its cones are simplicial, and $\Sigma$ is \emph{complete} if its \emph{support} $|\Sigma| = \bigcup_{\sigma \in \Sigma} \sigma$ equals $N_\R$.
\begin{theorem} \label{thm:smoothcompactsimplicial}
Let $\Sigma$ be a fan in $N_\R$. The toric variety $X_\Sigma$ is 
\begin{enumerate}
\item smooth if and only if $\Sigma$ is smooth,
\item compact in the classical topology if and only if $\Sigma$ is complete,
\item an \emph{orbifold} (i.e.~it has only finite quotient singularities) if and only if $\Sigma$ is simplicial.
\end{enumerate}
\end{theorem}
\begin{proof}
This is Theorem 3.1.19 in \cite{cox2011toric}. We point out that item 1 is easy to prove, since smoothness is defined locally. As a hint for point 3, let $\sigma \subset N_\R$ be a simplicial, full-dimensional cone with minimal ray generators $u_1, \ldots, u_n$. The lattice $N_1 = \Z \cdot \{ u_1, \ldots, u_n \}$ has finite index in $N$. As in Example \ref{ex:sublatticesfiniteindex}, it follows that $U_{\sigma,N}$ is the quotient of the smooth affine toric variety $U_{\sigma,N_1}$ by the action of a finite group. Hence, if $\Sigma$ is simplicial, $X_\Sigma$ locally looks like the quotient of a smooth variety by the action of a finite group. This means that it has only finite quotient singularities, so it is an orbifold.
\end{proof}
A toric variety satisfying condition 3 in Theorem \ref{thm:smoothcompactsimplicial} is called \emph{simplicial}. This will be important in our discussion of quotient constructions later. 

\begin{examplestar} \label{ex:chiaramara}
Consider the fan $\Sigma$ in $\R^2$ whose rays have generators $u_1 = (1,2), u_2 = (1,0), u_3 = (-3,-2), u_4 = (0,1)$. The surface $X_\Sigma$ can be embedded in $\PP^{88}$ via the $89$ lattice points of the polygon $P = {\rm Conv} ((0, 15), (0, 1), (2, 0), (10, 0))$ with normal fan $\Sigma$.
\begin{minted}{julia}
P = convex_hull([0 15; 0 1; 2 0; 10 0]); Σ = normal_fan(P);
X_Σ = NormalToricVariety(Σ)
\end{minted}
It is covered by 4 normal affine toric surfaces $U_{\sigma_{12}}, U_{\sigma_{23}}, U_{\sigma_{34}}, U_{\sigma_{14}}$ where $\sigma_{ij} = {\rm Cone}( u_i, u_j )$. Out of these, only $U_{14}$ is smooth. Here is how to check this in \texttt{Oscar.jl}. 
\begin{minted}{julia}
cover = affine_open_covering(X_Σ); 
[issmooth(U) for U in cover]
\end{minted}
This returns a boolean vector \texttt{[0;0;1;0]}. We embed $U_{\sigma_{12}}$ and $U_{\sigma_{23}}$ in affine space and determine $\phi_{\sigma_{12},\sigma_{23}}$ in coordinates. With \texttt{Oscar.jl} we compute the \texttt{hilbert\_basis} of the two cones $\sigma_{12}^\vee, \sigma_{23}^\vee$. These provide the embeddings
\[ U_{\sigma_{12}} \simeq Y_{12} = \{ xz - y^2 = 0 \} \subset \C^3, \qquad U_{\sigma_{23}}  \simeq Y_{23} = \{ uw - v^2 = 0 \} \subset \C^3. \]
The coordinates $x,y,z$ correspond to the blue marked lattice points in Figure \ref{fig:ex_11}, ordered from left to right. Similarly, $u,v,w$ correspond to the green squares. 
The overlap $U_{12}\cap U_{23} \subset X_\Sigma$ is given by points on $Y_{12}$ with $x \neq 0$, and points on $Y_{23}$ with $u \neq 0$.  On these open sets, 
\begin{equation*}
    \phi_{\sigma_{12},\sigma_{23}}(x,y,z) =  \left( \frac{1}{x}, \frac{y}{x^2}, \frac{y^2}{x^3} \right) = \left( \frac{1}{x}, \frac{y}{x^2}, \frac{z}{x^2} \right).
\end{equation*}
This map can be obtained from the cones $\sigma_{12}^\vee, \sigma_{23}^\vee$, writing the green lattice points in Figure \ref{fig:ex_11} as $\Z$-linear combinations of the blue ones.
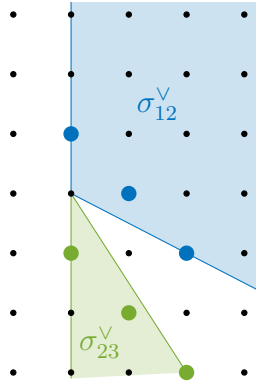
\begin{figure}[h!]
    \centering
    \begin{tikzpicture}[scale=1]
\begin{axis}[%
width=1.33in,
height=2in,
scale only axis,
xmin=-1.2,
xmax=3.2,
ymin=-3.2,
ymax=3.2,
ticks = none, 
ticks = none,
axis background/.style={fill=white},
axis x line = middle, 
axis y line = middle, 
axis line style={draw=none} 
]

\addplot [color=mycolor1,solid,fill opacity=0.2,fill = mycolor1,forget plot]
  table[row sep=crcr]{%
4   -2\\
0   0\\
0   4\\
4	4 \\
};

\addplot [color=mycolor5,solid,fill opacity=0.2,fill = mycolor5,forget plot]
  table[row sep=crcr]{%
2   -3\\
0   0\\
0   -3.1\\
};

\addplot[only marks,mark=*,mark size=1.0pt,black,
        ]  coordinates {
    (0,1) (0,2) (0,3)
    (1,0) (1,1) (1,2) (1,3)
    (2,-1) (2,0) (2,1) (2,2) (2,3)
    (3,-1) (3,0) (3,1) (3,2) (3,3)
};

\addplot[only marks,mark=*,mark size=2.5pt,thick,mycolor1,
        ]  coordinates {
    (0,1) (1,0) (2,-1)
};

\addplot[only marks,mark=*,mark size=1.0pt,black,
        ]  coordinates {
    (0,-3) (0,-2) (0,-1)
    (1,-3) (1,-2) 
    (2,-3)
};

\addplot[only marks,mark=*,mark size=2.5pt,thick,mycolor5,
        ]  coordinates {
    (0,-1) (1,-2) (2,-3)
};

\addplot[only marks,mark=*,mark size=1.0pt,black,
        ]  coordinates {
    (-1,-3) (-1,-2) (-1,-1) (-1,0) (-1,1) (-1,2) (-1,3)
    (1,-1) 
    (2,-2)
    (3,-3) (3,-2) (0,0)
};

\node[mycolor1] (P) at (axis cs:1.5,1.5) {$\sigma_{12}^\vee$};
\node[mycolor5] (P) at (axis cs:0.5,-2.5) {$\sigma_{23}^\vee$};

\end{axis}
\end{tikzpicture}
    \caption{The Hilbert bases of $\sigma_{12}^\vee$ and $\sigma_{23}^\vee$ give the isomorphism $\phi_{\sigma_{12},\sigma_{23}}$.}
    \label{fig:ex_11}
\end{figure}
\end{examplestar}

\subsection{Limits of one-parameter subgroups}

If $X$ is a separated, normal abstract toric variety with torus $T = N \otimes_\Z \C^*$, Theorem \ref{thm:classification} implies that $X \simeq X_\Sigma$ for some polyhedral fan $\Sigma$ in $N_\R$. In this section we will see that this fan can be recovered from limits of monomial curves $\lim_{t \rightarrow 0} \lambda^u(t)$. We start with an example.

\begin{example} \label{ex:limitsP1P1}
Let $X = \PP^1 \times \PP^1$. We have seen that $X = X_\Sigma$ where $\Sigma$ is the fan in Figure \ref{fig:fanP1P1} (Exercise \ref{ex:P1P1}). 
Its torus is $T = \{(1:t_1) \times (1:t_2) ~|~ (t_1,t_2) \in (\C^*)^2 \} \simeq (\C^*)^2$, with cocharacter lattice $N = \Z^2$. We fix $u = (a,b) \in N$ and consider the curves $\lambda^u(t) = (1:t^a) \times (1:t^b) \in \PP^1 \times \PP^1$. The limit $\lim_{t \rightarrow 0} \lambda^u(t)$ depends on $u$. There are nine possibilities: 
\begin{enumerate}
\item $\lim_{t \rightarrow 0} \lambda^u(t) = (1:1) \times (1:1)$ when $a = 0, b = 0$,
\item $\lim_{t \rightarrow 0} \lambda^u(t) = (1:0) \times (1:1)$ when $a>0,b = 0$,
\item $\lim_{t \rightarrow 0} \lambda^u(t) = (1:1) \times (1:0)$ when $a = 0, b >0$,
\item $\lim_{t \rightarrow 0} \lambda^u(t) = (0:1) \times (1:1)$ when $a < 0, b = 0$,
\item $\lim_{t \rightarrow 0} \lambda^u(t) = (1:1) \times (0:1)$ when $a = 0, b < 0$,
\item $\lim_{t \rightarrow 0} \lambda^u(t) = (1:0) \times (1:0)$ when $a >0 , b > 0$,
\item $\lim_{t \rightarrow 0} \lambda^u(t) = (0:1) \times (1:0)$ when $a < 0, b > 0$,
\item $\lim_{t \rightarrow 0} \lambda^u(t) = (0:1) \times (0:1)$ when $a < 0,  b < 0$,
\item $\lim_{t \rightarrow 0} \lambda^u(t) = (1:0) \times (0:1)$ when $a > 0, b < 0$.
\end{enumerate}
These correspond to the relative interiors of the 9 cones of $\Sigma$. For instance, if $u$ lies on the positive $a$-axis, which is the relative interior of the one-dimensional cone $\sigma = {\rm Cone}((1,0)) \in \Sigma$, then $\lim_{t \rightarrow 0} \lambda^u(t) = (1:0) \times (1:1)$. Below, we will call this limit point $\gamma_\sigma$, and in the \emph{orbit-cone} correspondence, we will associate the torus orbit $T \cdot \gamma_\sigma$ to $\sigma$.
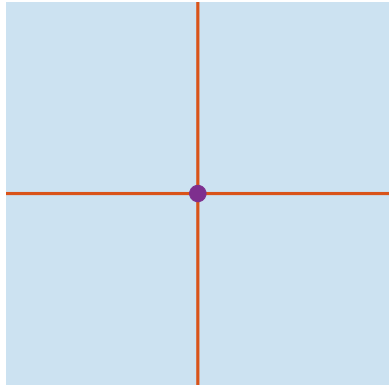
\begin{figure}[h!]
\centering
\begin{tikzpicture}[scale=1]
\begin{axis}[%
width=2in,
height=2in,
scale only axis,
xmin=-3.5,
xmax=3.5,
ymin=-3.5,
ymax=3.5,
ticks = none, 
ticks = none,
axis background/.style={fill=white},
axis line style={draw=none} 
]


\addplot [color=mycolor1,solid,fill opacity=0.2,fill = mycolor1,forget plot]
  table[row sep=crcr]{%
 5 0\\
5 5\\	
0 5\\
0 0 \\
5 0 \\
};

\addplot [color=mycolor1,solid,fill opacity=0.2,fill = mycolor1,forget plot]
  table[row sep=crcr]{%
 0 0 \\
 -5 0 \\
 -5 5\\
 0 5\\
 0 0\\
};

\addplot [color=mycolor1,solid,fill opacity=0.2,fill = mycolor1,forget plot]
  table[row sep=crcr]{%
 0 0 \\
 -5 0 \\
 -5 -5\\
 0 -5\\
 0 0\\
};

\addplot [color=mycolor1,solid,fill opacity=0.2,fill = mycolor1,forget plot]
  table[row sep=crcr]{%
 0 0 \\
 5 0 \\
 5 -5\\
 0 -5\\
 0 0\\
};

\addplot [very thick, color=mycolor2,solid,fill opacity=0.2,fill = mycolor1,forget plot]
  table[row sep=crcr]{%
 0 0 \\
 5 0 \\
};

\addplot [very thick, color=mycolor2,solid,fill opacity=0.2,fill = mycolor1,forget plot]
  table[row sep=crcr]{%
 0 0 \\
 -5 0 \\
};

\addplot [very thick, color=mycolor2,solid,fill opacity=0.2,fill = mycolor1,forget plot]
  table[row sep=crcr]{%
 0 0 \\
 0 5 \\
};

\addplot [very thick, color=mycolor2,solid,fill opacity=0.2,fill = mycolor1,forget plot]
  table[row sep=crcr]{%
 0 0 \\
0 -5 \\
};

\addplot[only marks,mark=*,mark size=3.1pt,mycolor4
        ]  coordinates {
  (0,0)
};

\end{axis}
\end{tikzpicture} 
\caption{The fan of $\PP^1 \times \PP^1$.}
\label{fig:fanP1P1}
\end{figure}
\end{example}

Let $\sigma$ be a strongly convex rational cone in $N_\R \simeq \R^n$. Recall that points of $U_\sigma$ are semigroup homomorphisms $\gamma : \S_\sigma \rightarrow \C$.

\begin{definition}[Distinguished point] \label{def:distinguished}
The \emph{distinguished point} of the affine toric variety $U_\sigma$ is the semigroup homomorphism 
\[ \gamma_\sigma: m \in \S_\sigma \longmapsto \begin{cases}
1 & m \in \S_\sigma \cap \sigma^\perp \\
0 & \text{otherwise}.
\end{cases} \]
\end{definition}
One can check that $\gamma_\sigma$ is indeed a semigroup homomorphism. By Proposition \ref{prop:fixedpoint}, $\gamma_\sigma$ is fixed under the torus action if and only if $\dim \sigma = n$, in which case $\sigma^\perp = \{0\}$. Note that when $\tau \preceq \sigma$ is a face of $\sigma$, $\gamma_\tau \in U_\sigma$ since $U_\tau \subset U_\sigma$. The point $\gamma_\tau$ is a semigroup homomorphism $\S_\tau \rightarrow \C$. Viewed as a point in $U_\sigma$, it is the restriction of $\gamma_\tau$ to $\S_\sigma \subset \S_\tau$. 

\begin{proposition} \label{prop:limits}
Let $\sigma$ be a strongly convex rational cone in $N_\R$ and $u \in N$. We have that $u \in \sigma$ if and only if the limit $\lim_{t \rightarrow 0} \lambda^u(t)$ exists in $U_\sigma$. Moreover, if $u$ lies in the relative interior of $\sigma$, we have $\lim_{t\rightarrow 0} \lambda^u(t) = \gamma_\sigma$.
\end{proposition}
\begin{proof}
The limit $\lim_{t \rightarrow 0} \lambda^u(t)$ exists in $U_\sigma$ if and only if 
\begin{align*}
\lim_{t \rightarrow 0} \chi^m(\lambda^u(t)) \text{ exists in $\C$ for all $m \in \sigma^\vee \cap M$} \Longleftrightarrow & \lim_{t \rightarrow 0} t^{\langle u, m \rangle} \text{ exists in $\C$ for all $m \in \sigma^\vee \cap M$} \\
\Longleftrightarrow & \langle u,m \rangle \geq 0 \text{ for all $m \in \sigma^\vee$} \\
\Longleftrightarrow & u \in (\sigma^\vee)^\vee = \sigma.
\end{align*}
If $u \in \sigma$ lies in the relative interior of $\sigma$, $\lim_{t \rightarrow 0} \lambda^u(t)$ is the semigroup homomorphism that sends $m \mapsto \lim_{t \rightarrow 0} t^{\langle u, m \rangle}$. This coincides with $\gamma_\sigma$, since $\langle u, m \rangle \geq 0$ and equality holds if and only if $m \in \sigma^\perp \cap M \subset \sigma^\vee \cap M = \S_\sigma$. 
\end{proof}
Proposition \ref{prop:limits} justifies our claim that $\Sigma$ can be recovered from limits of one-parameter subgroups in $X_\Sigma$.
\begin{exercise} \label{ex:limitsP2}
Repeat Example \ref{ex:limitsP1P1} for $X = \PP^2$. Write down the distinguished points in their homogeneous coordinates and associate them to the cones of the fan $\Sigma$ in Figure \ref{fig:fanP2}.
\end{exercise}

\subsection{Orbit-cone correspondence}

Let $\Sigma$ be a fan in $N_\R$. In this section, our goal is to show that the cones $\sigma$ in $\Sigma$ correspond to orbits $O(\sigma)$ of the torus action $T \times X_\Sigma \rightarrow X_\Sigma$. These orbits stratify $X_\Sigma$, and inclusion relations $\overline{O(\sigma)} \subset \overline{O(\tau)}$ of closed strata correspond to the reversed face relations $\tau \preceq \sigma$. 

For a cone $\sigma \in \Sigma$, we define $O(\sigma) = T \cdot \gamma_\sigma \subset U_\sigma \subset X_\Sigma$. These orbits are themselves tori. We first identify their lattices. Let $N_\sigma = \Z \cdot (\sigma \cap N)$ and $N(\sigma) = N/N_\sigma$. The dual $N(\sigma)^\vee = \Hom_\Z(N(\sigma), \Z)$ is the group $\sigma^\perp \cap M$. Therefore
\[ \Hom_\Z(\sigma^\perp \cap M, \C^*) \simeq T_{N(\sigma)}, \]
where $T_{N(\sigma)} = N(\sigma) \otimes_\Z \C^*$ is the torus with cocharacter lattice $N(\sigma)$. Explicitly, if $[u]$ is the class of $u \in N$ in $N(\sigma)$, a point $[u] \otimes t \in T_{N(\sigma)}$ gives the homomorphism of groups $\sigma^\perp \cap M \rightarrow \C^*$ given by $m \mapsto \chi^m(\lambda^u(t)) = t^{\langle u, m \rangle}$. 

\begin{lemma} \label{lem:orbittori}
 Let $\sigma \subset N_\R$ be a strongly convex rational cone. We have 
\begin{align*}
O(\sigma) &= \{ \gamma: \S_\sigma \rightarrow \C ~|~ \gamma(m) \neq 0 \Leftrightarrow m \in \sigma^\perp \cap M \} \,  \simeq \,  \Hom_\Z(\sigma^\perp \cap M, \C^*) \, \simeq \, T_{N(\sigma)}.
\end{align*}
\end{lemma}
\begin{proof}
Let $O' = \{ \gamma: \S_\sigma \rightarrow \C ~|~ \gamma(m) \neq 0 \Leftrightarrow m \in \sigma^\perp \cap M \}$. We establish the isomorphism $O' \simeq \Hom_\Z(\sigma^\perp \cap M, \C^*)$. If $\gamma \in O'$, then $\hat{\gamma} = \gamma_{|\sigma^\perp \cap M}: \sigma^\perp \cap M \rightarrow \C^*$ is a group homomorphism. Conversely, if $\hat{\gamma}: \sigma^\perp \rightarrow \C^*$ is a group homomorphism, then 
\[ \gamma: S_\sigma \rightarrow \C \quad \text{given by} \quad  
\gamma(m) = \begin{cases} \hat{\gamma}(m) & m \in \sigma^\perp \cap M \\
0 & \text{otherwise} \end{cases} \]
is an element of $O(\sigma)$. This defines the inverse of $\gamma \mapsto \hat{\gamma}$, which establishes the isomorphism.

It remains to show that $O' = O(\sigma)$. Clearly $\gamma_\sigma \in O'$. Also, by $(t \cdot \gamma)(m) = \chi^m(t) \gamma(m)$ we see that $O'$ is $T$-invariant. We conclude $O(\sigma) \subset O'$. It remains to show that $T$ acts transitively on $O'$. This follows from tensoring the exact sequence $N \rightarrow N(\sigma) \rightarrow 0$ with $\C^*$, which gives $T \rightarrow T_{N(\sigma)} \rightarrow 1$. Hence $T$ acts transitively on $T_{N(\sigma)}$. This action is compatible with $T_{N(\sigma)} \simeq \Hom_\Z(\sigma^\perp \cap M, \C^*) \simeq O'$. We encourage the reader to check this.
\end{proof}

We are now ready to state the main result of this section. 

\begin{theorem}[The orbit-cone correspondence] \label{thm:orbitcone}
Let $X_\Sigma$ be the toric variety of a fan $\Sigma$ in $N_\R \simeq \R^n$, with dense torus $T$. 
\begin{enumerate}
\item The map $\sigma \mapsto O(\sigma)$ is a bijection of $\{$ cones $\sigma \in \Sigma$ $\}$ and $\{$ $T$-orbits in $X_\Sigma$ $\}$.
\item For each $\sigma \in \Sigma$, $\dim \sigma + \dim O(\sigma) = n$. 
\item The affine open subset $U_\sigma$ is the disjoint union of orbits $U_\sigma = \bigsqcup_{\tau \preceq \sigma} O(\tau)$. 
\item We have $\tau \preceq \sigma$ if and only if $O(\sigma) \subset \overline{O(\tau)}$ and $\overline{O(\tau)} = \bigsqcup_{\tau \preceq \sigma} O(\sigma)$. This holds for the closure in $X_\Sigma$ in both the Zariski topology and the classical topology.
\end{enumerate}
\end{theorem}
\begin{proof}
1.~Let $O \subset X_\Sigma$ be a $T$-orbit. Since $X_\Sigma$ is covered by the $T$-invariant affine toric varieties $U_\sigma$, $\sigma \in \Sigma$, there is a minimal cone $\sigma$ (with respect to inclusion) such that $O \subset U_\sigma$. Point 1 will follow from $O = O(\sigma)$. For $\gamma \in O$, there is a face $\tau \preceq \sigma$ for which
\[ \{ m \in \S_\sigma ~|~ \gamma(m) \neq 0 \} = \sigma^\vee \cap \tau^\perp \cap M.\]
This means that $\gamma$ extends to a semigroup homomorphism $\S_\tau \rightarrow \C$, and hence $O \subset U_\tau$, which implies $\tau = \sigma$ by minimality of $\sigma$. Then $O = O(\sigma)$ by Lemma \ref{lem:orbittori}.

2.~This is a straightforward consequence of Lemma \ref{lem:orbittori}.

3.~Since $U_\sigma$ is $T$-invariant, it is a disjoint union of orbits. If $\tau \preceq \sigma$, we have $O(\tau) \subset U_\tau \subset U_\sigma$. The proof of part 1 implies that any orbit $O \subset U_\sigma$ is $O(\tau)$ for some face $\tau \preceq \sigma$. 

4.~We show that $\tau \preceq \sigma$ if and only if $O(\sigma) \subset \overline{O(\tau)}$ in the classical topology. For the `only if' direction, note that $\overline{O(\tau)}$ is $T$-invariant by continuity of $T \times O(\tau) \rightarrow O(\tau)$. Hence, it suffices to show that $\overline{O(\tau)} \cap O(\sigma) \neq \emptyset$. We prove this using limits of one-parameter subgroups. Pick $u \in \sigma$ in the relative interior of $\sigma$ and define $\gamma(t) = \lambda^u(t) \cdot \gamma_\tau$. For $t \in \C^*$, this semigroup homomorphism in $O(\tau) \subset U_\sigma$ sends $m \in \S_\sigma$ to $t^{\langle u, m \rangle} \gamma_\tau(m)$. Therefore, $\lim_{t \rightarrow 0} \gamma(t) $ sends $m \in (\sigma^\vee \setminus \sigma^\perp) \cap M$ to 0, and $m \in \sigma^\perp \cap M$ to $\gamma_\tau(m) = 1$, since $\sigma^\perp \subset \tau^\perp$. Hence $\gamma(0) = \lim_{t \rightarrow 0} \gamma(t) = \gamma_\sigma \in O(\sigma)$ and by construction $\gamma(0) \in \overline{O(\tau)}$. 

For the `if' direction, suppose that $O(\sigma) \subset \overline{O(\tau)}$. This implies $O(\tau) \subset U_\sigma$. Indeed, otherwise $U_\sigma \cap O(\tau) = \emptyset$ ($U_\sigma$ is $T$-invariant), but then $U_\sigma \cap \overline{O(\tau)} = \emptyset$ ($X_\Sigma \setminus U_\sigma$ is closed and contains $O(\tau)$, so it must contain $\overline{O(\tau)}$). By part 1, $O(\tau) \subset U_\sigma$ implies $\tau \preceq \sigma$. For the proof of point 4 in the Zariski topology, see \cite[Thm.~3.2.6]{cox2011toric}.
\end{proof}

Note that, as a particular case of part 4 of Theorem \ref{thm:orbitcone}, for $\tau = \{0\}$ we have 
\[ X_\Sigma = \bigsqcup_{\sigma \in \Sigma} O(\sigma). \]
We reiterate that the cones $\sigma \in \Sigma$ give an \emph{affine open covering}, as well as a \emph{stratification} of $X_\Sigma$. A cone $\sigma \in \Sigma$ corresponds to 
\begin{enumerate}
\item an affine open subset $U_\sigma$ of dimension $n$ and
\item a torus orbit $O(\sigma) \subset U_\sigma$ of dimension $n - \dim(\sigma)$.
\end{enumerate}

In case $X_\Sigma \simeq X_P$ is the toric variety of the full-dimensional polytope $P \subset \R^n$, cones of $\sigma$ also correspond to faces of $P$. By Theorem \ref{thm:orbitcone}, this means that torus orbits correspond to faces of $P$. Moreover, the face $F_\sigma$ of an orbit $O(\sigma)$ has the same dimension $\dim O(\sigma) = n - \dim(\sigma)$, and $F_\sigma \subset F_\tau \Longleftrightarrow \tau \preceq \sigma  \Longleftrightarrow O(\sigma) \subset \overline{O(\tau)}$. 
\begin{example} \label{ex:squarefromfan}
The surface $\PP^1 \times \PP^1$ is $X_P$, where $P = [0,1]^2 \subset \R^2$ is the square. Figure \ref{fig:P1P1square} illustrates the orbit-cone(-face) correspondence. The vertices of $P$ correspond to the full-dimensional cones (blue), and to the torus fixed points. Two such points are connected by a one-dimensional orbit if they lie on an edge of $P$ (orange). All orbits lie in the closure of the dense orbit corresponding $\sigma = \{0\}$, and therefore to the `dense' face of $P$ (purple).
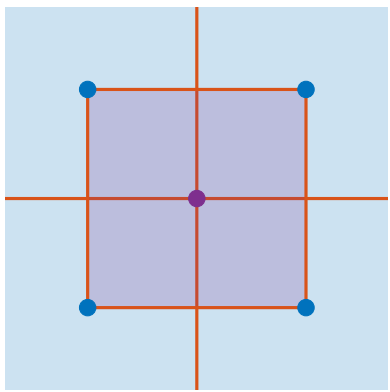
\begin{figure}
\centering
\begin{tikzpicture}[scale=1]
\begin{axis}[%
width=2in,
height=2in,
scale only axis,
xmin=-3.5,
xmax=3.5,
ymin=-3.5,
ymax=3.5,
ticks = none, 
ticks = none,
axis background/.style={fill=white},
axis line style={draw=none} 
]


\addplot [color=mycolor1,solid,fill opacity=0.2,fill = mycolor1,forget plot]
  table[row sep=crcr]{%
 5 0\\
5 5\\	
0 5\\
0 0 \\
5 0 \\
};

\addplot [color=mycolor1,solid,fill opacity=0.2,fill = mycolor1,forget plot]
  table[row sep=crcr]{%
 0 0 \\
 -5 0 \\
 -5 5\\
 0 5\\
 0 0\\
};

\addplot [color=mycolor1,solid,fill opacity=0.2,fill = mycolor1,forget plot]
  table[row sep=crcr]{%
 0 0 \\
 -5 0 \\
 -5 -5\\
 0 -5\\
 0 0\\
};

\addplot [color=mycolor1,solid,fill opacity=0.2,fill = mycolor1,forget plot]
  table[row sep=crcr]{%
 0 0 \\
 5 0 \\
 5 -5\\
 0 -5\\
 0 0\\
};

\addplot [very thick, color=mycolor2,solid,fill opacity=0.2,fill = mycolor1,forget plot]
  table[row sep=crcr]{%
 0 0 \\
 5 0 \\
};

\addplot [very thick, color=mycolor2,solid,fill opacity=0.2,fill = mycolor1,forget plot]
  table[row sep=crcr]{%
 0 0 \\
 -5 0 \\
};

\addplot [very thick, color=mycolor2,solid,fill opacity=0.2,fill = mycolor1,forget plot]
  table[row sep=crcr]{%
 0 0 \\
 0 5 \\
};

\addplot [very thick, color=mycolor2,solid,fill opacity=0.2,fill = mycolor1,forget plot]
  table[row sep=crcr]{%
 0 0 \\
0 -5 \\
};

\addplot[only marks,mark=*,mark size=3.1pt,mycolor4
        ]  coordinates {
  (0,0)
};

\addplot [color=mycolor1,solid,fill opacity=0.2,fill = mycolor4,forget plot]
  table[row sep=crcr]{%
  2 2\\
2 -2\\	
-2 -2\\
-2 2\\
2 2\\
};

\addplot[only marks,mark=*,mark size=3.1pt,mycolor1
        ]  coordinates {
  (2,2) (2,-2) (-2,2) (-2, -2)
};

\addplot [very thick, color=mycolor2,solid,fill opacity=0.2,fill = mycolor2,forget plot]
  table[row sep=crcr]{%
2 2 \\
2 -2 \\
};

\addplot [very thick, color=mycolor2,solid,fill opacity=0.2,fill = mycolor2,forget plot]
  table[row sep=crcr]{%
2 -2 \\
-2 -2 \\
};

\addplot [very thick, color=mycolor2,solid,fill opacity=0.2,fill = mycolor2,forget plot]
  table[row sep=crcr]{%
-2 -2 \\
-2 2 \\
};

\addplot [very thick, color=mycolor2,solid,fill opacity=0.2,fill = mycolor2,forget plot]
  table[row sep=crcr]{%
-2 2 \\
2 2 \\
};

\end{axis}
\end{tikzpicture} 
\caption{Faces of $P$ correspond to torus orbits of $X_P$: an illustration for $\PP^1 \times \PP^1$.}
\label{fig:P1P1square}
\end{figure}
\end{example}

\begin{exercise}
Repeat Example \ref{ex:squarefromfan} for $X_P = \PP^2$. Note that, because our convention is to use \emph{inward} pointing facet normals, we usually get a rotated version of $P$ in pictures like Figure \ref{fig:P1P1square}. This is not seen in Figure \ref{fig:P1P1square} because of symmetry.
\end{exercise}

\begin{example} \label{ex:removecones}
By Theorem \ref{thm:orbitcone}, if we remove cones from the fan $\Sigma$ to obtain a new fan $\Sigma'$, we obtain $X_{\Sigma'}$ by removing the corresponding torus orbits from $X_\Sigma$. For instance, $\C^2 = X_\Sigma$ where $\Sigma$ is the nonnegative quadrant and all its faces. Removing the 2-dimensional cone, we see that $\C^2 \setminus \{0\}$ is the toric variety of the fan $\Sigma' = \{ \{0\}, {\rm Cone}((1,0)), {\rm Cone}((0,1)) \}$. 
\end{example}

\subsection{Orbit closures as toric varieties}
Following \cite{cox2011toric}, we introduce the notation $V(\tau) = \overline{O(\tau)}$ for the orbit closures. We now characterize $V(\tau)$ as a normal toric variety. We have seen in the previous section that the dense open subset $O(\tau) \simeq T_{N(\tau)}$ of $V(\tau)$ is a torus. Moreover, The $T$-orbits in 
\[ V(\tau) = \bigsqcup_{\tau \preceq \sigma} O(\sigma) \]
can be viewed as $T_{N(\tau)}$-orbits via the commuting diagram
\begin{center}
\begin{tikzcd}
T \times O(\sigma) \ar[r] \ar[dr]
    & O(\sigma)  \\
    & T_{N(\tau)} \times O(\sigma) \ar[u]
\end{tikzcd}.
\end{center}
Here the diagonal arrow sends $(u \otimes t) \times p \in (N \otimes_\Z \C^*) \times O(\sigma) \simeq T \times O(\sigma)$ to $([u] \otimes t) \times p$, where $[u]$ is the class of $u$ in $N(\tau) = N/N_\tau$. The following is \cite[Exer.~3.2.7, Prop.~3.2.7]{cox2011toric}.

\begin{proposition} \label{prop:orbitfan}
Let $\overline{\sigma}$ be the image of $\sigma \in \Sigma$ under $N_\R \rightarrow N(\tau)_\R$. Consider the collection of cones ${\rm Star}(\tau) = \{\overline{\sigma} \subset N(\tau)_\R ~|~ \tau \preceq \sigma \in \Sigma\}$. The set ${\rm Star}(\tau)$ is a fan in $N(\tau)_\R$ and 
\[ V(\tau) = \overline{O(\tau)} = X_{{\rm Star}(\tau)}. \]
\end{proposition}

\begin{example}
Let $\tau$ be the ray $\R_{\geq 0} \cdot (1,1)$ of the fan from Example \ref{ex:blowupC2}. The closure $V(\tau)$ of the corresponding one-dimensional torus orbit has cocharacter lattice $N(\tau) \simeq \Z$, and its fan in $N(\tau)_R \simeq \R$ is complete. This is illustrated in Figure \ref{fig:fanBl0C2proj}. By Example \ref{ex:1D} and Proposition \ref{prop:orbitfan}, $V(\tau) \simeq \PP^1$. This is the exceptional divisor of ${\rm Bl}_0(\C^2)$. 
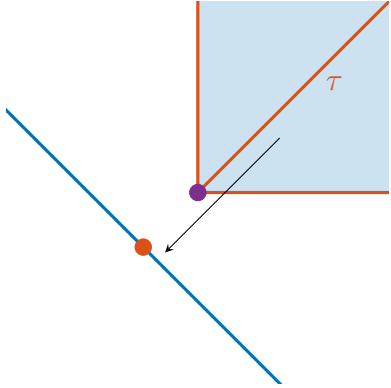
\begin{figure}
\centering
\begin{tikzpicture}[scale=1]
\begin{axis}[%
width=2in,
height=2in,
scale only axis,
xmin=-3.5,
xmax=3.5,
ymin=-3.5,
ymax=3.5,
ticks = none, 
ticks = none,
axis background/.style={fill=white},
axis line style={draw=none} 
]


\addplot [color=mycolor1,solid,fill opacity=0.2,fill = mycolor1,forget plot]
  table[row sep=crcr]{%
 5 0\\
5 5\\	
0 0 \\
5 0 \\
};

\addplot [color=mycolor1,solid,fill opacity=0.2,fill = mycolor1,forget plot]
  table[row sep=crcr]{%
0 5\\
5 5\\	
0 0 \\
0 5\\
};

\addplot [very thick, color=mycolor2,solid,fill opacity=0.2,fill = mycolor1,forget plot]
  table[row sep=crcr]{%
 0 0 \\
 5 0 \\
};

\addplot [very thick, color=mycolor2,solid,fill opacity=0.2,fill = mycolor1,forget plot]
  table[row sep=crcr]{%
 0 0 \\
 0 5 \\
};

\addplot [very thick, color=mycolor2,solid,fill opacity=0.2,fill = mycolor1,forget plot]
  table[row sep=crcr]{%
 0 0 \\
 5 5 \\
};

\addplot[only marks,mark=*,mark size=3.1pt,mycolor4
        ]  coordinates {
  (0,0)
};

\addplot [very thick, color=mycolor1,solid,fill opacity=0.2,fill = mycolor1,forget plot]
  table[row sep=crcr]{%
 -6 4\\
 4 -6 \\
};

\addplot[only marks,mark=*,mark size=3.1pt,mycolor2
        ]  coordinates {
  (-1,-1)
};

\node (P) at (axis cs:2.5,2) {\textcolor{mycolor2}{$\tau$}};

\draw[->,>=stealth] (axis cs:1.5,1.0) -- (axis cs:-0.6,-1.1);

\end{axis}
\end{tikzpicture} 
\caption{The exceptional divisor in ${\rm Bl}_0(\C_2)$ is the toric variety $V(\tau)$.}
\label{fig:fanBl0C2proj}
\end{figure}
\end{example}
Proposition \ref{prop:orbitfan} has the following corollary in the case where $X_\Sigma \simeq X_P$.
\begin{proposition} \label{prop:facesorbits}
If $\Sigma = \Sigma_P$ is the normal fan of a full-dimensional polytope $P \subset M_\R$, then torus orbits correspond to faces $Q \preceq P$ and $V(\sigma_Q) \simeq X_Q$, where $Q$ is considered as a full-dimensional polytope in its affine span. 
\end{proposition}
\vspace{0.5cm}
\begin{minipage}{0.3 \linewidth}
\centering
\tdplotsetmaincoords{70}{110}
\begin{tikzpicture}[baseline=(A.base),scale=1.5,tdplot_main_coords]
\coordinate (A) at (0,0,0);
\coordinate (B) at (1,0,0);
\coordinate (C) at (0,1,0);
\coordinate (D) at (0,0,2);
\coordinate (E) at (1,0,2);
\coordinate (F) at (0,1,2);
			
\draw[fill opacity=0.2,fill = mycolor2,] (A)--(B)--(C)--cycle;
\draw[fill opacity=0.2,fill = mycolor2,] (D)--(E)--(F)--cycle;
\draw[fill opacity=0.2,fill = mycolor2,] (A)--(B)--(E)--(D)--cycle;
\draw[fill opacity=0.2,fill = mycolor2,] (A)--(C)--(F)--(D)--cycle; 
\draw[fill opacity=0.2,fill = mycolor2,] (B)--(C)--(F)--(E)--cycle; 
\end{tikzpicture}
\end{minipage}
\begin{minipage}{0.6 \linewidth}
\begin{exercise}
Consider the triangular prism $P$ shown on the left. Show that $X_P \simeq \PP^2 \times \PP^1$. Check that $X_P$ has five two-dimensional torus orbits, whose closures are isomorphic to either $\PP^2$ or $\PP^1 \times \PP^1$. Describe these orbits in homogeneous coordinates.
\end{exercise}
\end{minipage}

\subsection{Toric morphisms} \label{sec:tormorph}
In this section, we study toric morphisms between abstract toric varieties. 
\begin{definition}[Toric morphism]
Let $X_{\Sigma_1}, X_{\Sigma_2}$ be normal toric varieties, with dense tori $T_1, T_2$ respectively. A morphism $\phi: X_{\Sigma_1} \rightarrow X_{\Sigma_2}$ is called \emph{toric} if $\phi(T_1) \subset T_2$ and $\phi_{|T_1}$ is a group homomorphism. 
\end{definition}
A morphism $\phi: X_{\Sigma_1} \rightarrow X_{\Sigma_2}$ is called \emph{equivariant} if $\phi(t \cdot p) = \phi(t) \cdot \phi(p)$ for $t \in T_1, p \in X_{\Sigma_1}$. 
\begin{proposition} \label{prop:equivariance}
Toric morphisms $\phi: X_{\Sigma_1} \rightarrow X_{\Sigma_2}$ are equivariant.
\end{proposition}
\begin{proof}
The proof is identical to that of Proposition \ref{prop:equivariantaff}.
\end{proof}

Similar to what we saw in Proposition \ref{prop:compcone}, it turns out that all toric morphisms between toric varieties correspond to linear maps that respect their fans. We now make this precise. 
\begin{definition}[Compatible]
Let $\Sigma_i$ be a fan in $(N_i)_\R$ for $i = 1, 2$. A $\Z$-linear map $\overline{\phi}: N_1 \rightarrow N_2$ is \emph{compatible} with $\Sigma_1, \Sigma_2$ if for each $\sigma_1 \in \Sigma_1$, there is $\sigma_2 \in \Sigma_2$ such that $\overline{\phi}_\R(\sigma_1) \subset \sigma_2$. Here, we used some notation from Section \ref{subsec:toricmorphismsaff}.
\end{definition}

\begin{example} \label{ex:compatible}
Consider the fan $\Sigma_1$ in $\R^2$ shown in Figure \ref{fig:H2} and the complete fan $\Sigma_2$ in $\R$.
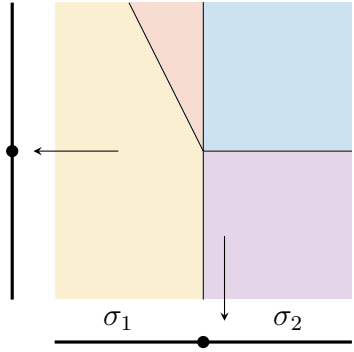
\begin{figure}
\centering
\begin{tikzpicture}[scale=1]
\begin{axis}[%
width=2in,
height=2in,
scale only axis,
xmin=-5.5,
xmax=3.5,
ymin=-5.5,
ymax=3.5,
ticks = none, 
ticks = none,
axis background/.style={fill=white},
axis line style={draw=none} 
]


\addplot [color=white,solid,fill opacity=0.2,fill = mycolor1,forget plot]
  table[row sep=crcr]{%
 5 0\\
5 5\\	
0 5\\
0 0 \\
5 0 \\
};

\addplot [color=white,solid,fill opacity=0.2,fill = mycolor2,forget plot]
  table[row sep=crcr]{%
 0 0 \\
 -3.5 7 \\
 0 7\\
 0 0\\
};

\addplot [color=white,solid,fill opacity=0.2,fill = mycolor3,forget plot]
  table[row sep=crcr]{%
 0 0 \\
 -3.5 7 \\
 -3.5 -3.5\\
 0 -3.5\\
 0 0\\
};

\addplot [color=white,solid,fill opacity=0.2,fill = mycolor4,forget plot]
  table[row sep=crcr]{%
 0 0 \\
 5 0 \\
 5 -3.5\\
 0 -3.5\\
 0 0\\
};

\addplot [color=black,solid,fill opacity=0.2,fill = mycolor1,forget plot]
  table[row sep=crcr]{%
 0 0 \\
 5 0 \\
};

\addplot [color=black,solid,fill opacity=0.2,fill = mycolor1,forget plot]
  table[row sep=crcr]{%
 0 0 \\
 0 5 \\
};

\addplot [color=black,solid,fill opacity=0.2,fill = mycolor1,forget plot]
  table[row sep=crcr]{%
 0 0 \\
 -1.75  3.5 \\
};

\addplot [color=black,solid,fill opacity=0.2,fill = mycolor1,forget plot]
  table[row sep=crcr]{%
 0 0 \\
 0  -3.5 \\
};

\addplot [very thick, color=black,solid,fill opacity=0.2,fill = mycolor1,forget plot]
  table[row sep=crcr]{%
 -3.5 -4.5 \\
 3.5  -4.5 \\
};

\addplot [very thick, color=black,solid,fill opacity=0.2,fill = mycolor1,forget plot]
  table[row sep=crcr]{%
 -4.5 -3.5 \\
 -4.5  3.5 \\
};

\addplot[only marks,mark=*,mark size=2.1pt,black
        ]  coordinates {
    (0,-4.5)
    (-4.5,0)
};

\node (P) at (axis cs:-2,-4) {$\sigma_1$};
\node (P) at (axis cs:2,-4) {$\sigma_2$};

\draw[->,>=stealth] (axis cs:0.5,-2) -- (axis cs:0.5,-4);
\draw[->,>=stealth] (axis cs:-2,0.0) -- (axis cs:-4,-0.0);


\end{axis}
\end{tikzpicture} 
\caption{A compatible and a non-compatible projection of the fan of a Hirzebruch surface.}
\label{fig:H2}
\end{figure}
The varieties $X_{\Sigma_1}$ and $X_{\Sigma_2}$ are a Hirzebruch surface and $\PP^1$ respectively. 
The coordinate projection $\overline{\phi} (x,y) = x$ is compatible with $\Sigma_1, \Sigma_2$. The orange and the yellow cone of $\Sigma_1$, as well as all their faces, are mapped into $\sigma_1$. The blue and purple cones are mapped into $\sigma_2$. The other coordinate projection $\overline{\phi} (x,y) = y$ is not compatible with $\Sigma_1, \Sigma_2$, as the image of the yellow cone is not contained in any of the cones of $\Sigma_2$. 
\end{example}

\begin{definition}[Glued morphisms] \label{def:gluedmorphism}
Let $X = \bigcup_{i} U_i$ and $Y$ be varieties, and let $\phi_i : U_i \rightarrow Y$ be morphisms. A morphism $\phi: X \rightarrow Y$ is said to be \emph{glued} from the $\phi_i$ if $\phi_{|U_i} = \phi_i$ for all $i$.
\end{definition}

\begin{lemma} \label{lem:gluedmorphism}
Let $X, Y, U_i, \phi_i$ be as in Definition \ref{def:gluedmorphism}. A morphism $\phi: X \rightarrow Y$ glued from $\phi_i$ exists if and only if $(\phi_i)_{|U_i \cap U_j} = (\phi_j)_{|U_i \cap U_j}$ for all $i, j$. 
\end{lemma}

The proof of Lemma \ref{lem:gluedmorphism} is left as an exercise \cite[Exer.~3.3.1]{cox2011toric}. We are now ready to state the main result of this section. 

\begin{theorem} \label{thm:toricmorph}
Let $\Sigma_i$ be a fan in $(N_i)_\R, i = 1, 2$ and let $X_{\Sigma_1}, X_{\Sigma_2}$ be the corresponding toric varieties with dense tori $T_1$ and $T_2$ respectively. 
\begin{enumerate}
\item If $\overline{\phi}: N_1 \rightarrow N_2$ is a $\Z$-linear map that is compatible with $\Sigma_1, \Sigma_2$, then there is a toric morphism $\phi: X_{\Sigma_1} \rightarrow X_{\Sigma_2}$ satisfying $\phi_{|T_1} = \overline{\phi} \otimes {\rm id} : u \otimes t \longmapsto \overline{\phi}(u) \otimes t \in T_2$.
\item If $\phi: X_{\Sigma_1} \rightarrow X_{\Sigma_2}$ is a toric morphism, then $\phi$ induces a $\Z$-linear map $\overline{\phi} : N_1 \rightarrow N_2$, compatible with $\Sigma_1, \Sigma_2$, such that $\phi_{|T_1} = \overline{\phi} \otimes {\rm id}$.
\end{enumerate}
\end{theorem}
\begin{proof}
1. Fix $\sigma_1 \in \Sigma_1$. By compatibility, there is a cone $\sigma_2 \in \Sigma_2$ such that $\overline{\phi}(\sigma_1) \subset \sigma_2$. By Proposition \ref{prop:compcone}, $\overline{\phi}$ induces a morphism $\phi_{\sigma_1}: U_{\sigma_1} \rightarrow U_{\sigma_2}$. If $\overline{\phi}$ is given by a $n_2 \times n_1$ integer matrix $F$, $\phi_{\sigma_1}$ comes from $\widehat{\phi}: m \mapsto F^\top m$. The maps $\{ \phi_{\sigma_1} \}_{\sigma_1 \in \Sigma_1}$ agree on overlaps, so by Lemma \ref{lem:gluedmorphism}, they glue to $\phi: X_{\Sigma_1} \rightarrow X_{\Sigma_2}$. The glued morphism $\phi$ is toric since $\phi_{\{0\}}: T_1 \rightarrow T_2$ is the group homomorphism $\overline{\phi} \otimes {\rm id}: N_1 \otimes_\Z \C^* \rightarrow N_2 \otimes_\Z \C^*$. In coordinates, this is given by $t \mapsto (t^{F_{i,:}})_{i = 1, \ldots, n_2}$, see Example \ref{ex:F}.

2. By definition, $\phi_{|T_1}$ is a group homomorphism. It induces $\overline{\phi}: N_1 \rightarrow N_2$ by sending $u \in N_1$ to the cocharacter $\phi_{|T_1} \circ \lambda^u: \C^* \rightarrow T_2$. Since $\phi$ is equivariant, it sends the orbit $O(\sigma_1)$ into an orbit $O(\sigma_2)$, where $\sigma_i \in \Sigma_i$. By Proposition \ref{prop:compcone}, to show that $\overline{\phi}_\R(\sigma_1) \subset \sigma_2$, it suffices to show that $\phi(U_{\sigma_1}) \subset U_{\sigma_2}$. By the orbit-cone correspondence (Theorem \ref{thm:orbitcone}), we have $U_{\sigma_i} = \bigsqcup_{\tau_i \preceq \sigma_i} O(\tau_i), i = 1, 2$. We need to show that for each $\tau_1 \preceq \sigma_1$, $\phi(O(\tau_1)) \subset O(\tau_2)$ for some face $\tau_2 \preceq \sigma_2$. Let $\tau_2$ be such that $\phi(O(\tau_1)) \subset O(\tau_2)$. By Theorem \ref{thm:orbitcone}, $O(\sigma_1) \subset \overline{O(\tau_1)}$. Since $\phi$ is continuous, we have $\phi(\overline{O(\tau_1)}) \subset \overline{O(\tau_2)}$. We conclude that $\phi(O(\sigma_1)) \subset O(\sigma_2) \cap \overline{O(\tau_2)}$, and since $\overline{O(\tau_2)}$ is torus invariant, we must have $O(\sigma_2) \subset \overline{O(\tau_2)}$, which implies $\tau_2 \preceq \sigma_2$ by the orbit-cone correspondence.
\end{proof}

\begin{example}
Let $N_i = \Z^2, i = 1, 2$ and $\overline{\phi}: u \mapsto \ell \cdot u$. This is compatible with $\Sigma_1 = \Sigma_2 = \Sigma_{\PP^2}$ (the fan in Figure \ref{fig:fanP2}). The matrix representing $\overline{\phi}$ is $F = \begin{bmatrix}
\ell & 0 \\ 0 & \ell 
\end{bmatrix}$, so that $\phi_{|(\C^*)^2}$ is $(t_1, t_2) \mapsto(t_1^\ell, t_2^\ell)$. Globally, $\phi$ is given by 
$\phi((x_0:x_1:x_2)) = (x_0^\ell:x_1^\ell:x_2^\ell)$. 
\end{example}

\begin{exercise}
Let $\phi: X_{\Sigma_1} \rightarrow \PP^1$ be the toric morphism corresonding to the compatible coordinate projection from Example \ref{ex:compatible}. Embed $X_{\Sigma_1}$ in projective space and compute the map $\phi$ in homogeneous coordinates. 
\end{exercise}

The correspondence between cones of $\Sigma_1$ and $\Sigma_2$ induced by the compatible map $\overline{\phi}: N_1 \rightarrow N_2$ mirrors a correspondence of torus orbits of $X_{\Sigma_1}$ and $X_{\Sigma_2}$ induced by $\phi: X_{\Sigma_1} \rightarrow X_{\Sigma_2}$. Here is the precise statement, which uses some notation from Definition \ref{def:distinguished}.
\begin{lemma} \label{lem:imagedistpoints}
Let $\phi: X_{\Sigma_1} \rightarrow X_{\Sigma_2}$ be the toric morphism coming from $\overline{\phi}:N_1 \rightarrow N_2$. Given $\sigma_1 \in \Sigma_1$, let $\sigma_2 \in \Sigma_2$ be the minimal cone such that $\overline{\phi}_\R(\sigma_1) \subset \sigma_2$. We have 
\begin{itemize}
\item[(a)] $\phi(\gamma_{\sigma_1}) = \gamma_{\sigma_2}$, 
\item[(b)] $\phi(O(\sigma_1)) \subset O(\sigma_2)$ and $\phi(V(\sigma_1)) \subset V(\sigma_2)$, 
\item[(c)] $\phi_{|V(\sigma_1)}: V(\sigma_1) \rightarrow V(\sigma_2)$ is a toric morphism. 
\end{itemize}
\end{lemma}
\begin{proof}
(a) If $u$ is in the relative interior of $\sigma_1$, $\overline{\phi}(u)$ is in the relative interior of $\sigma_2$. This follows from minimality of $\sigma_2$. Therefore, using Proposition \ref{prop:limits},
\[ \phi(\gamma_{\sigma_1}) = \phi \left ( \lim_{t \rightarrow 0} \lambda^u(t) \right ) = \lim_{t \rightarrow 0} \phi(\lambda^u(t)) = \lim_{t \rightarrow 0} \lambda^{\overline{\phi}(u)}(t) = \gamma_{\sigma_2}. \]
Point (b) follows directly from (a), as $\phi$ is equivariant (Proposition \ref{prop:equivariance}) and continuous. By equivariance, $\phi_{|O(\sigma_1)} : O(\sigma_1) \rightarrow O(\sigma_2)$ is a group homomorphism, which implies (c). 
\end{proof}

As applications of toric morphisms, we discuss quotients arising from sublattices of finite index, torus factors, and blow-ups from refinements. 

\subsubsection{Sublattices of finite index}
The observation in Example \ref{ex:sublatticesfiniteindex} is an instance of the following statement \cite[Prop.\ 3.3.7]{cox2011toric}.
\begin{proposition}
Let $N' \subset N$ be a sublattice of finite index. Let $\Sigma$ be a fan in $(N')_\R = N_\R$ and $G = N/N'$. Then $\overline{\phi}: N' \hookrightarrow N$ is compatible with $\Sigma, \Sigma$ and induces the toric morphism $\phi: X_{\Sigma,N'} \rightarrow X_{\Sigma,N}$. 
The finite group $G \simeq \Hom_\Z(G,\C^*) \subset T_{N'}$ acts on $X_{\Sigma,N'}$.
The morphism $\phi$ is constant on $G$-orbits and presents $X_{\Sigma,N}$ as the quotient $X_{\Sigma,N'}/G$. 
\end{proposition}
The inclusion $ \Hom_\Z(G,\C^*) \subset T_{N'} \subset T_{N'}$ can be seen from taking $\Hom_\Z( -, \C^*)$ of 
\[ 0  \longrightarrow N' \longrightarrow N \longrightarrow N/N' \longrightarrow 0. \]

\subsubsection{Torus factors}
A variety $X$ is said to have a \emph{torus factor} if $X \simeq X' \times (\C^*)^r$ for some $r >0$. For toric varieties, this happens when the rays of the fan span a lower dimensional space. The proof is a nice application of Theorem \ref{thm:toricmorph}.
\begin{theorem} \label{thm:torusfactor}
Let $\Sigma$ be a fan in $N_\R$. The following are equivalent: 
\begin{enumerate}
\item[(a)] $X_\Sigma$ has a torus factor.
\item[(b)] There is a non-constant morphism $X_\Sigma \rightarrow \C^*$.
\item[(c)] The rays $\rho \in \Sigma(1)$ do not span $N_\R$. 
\end{enumerate}
\begin{proof}
(a) $\Rightarrow$ (b). A non-constant morphism is obtained from composing the projection $X_\Sigma \rightarrow (\C^*)^r$ with a non-constant character $\chi^m$ of $(\C^*)^r$.

(b) $\Rightarrow$ (c). If $\phi: X_\Sigma \rightarrow \C^*$ is a non-constant morphism, then $\phi_{|T}:T \rightarrow \C^*$ is non-constant. Therefore, $\phi_{|T} = c \cdot \chi^m$ for some $c \in \C^*, m \neq 0$. Multiplying by $c^{-1}$, we may assume that $\phi_{|T} = \chi^m$, and thus that $\phi$ is toric. The corresponding map $\overline{\phi}$ is given by $\overline{\phi}(u) = \langle u, m \rangle$. Since $\overline{\phi}$ is compatible with $\Sigma$ and $\{0\}$, we have $u_\rho \in \ker \overline{\phi}$, for all $\rho \in \Sigma(1)$. Therefore, the ray generators $u_\rho$ are $\R$-linearly dependent. 

(c) $\Rightarrow$ (a). See \cite[Prop.~3.3.9]{cox2011toric}.
\end{proof}
\end{theorem}

\subsubsection{Refinements}
A fan $\Sigma'$ in $N_\R$ is said to \emph{refine} the fan $\Sigma$ in $N_\R$ if every cone of $\Sigma'$ is contained in a cone of $\Sigma$, and the supports of $\Sigma'$ and $\Sigma$ coincide: $|\Sigma'| = | \Sigma|$. If $\Sigma'$ refines $\Sigma$, then the identity map $\overline{\phi} = \id$ is compatible with $\Sigma', \Sigma$. This leads to toric morphisms that are \emph{blow-downs}. 
\begin{example}
The fan in Figure \ref{fig:fanBl0C2} refines the fan of $\C^2$, i.e.\ the nonnegative quadrant $\R_{\geq 0}^2$ and all its faces. The corresponding toric morphism is the blow-down ${\rm Bl}_0(\C^2) \rightarrow \C^2$. 
\end{example}
We now discuss how this generalizes. Let $\sigma$ be an $n$-dimensional smooth cone of $\Sigma$. The \emph{star subdivision} $\Sigma^*(\sigma)$ of $\Sigma$ along $\sigma$ is given by $\Sigma \setminus \sigma \cup \Sigma'(\sigma)$, where $\Sigma'(\sigma)$ is constructed as follows. Let $\sigma = {\rm Cone}(u_1, \ldots, u_n)$, where the $u_i$ are primitive ray generators. Set $u_0 = u_1 + \cdots + u_n$. The fan $\Sigma'(\sigma)$ consists of all cones ${\rm Cone}(u_0, \ldots, \hat{u}_i, \ldots, u_n), i = 1, \ldots, n$, where $\hat{u}_i$ indicates that $u_i$ is omitted, and their faces. 
\begin{proposition}
The star subdivision $\Sigma^*(\sigma)$ of $\Sigma$ along $\sigma$ refines $\Sigma$, and the induced toric morphism $\phi: X_{\Sigma^*(\sigma)} \rightarrow X_\Sigma$ makes $X_{\Sigma^*(\sigma)}$ the blow-up of $X_\Sigma$ at the distinguished point $\gamma_\sigma$ (Definition \ref{def:distinguished}). 
\end{proposition}
\begin{proof}
We sketch the proof. For details, see \cite[Prop.~3.3.15]{cox2011toric}. It is clear that $\Sigma^*(\sigma)$ refines $\Sigma$. By restricting $\phi$, we may assume that $X_\Sigma = U_\sigma$. By Example \ref{ex:smooth}, $U_\sigma \simeq \C^n$. As in Example \ref{ex:blowupC2}, the gluing construction of $X_{\Sigma^*(\sigma)}$ shows that $X_{\Sigma^*(\sigma)} \simeq {\rm Bl}_0(\C^n)$.
\end{proof}

\section{Divisors on toric varieties} \label{sec:divisors}
In this section, we develop the theory of divisors on normal toric varieties. The discussion features the groups of Weil divisors, Cartier divisors, principal divisors, the class group and the Picard group of a toric variety. We will see how these groups have a very explicit description in terms of the fan. We also discuss sheaves associated to divisors. Their global sections will correspond to graded pieces of the Cox rings discussed in Section \ref{sec:homogeneous}. 

\subsection{Background on divisors} \label{subsec:backgrounddiv}
Let $X$ be an irreducible (abstract) variety. A \emph{prime divisor} $D \subset X$ is an irreducible subvariety of codimension 1. To a rational function $f \in \C(X)^* = \C(X) \setminus \{0\}$, we want to associate a $\Z$-linear combination of divisors $\sum a_i D_i$ encoding the `order of vanishing of $f$ along $D_i$'. We will see that this works best when $X$ is normal, which is one of the main reasons for our emphasis on \emph{normal} toric varieties. 

For a prime divisor $D \subset X$, we define 
\[ {\cal O}_{X,D} = \{ f \in \C(X)~|~ f \text{ is defined on } U, \text{ with } U \cap D \neq \emptyset \}. \]
These are the rational functions defined \emph{somewhere} on $D$, and therefore \emph{almost everywhere} on $D$. Since $X$ is irreducible, if $U \subset X$ is open and nonempty, we have $\C(X) = \C(U)$. Moreover, if $U \cap D \neq \emptyset$, ${\cal O}_{X,D} = {\cal O}_{U, U \cap D}$. Therefore, to describe ${\cal O}_{X,D}$ we may assume that $X = \Specm(R)$ is affine. In this case, we have
\[ \{ \text{ prime divisors of $X$ } \} \overset{1:1}{\longleftrightarrow} \{ \text{ codimension 1 prime ideals of $R$ } \}. \]
If a prime divisor $D \subset X$ corresponds to the prime ideal $\p \subset R$, the ring ${\cal O}_{X,D}$ is 
\[ {\cal O}_{X,D} = \Bigl \{ \frac{g}{h} \in K ~|~ g,h \in R, \, h \notin \p \Bigr \} = R_\p, \]
where $K$ is the field of fractions of $R$. This is a local ring with maximal ideal $\p R_\p$. We state the following commutative algebra result without proof. See \cite[Sec.~4.0]{cox2011toric} for details. 

\begin{proposition} \label{prop:DVR}
Let $R$ be a \emph{normal} domain and $\p \subset R$ a codimension one prime ideal with prime divisor $D = V(\p) \subset X$. There is a group homomorphism $\nu_D: K^* \rightarrow \Z$ such that 
\begin{itemize}
\item[(i)] $\nu_D(f + g) \geq \min(\nu_D(f), \nu_D(g))$, whenever $f, g, f+g \in K^*$, 
\item[(ii)] $R_\p = \{ f \in K^* ~|~ \nu_D(f) \geq 0 \} \cup \{ 0 \}$.
\end{itemize}
\end{proposition}

A homomorphism $\nu_D: K^* \rightarrow \Z$ satisfying (i) in Proposition  \ref{prop:DVR} is called a \emph{discrete valuation}. The ring $\{ f \in K^* ~|~ \nu_D(f) \geq 0 \} \cup \{ 0 \}$ is the corresponding \emph{discrete valuation ring}. 

We now describe the map $\nu_D$ from Proposition \ref{prop:DVR} explicitly. If $R, \p, D$ are as in Proposition \ref{prop:DVR}, $R_\p$ is a principal ideal domain and all its ideals are of the form $\langle \pi^k \rangle$, where $\pi$ is a generator of the maximal ideal $\p R_\p$.
For $f \in R_\p \setminus \{0\}$ we have $\nu_D(f) = k$ where $k$ is the largest integer for which $f \in \langle \pi^k \rangle$. In particular, we indeed have
\[ \p R_\p = \langle \pi \rangle = \{ f \in R_\p ~|~ \nu_D(f) > 0 \} \cup \{ 0 \}.\] 
If $f \in K^* \setminus R_\p$, then $\nu_D(f) = -k$ where $k$ is the largest integer for which $f^{-1} \in \langle \pi^k \rangle$. 

We now formulate Proposition \ref{prop:DVR} in the general case where $X$ is not necessarily affine. 

\begin{corollary} \label{cor:DVR}
Let $X$ be a normal variety and $D \subset X$ a prime divisor. There is a discrete valuation $\nu_D: \C(X)^* \rightarrow \Z$ with discrete valuation ring ${\cal O}_{X,D}$. 
\end{corollary}
\begin{proof}
Choose an affine chart $\Specm(R) \subset X$ such that $U \cap D \neq \emptyset$. Then $\C(X) = K$ is the field of fractions of $R$, and ${\cal O}_{X,D} = R_\p$ for a prime ideal $\p$. Now apply Proposition \ref{prop:DVR}.
\end{proof}
For a nonzero rational function $f \in \C(X)^*$, we say that $f$ \emph{vanishes with order} $\nu_D(f)$ along $D$ if $\nu_D(f) >0$, or that $f$ \emph{has a pole of order} $-\nu_D(f)$ along $D$ if $\nu_D(f)<0$.

Let $\Div(X)$ be the free abelian group generated by the prime divisors on $X$. A \emph{Weil divisor} is an element $E = \sum_{D \subset X} a_{D} D$ of $\Div(X)$, where the sum is over all prime divisors $D$ on $X$. Only finitely many integer coefficients $a_{D}$ are allowed to be nonzero. A Weil divisor $E$ is \emph{effective} if all coefficients are nonnegative. In this case we write $E \geq 0$.  The \emph{support} of $E$ is the subvariety ${\rm Supp}(E) = \cup_{a_{D} \neq 0}  \, D  \subset X$. Prime divisors are Weil divisors with only one non-zero coefficient, and we often write $D = \sum a_i D_i$ for a Weil divisor $D \in \Div(X)$, where it is understood that the sum is over a finite index set and the $D_i$ are prime. 

If $X$ is normal, the map $\nu_D$ from Corollary \ref{cor:DVR} gives a way of associating a Weil divisor to a rational function $f \in \C(X)^*$. We define 
\[ \div(f) = \sum_{D \subset X} \nu_D(f) \cdot D. \]
This is a Weil divisor since only finitely many coefficients are nonzero \cite[Lem.~4.0.9]{cox2011toric}. Weil divisors arising in this way are called \emph{principal divisors}. From the observations 
\[ \div(fg) = \div(f) + \div(g), \quad \div(f^{-1}) = - \div(f), \]
it is clear that they form a subgroup. We denote this group by $\Div_0(X) \subset \Div(X)$. 
\begin{example}
Let $f = (x-a_1)^{m_1} \cdots (x-a_r)^{m_r} \in \C[x]$, where $a_i$ are distinct complex numbers. Viewed as a rational function on $\C$, $f$ gives a divisor $\div(f) = \sum_{i=1}^r m_i \cdot \{ a_i \} \in \Div_0(\C)$. Viewed as a rational function on $\PP^1$, $\div(f) = \sum_{i=1}^r m_i \cdot \{a_i\} - (\sum_{i=1}^r m_i) \cdot \{ \infty \}$.
\end{example}
\begin{example} \label{ex:divPn}
On $X = \PP^n$, let $D_i$ be the zero locus of the $i$-th homogeneous coordinate function $x_i$, $i = 1, \ldots, n+1$. Rational functions on $X$ are fractions of homogeneous polynomials of the same degree. One checks that $D_i - D_j = \div(x_i/x_j)$ is a principal divisor, but $D_i \in \Div(X)\setminus \Div_0(X)$ for any $i,j \in \{1, \ldots, n+1\}$.
\end{example}
For an open subset $U \subset X$ and a Weil divisor $D = \sum a_i D_i \in \Div(X)$, we define the \emph{restriction} of $D$ to $U$ as 
\[ D_{|U} = \Bigl ( \sum a_i D_i \Bigr )_{|U} = \sum_{D_i \cap U \neq \emptyset} a_i \cdot (D_i \cap U) \in \Div(U). \]
The divisor $D \in \Div(X)$ is called \emph{Cartier} if it is \emph{locally principal}, meaning that there is an open covering $X = \bigcup_i U_i$ such that $D_{|U_i} \in \Div_0(U_i)$ is the divisor of a rational function on $U_i$. The Cartier divisors form a subgroup of $\Div(X)$, denoted $\CDiv(X) \subset \Div(X)$, and all principal divisors are Cartier: 
\[ \Div_0(X) \subset \CDiv(X) \subset \Div(X).\]
The last inclusion is an equality if $X$ is smooth, see \cite[Thm.~4.0.22(b)]{cox2011toric}.
\begin{theorem} \label{thm:CarisWeil}
If $X$ is smooth, all Weil divisors are Cartier. That is, $\CDiv(X) = \Div(X)$.
\end{theorem}
\begin{example} \label{ex:cartierP1}
Let $X = \PP^1 = U_x \cup U_y$, where $U_x = \PP^1 \setminus \{ x = 0 \}, U_y = \PP^1 \setminus \{ y  = 0 \}$ are as in Example \ref{ex:glueP1}. We have seen that  $D_x = \{ x = 0 \}$ is not principal (Example \ref{ex:divPn}). However, it is locally principal. Indeed, this is implied by Theorem \ref{thm:CarisWeil}, and seen from
\[ (D_x)_{|U_x} = \div(1), \quad (D_x)_{|U_y} = \div \Bigl ( \frac{x}{y} \Bigr ). \qedhere \]
\end{example}
In what follows, we will use the terminology \emph{divisor} for Weil divisors. Two divisors $D, E \in \Div(X)$ are said to be \emph{linearly equivalent} if $D-E \in \Div_0(X)$. This equivalence relation gives two important quotient groups. 
\begin{definition}[Class and Picard group]
The \emph{divisor class group} of a normal variety $X$ is $\Cl(X) = \Div(X)/\Div_0(X)$. The \emph{Picard group} of $X$ is $\Pic(X) = \CDiv(X)/\Div_0(X)$.
\end{definition}

Here is an important result on the class group of certain affine varieties which will help us deal with more complicated cases later.

\begin{theorem} \label{thm:importantzeroclass}
Let $R$ be a unique factorization domain and $X = \Specm(R)$. We have 
\begin{itemize}
\item[(a)] $R$ is normal and every codimension 1 prime ideal is principal. 
\item[(b)] $\Cl(X) = 0$. 
\end{itemize}
\end{theorem}
\begin{proof}
(a) It is a standard exercise in commutative algebra to show that every unique factorization domain is normal. Let $\p$ be a codimension 1 prime ideal and $f \in \p \setminus \{0\}$. Then $f = c \, \prod_{i=1}^s f_i^{a_i}$, where the $f_i$ are prime and $c$ is a unit. Since $\p$ is prime, $f_i \in \p$ for some $i$. This means that $\langle f_i \rangle \subset \p$, and since $\p$ is codimension 1, $\p = \langle f_i \rangle$. 

(b) Let $D_i$ be a prime divisor. Its prime ideal $\p_i \subset R$ is principal by part (a). We write $\p_i = \langle f_i \rangle$. It follows that if $D = \sum_{i=1}^s a_i D_i$, then $D = \div(\prod_{i=1}^s f_i^{a_i})$. Indeed, $ \div(\prod_{i=1}^s f_i^{a_i} )= \sum_{i = 1}^s a_i \nu_{D_i}(f_i) D_i$ and $\nu_{D_i}(f_i) = 1$ since $f_i$ generates $\p_i R_{\p_i}$. 
\end{proof}
\begin{example} \label{ex:zeroclass}
The class group of the affine space $\C^n$ is 0, and so is that of a torus $(\C^*)^n$.
\end{example}
In what follows, we will write $[D]$ for the residue class of $D \in \Div(X)$ in $\Cl(X)$. Our next theorem uses the exercise below.
\begin{exercise} \label{ex:check}
Check that the restriction $\Div(X) \rightarrow \Div(U): D \mapsto D_{|U}$ induces a well-defined map $\Cl(X) \rightarrow \Cl(U): [D] \mapsto [D_{|U}]$. 
\end{exercise}

\begin{theorem} \label{thm:importantES}
Let $X$ be a normal variety and $U \subset X$ a nonempty open subset. Let $D_1, \ldots, D_s$ be the irreducible components of $X \setminus U$ that are prime divisors. Then 
\[ \bigoplus_{i=1}^s \Z \cdot D_i \longrightarrow \Cl(X) \longrightarrow \Cl(U) \longrightarrow 0 \]
is exact. The maps are $\sum_{i=1}^s a_i D_i \mapsto [\sum_{i=1}^s a_i D_i]$ and $[D] \mapsto [D_{|U}]$ (Exercise \ref{ex:check}).
\end{theorem}
\begin{proof}
We sketch the proof. Exactness at $\Cl(U)$ follows from the fact that $D' = \sum a_i D_i' \in \Div(U)$ is the restriction of $D = \sum a_i D_i \in \Div(X)$, where $D_i = \overline{D_i'}$ is the Zariski closure of $D_i'$ in $X$. It remains to show exactness at $\Cl(X)$. Clearly the composition is zero. Suppose $[D]$ restricts to zero in $\Cl(U)$. Then $D_{|U}$ is the divisor of some $f \in \C(U)^*$. Since $\C(X) = \C(U)$, $f$ defines a divisor $\div(f) \in \Div(X)$ which restricts to $D_{|U}$. Hence $(D - \div(f))_{|U} = D_{|U} - \div(f)_{|U} = 0$, and $D-\div(f)$ is supported in $\cup_{i=1}^s D_i$. This means that $D$ is linearly equivalent to an element $ \sum_{i=1}^s a_i D_i \in \bigoplus_{i=1}^s \Z \cdot D_i$, so that $[D] = [\sum_{i=1}^s a_i D_i]$.
\end{proof}
Here is a first example of how Theorems \ref{thm:importantzeroclass} and \ref{thm:importantES} can help us compute class groups. 
\begin{example} \label{ex:classP1}
Let $X = \PP^1$ and $U = U_y = \PP^1 \setminus D_y \simeq \C$ is the open set from Example \ref{ex:cartierP1}. By Theorem \ref{thm:importantES}, we have the exact sequence
\[ \Z \cdot D_y \longrightarrow \Cl(\PP^1) \longrightarrow \Cl(\C) \longrightarrow 0. \]
In this case, the first map is injective, since $k \cdot D_y$ is principal if and only if $k = 0$. By Theorem \ref{thm:importantzeroclass}, $\Cl(\C) = 0$, so that $\Cl(\PP^1) \simeq \Z$.
\end{example}
We conclude with a lemma that will be useful later. For a proof, see \cite[Prop.~4.0.16]{cox2011toric}.

\begin{lemma} \label{lem:useful}
Let $X$ be a normal variety and $f \in \C(X)^*$. 
\begin{itemize}
\item[(a)] The divisor $\div(f)$ is effective if and only if $f: X \rightarrow \C$ is a morphism.
\item[(b)] We have $\div(f) = 0$ if and only if $f: X \rightarrow \C^*$ is a morphism. 
\end{itemize}
\end{lemma}

\subsection{Weil divisors on toric varieties}
We turn back to the case where $X = X_\Sigma$ is a normal toric variety coming from a fan $\Sigma$ in $N_\R \simeq \R^n$. Let $\Sigma(1)$ be the set of rays of $\Sigma$. By the orbit-cone correspondence (Theorem \ref{thm:orbitcone}), this gives torus invariant prime divisors $V(\rho) = \overline{O(\rho)}$, $\rho \in \Sigma(1)$. To emphasize that we think of these as divisors, we write $D_\rho = V(\rho)$. Our first result gives the order of vanishing of a character along the boundary of the torus $T \subset X_\Sigma$. That is, we compute the valuation map $\nu_\rho = \nu_{D_\rho}: \C(X_\Sigma)^* \rightarrow \Z$ from Proposition \ref{prop:DVR} on characters. This makes sense because characters are rational functions on $X_\Sigma$. They are defined on an open subset containing $T$. 

\begin{proposition}
Let $u_\rho$ be the primitive ray generator of $\rho \in \Sigma(1)$. For any $m \in M$, we have $\nu_\rho(\chi^m) = \langle u_\rho, m \rangle$.
\end{proposition}
\begin{proof}
Let $e_1 = u_\rho, e_2, \ldots, e_n$ be a basis of the cocharacter lattice $N$. Note that $u_\rho$ can be extended to such a basis because $u_\rho$ is primitive. The affine toric variety $U_\rho$ is isomorphic to 
\begin{align}  \label{eq:locring}
\C \times (\C^*)^{n-1} = \Specm(\C[x_1,x_2^{\pm 1}, \ldots, x_n^{\pm 1}])
\end{align}
(see Example \ref{ex:smooth}) and $U_\rho \cap D_\rho$ is defined by $x_1 = 0$. Therefore ${\cal O}_{X_\Sigma,D_\rho} = {\cal O}_{U_\rho, U_\rho \cap D_\rho} = \C[x_1, \ldots, x_n]_{\langle x_1 \rangle}$ and $\nu_\rho(f) = k$ where $k$ is such that $f = x_1^k \frac{g}{h}$ for $g, h \in \C[x_1, \ldots, x_n] \setminus \langle x_1 \rangle$. Let $m_1, \ldots, m_n$ be the dual basis of $M$ with respect to $e_1, \ldots, e_n$. The variables $x_i$ in \eqref{eq:locring} correspond to the characters $\chi^{m_i}$. A character $m \in M$ can be written as $\sum_{i=1}^n a_i m_i$, so that $\langle e_i, m \rangle = a_i$. The character $\chi^m = \chi^{\sum_{i=1}^n a_i m_i} = \prod_{i=1}^n (\chi^{m_i})^{a_i}$ is the restriction of $x_1^{a_1}\cdots x_n^{a_n}$ to the torus, so that $\nu_\rho(\chi^m) = a_1 = \langle e_1, m \rangle = \langle u_\rho, m \rangle$. 
\end{proof}
As a consequence, we get the following elegant formula for the divisor of a character. 
\begin{corollary} \label{cor:divchar}
For $m \in M$, the divisor $\div(\chi^m)$ is given by $\sum_{\rho \in \Sigma(1)} \langle u_\rho, m \rangle \, D_\rho$.
\end{corollary}
This uses the fact that $\div(\chi^m)_{|T} = 0$, since $\chi^m: T \rightarrow \C^*$ is a morphism (Lemma \ref{lem:useful}). Divisors supported on $\bigcup_{\rho \in \Sigma(1)} D_\rho$ are called \emph{torus invariant divisors}. They form a free subgroup $\Div_T(X_\Sigma) = \bigoplus_{\rho \in \Sigma(1)} \Z \cdot D_\rho \subset \Div(X_\Sigma)$ of finite rank $k = |\Sigma(1)|$. 
\begin{theorem} \label{thm:classgroup}
Let $M \rightarrow \Div_T(X)$ be the map that sends $m \mapsto \div(\chi^m)$ and $\Div_T(X_\Sigma) \rightarrow \Cl(X_\Sigma)$ sends a divisor to its class. There is an exact sequence 
\begin{equation} \label{eq:firstES}
M \longrightarrow \Div_T(X_\Sigma) \longrightarrow \Cl(X_\Sigma) \longrightarrow 0. 
\end{equation}
Furthermore, this extends to a short exact sequence 
\begin{equation} \label{eq:secondES}
0 \longrightarrow M \longrightarrow \Div_T(X_\Sigma) \longrightarrow \Cl(X_\Sigma) \longrightarrow 0. 
\end{equation}
if and only if the rays of $\Sigma$ span $N_\R$, i.e.\ $X_\Sigma$ has no torus factor (Theorem \ref{thm:torusfactor}). 
\end{theorem}
\begin{proof}
Theorem \ref{thm:importantES} gives an exact sequence 
\[ \Div_T(X_\Sigma) \longrightarrow \Cl(X_\Sigma) \longrightarrow \Cl(T) \longrightarrow 0. \]
By Theorem \ref{thm:importantzeroclass}, $\Cl(T) = 0$, which proves exactness at $\Cl(X_\Sigma)$ in \eqref{eq:firstES}. We now show exactness at $\Div_T(X_\Sigma)$. The composition is clearly zero, since divisors of characters are principal. Suppose $[D] = 0$ for some $D \in \Div_T(X_\Sigma)$. Then $D = \div(f)$ and $\div(f)_{|T} = 0$. By Lemma \ref{lem:useful}, $f: T \rightarrow \C^*$ is a morphism, and hence $f = c \, \chi^m$ for some $m \in M$ and some $c \in \C^*$ (we used this in the proof of Proposition \ref{prop:characters}). We conclude that $D = \div(f) = \div(c \, \chi^m) = \div(\chi^m)$. Exactness of \eqref{eq:secondES} at $M$ if $\{u_\rho\}_{\rho \in \Sigma(1)}$ span $N_\R$ is an easy exercise.
\end{proof}

An important consequence of Theorem \ref{thm:classgroup} is that, when $X_\Sigma$ has no torus factors, $\Cl(X_\Sigma)$ can be computed using elementary linear algebra over $\Z$. After fixing coordinates, we may assume $M = N = \Z^n$. Let $F = [ u_1 ~ \cdots ~ u_k ] \in \Z^{n \times k}$ be the matrix whose columns are the primitive ray generators of the rays in $\Sigma(1)$. By Corollary \ref{cor:divchar}, the map $M \rightarrow \Div_T(X_\Sigma)$ is $F^\top: \Z^n \rightarrow \Z^k$. By Theorem \ref{thm:classgroup} the class group is given by $\Cl(X_\Sigma) \simeq \Z^k/\im F^\top$. 

\begin{example} \label{ex:divisorsP2}
If $X_\Sigma = \PP^2$, the matrix $F$ containing the rays in the fan of Figure \ref{fig:fanP2} is
\[ F = \begin{bmatrix}
1 &0 & -1 \\ 0 &1 & -1
\end{bmatrix}. \]
The class group is $\Z \cdot [D_1] + \Z \cdot [D_2] + \Z \cdot [D_3]$ with relations $[D_1] - [D_3] = 0$ (first row of $F$) and $[D_2] - [D_3] = 0$ (second row of $F$). It is generated by $[D_3]$, as $[a D_1 + b D_2 + c D_3] = (a+b+c)[D_3]$. We conclude $\Cl(\PP^2) \simeq \Z$.
\end{example}
\begin{example} \label{ex:divisorsP1P1}
If $X_\Sigma = \PP^1 \times \PP^1$, i.e.\ $\Sigma$ is the fan of Figure \ref{fig:fanP1P1}, we obtain 
\[ F = \begin{bmatrix}
1 & 0 & -1 & 0\\ 0 & 1 & 0 & -1
\end{bmatrix}. \]
The class group $\Cl(\PP^1 \times \PP^1)$ is isomorphic to $\Z^2$, and it is generated by $[D_1]$ and $[D_3]$, as $[a D_1 + b D_2 + c D_3 + d D_4] = (a+c)[D_1] + (b+d)[D_3]$.
\end{example}
\begin{exercise} \label{ex:doublepillow}
Consider the complete fan $\Sigma$ in $\R^2$ with ray generators given by 
\[ F = \begin{bmatrix}
1 &-1 &-1 &1\\
1 & 1 & -1 & -1
\end{bmatrix}.\]
This is the normal fan of a diamond. Show that $\Cl(X_\Sigma)$ is isomorphic to $\Z^2 \oplus \Z/2\Z$. 
\end{exercise}

\begin{examplestar} \label{ex:chiara8} 
The class group $\Cl(X_\Sigma)$ can be computed in \texttt{Oscar.jl} using the command \texttt{class\_group}. For $X_\Sigma$ from Example* \ref{ex:chiaramara} we find $\Cl(X_\Sigma) = \Z^2$, see Figure \ref{fig:output3}.
\begin{figure}
\centering
\includegraphics[scale=0.27]{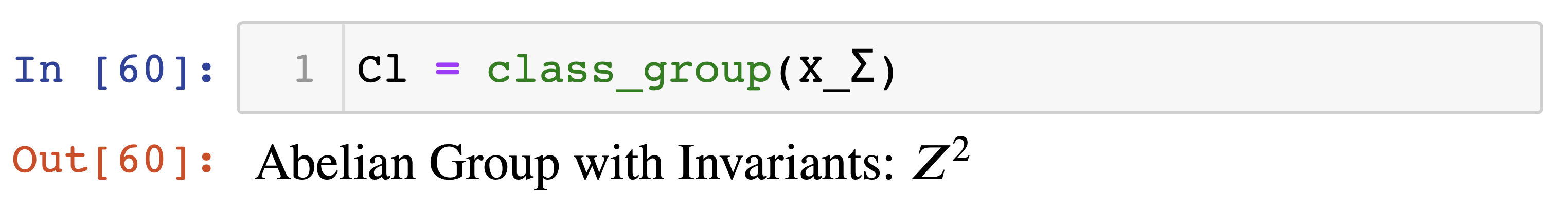}
\caption{Computing class groups in \texttt{Oscar.jl}.}
\label{fig:output3}
\end{figure}
\end{examplestar}

\subsection{Cartier divisors on toric varieties}
We now study Cartier divisors on $X_\Sigma$. We start with the analog of Theorem \ref{thm:classgroup}. Among the torus invariant divisors in $\Div_T(X_\Sigma)$, those that are Cartier form a subgroup denoted $\CDiv_T(X_\Sigma) = \Div_T(X_\Sigma) \cap \CDiv(X_\Sigma)$. Clearly, $\CDiv_T(X_\Sigma)$ contains all divisors of characters.
\begin{theorem}
Let $M \rightarrow \CDiv_T(X_\Sigma)$ be the map that sends $m \mapsto \div(\chi^m)$ and $\CDiv_T(X_\Sigma) \rightarrow \Pic(X_\Sigma)$ sends a divisor to its class. There is an exact sequence
\begin{equation} \label{eq:firstEScart}
M \longrightarrow \CDiv_T(X_\Sigma) \longrightarrow \Pic(X_\Sigma) \longrightarrow 0. 
\end{equation}
Furthermore, this extends to a short exact sequence 
\begin{equation} \label{eq:secondEScart}
0 \longrightarrow M \longrightarrow \CDiv_T(X_\Sigma) \longrightarrow \Pic(X_\Sigma) \longrightarrow 0. 
\end{equation}
if and only if the rays of $\Sigma$ span $N_\R$, i.e.\ $X_\Sigma$ has no torus factor (Theorem \ref{thm:torusfactor}). 
\end{theorem}
\begin{proof}
By Theorem \ref{thm:classgroup}, every divisor $D \in \CDiv(X_\Sigma)$ is linearly equivalent to a torus invariant divisor $\sum_{\rho \in \Sigma(1)} a_\rho D_\rho \in \CDiv_T(X_\Sigma)$, which shows exactness at $\Pic(X_\Sigma)$. Also, by Theorem \ref{thm:classgroup} those elements of $\CDiv_T(X_\Sigma)$ whose class is $[0]$ are precisely the ones coming from characters, which shows exactness at $\CDiv_T(X_\Sigma)$. Injectivity of $M \rightarrow \CDiv_T(X_\Sigma)$ under the extra assumption on the rays is an easy exercise.
\end{proof}

Our next goal is to describe the groups $\CDiv_T(X_\Sigma)$ and $\Pic(X_\Sigma)$ explicitly. It turns out this is quite easy if $X_\Sigma = U_\sigma$ is affine, see \cite[Prop.~4.2.2]{cox2011toric}.
\begin{proposition} \label{prop:picaffine}
Let $\sigma \subset N_\R$ be a strongly convex rational cone. We have $\Pic(U_\sigma) = 0$, i.e., $M \rightarrow \CDiv_T(X_\Sigma)$ is surjective. 
\end{proposition}
\begin{examplestar} \label{ex:chiara9} 
Let $\Sigma$ be the fan consisting of the 2-dimensional cone $\sigma = {\rm Cone}( (-1, -2), (1,0) ) \subset \R^2$ and all its faces, illustrated in Figure \ref{fig:torsionPic}. The corresponding toric variety $X_\Sigma = U_\sigma$ is affine. By Proposition \ref{prop:picaffine}, its Picard group is trivial: $\Pic(X_\Sigma) = 0$. The map $M \rightarrow \Div_T(X_\Sigma)$ is represented by the matrix of ray generators $F^\top = \begin{bmatrix}
-1 & 1 \\ -2 & 0
\end{bmatrix}$, whose image has index 2 in $\Z^2$. By Theorem \ref{thm:classgroup}, this implies that $\Cl(X_\Sigma) \simeq \Z/2 \Z$. We now remove the full-dimensional cone $\sigma$ from $\Sigma$, to obtain the fan $\Sigma'$ on the right side of Figure \ref{fig:torsionPic}. All cones in $\Sigma'$ are smooth, so that by Theorem \ref{thm:CarisWeil} we have $\Cl(X_{\Sigma'}) = \Pic(X_{\Sigma'})$. The class group only depends on the rays of the fan, so that $\Pic(X_{\Sigma'}) \simeq \Z/2 \Z$.
\begin{figure}
\centering
\begin{tikzpicture}[scale=1]
\begin{axis}[%
width=2in,
height=1in,
scale only axis,
xmin=-3.5,
xmax=3.5,
ymin=-3.5,
ymax=0.5,
ticks = none, 
ticks = none,
axis background/.style={fill=white},
axis line style={draw=none} 
]


\addplot [color=mycolor1,solid,fill opacity=0.2,fill = mycolor1,forget plot]
  table[row sep=crcr]{%
 5 0\\
5 -5\\	
-2.5 -5\\
0 0 \\
5 0 \\
};

\addplot [very thick, color=mycolor2,solid,fill opacity=0.2,fill = mycolor1,forget plot]
  table[row sep=crcr]{%
 0 0 \\
 5 0 \\
};

\addplot [very thick, color=mycolor2,solid,fill opacity=0.2,fill = mycolor1,forget plot]
  table[row sep=crcr]{%
 0 0 \\
 -2.5 -5 \\
};

\addplot[only marks,mark=*,mark size=1.1pt,black
        ]  coordinates {
(-2,0) (-1,0) (0,0) (1,0) (2,0) (3,0) 
(-2,-1) (-1,-1) (0,-1) (1,-1) (2,-1) (3,-1) 
(-2,-2) (-1,-2) (0,-2) (1,-2) (2,-2) (3,-2) 
(-2,-3) (-1,-3) (0,-3) (1,-3) (2,-3) (3,-3) 
};

\addplot[only marks,mark=*,mark size=3.1pt,mycolor4
        ]  coordinates {
  (0,0)
};

\end{axis}
\end{tikzpicture} 
\qquad 
\begin{tikzpicture}[scale=1]
\begin{axis}[%
width=2in,
height=1in,
scale only axis,
xmin=-3.5,
xmax=3.5,
ymin=-3.5,
ymax=0.5,
ticks = none, 
ticks = none,
axis background/.style={fill=white},
axis line style={draw=none} 
]


\addplot [very thick, color=mycolor2,solid,fill opacity=0.2,fill = mycolor1,forget plot]
  table[row sep=crcr]{%
 0 0 \\
 5 0 \\
};

\addplot [very thick, color=mycolor2,solid,fill opacity=0.2,fill = mycolor1,forget plot]
  table[row sep=crcr]{%
 0 0 \\
 -2.5 -5 \\
};

\addplot[only marks,mark=*,mark size=1.1pt,black
        ]  coordinates {
(-2,0) (-1,0) (0,0) (1,0) (2,0) (3,0) 
(-2,-1) (-1,-1) (0,-1) (1,-1) (2,-1) (3,-1) 
(-2,-2) (-1,-2) (0,-2) (1,-2) (2,-2) (3,-2) 
(-2,-3) (-1,-3) (0,-3) (1,-3) (2,-3) (3,-3) 
};

\addplot[only marks,mark=*,mark size=3.1pt,mycolor4
        ]  coordinates {
  (0,0)
};

\end{axis}
\end{tikzpicture} 
\caption{The fans $\Sigma$ and $\Sigma'$ from Example* \ref{ex:chiara9}.}
\label{fig:torsionPic}
\end{figure}
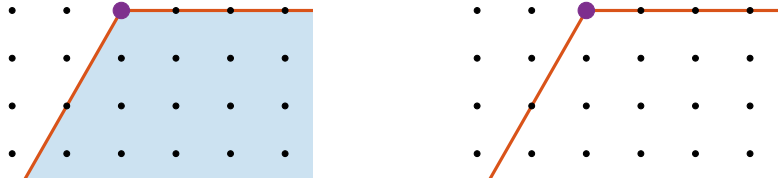

We verify this using \texttt{Oscar.jl} with the following commands.
\begin{minted}{julia}
Σ = positive_hull([-1 -2; 1 0]);
ΣΣ = PolyhedralFan([-1 -2; 1 0], IncidenceMatrix([[1],[2]]));
X_Σ = NormalToricVariety(Σ); X_ΣΣ = NormalToricVariety(ΣΣ);
Pic1 = picard_group(X_Σ); Pic2 = picard_group(X_ΣΣ);
\end{minted}
Note that $X_{\Sigma'}$ is obtained from $X_{\Sigma}$ by removing the torus fixed point (Example \ref{ex:removecones}).
\end{examplestar}

Example* \ref{ex:chiara9} shows that the Picard group of a normal toric variety may have torsion. The following proposition provides an easy criterion to exclude such a scenario. 

\begin{proposition}
If $\Sigma$ is a fan in $N_\R \simeq \R^n$ containing a cone of dimension $n$, then $\Pic(X_\Sigma)$ is a free abelian group.
\end{proposition}
\begin{proof}
By \eqref{eq:secondEScart}, it suffices to show that if $k D = \div(\chi^m)$ for some $D \in \CDiv_T(X_\Sigma), k \in \Z_{>0}$, then there exists $m' \in M$ such that $D = \div(\chi^{m'})$.

Let $D = \sum_{\rho \in \Sigma(1)} a_\rho D_\rho$ and $\sigma \in \Sigma(n)$. We have $D_{|U_\sigma} \in \CDiv_T(U_\sigma)$, so that
\[ D_{|U_\sigma} = \sum_{\rho \in \sigma(1)} a_\rho D_\rho = \div(\chi^{m'})_{|U_\sigma} \]
for some $m' \in M$, by Proposition \ref{prop:picaffine}. Therefore, by Corollary \ref{cor:divchar}, $a_\rho = \langle u_\rho, m' \rangle$ for $\rho \in \sigma(1)$. On the other hand, $k D = \div(\chi^m)$ implies $k a_\rho = \langle u_\rho, m \rangle$ for $\rho \in \Sigma(1)$. Since $\sigma$ has dimension $n$, the map $m \mapsto (\langle u_\rho, m \rangle)_{\rho \in \sigma(1)}$ is injective, so that
\[ \langle u_\rho, k m' \rangle = \langle u_\rho, m \rangle \quad \text{ for } \quad \rho \in \sigma(1) \]
implies $km' = m$ and $\div(\chi^{m'}) = \sum_{\rho \in \Sigma(1)} \langle u_\rho, m' \rangle D_\rho = \sum_{\rho \in \Sigma(1)} k^{-1} \langle u_\rho, m \rangle D_\rho = D$.
\end{proof}

If the fan $\Sigma$ is smooth, Theorem \ref{thm:CarisWeil} implies that $\Cl(X_\Sigma) = \Pic(X_\Sigma)$. This is an `if and only if' in our setting. 
\begin{proposition} \label{prop:smoothiffcarisweil}
The toric variety $X_\Sigma$ is smooth if and only if $\CDiv(X_\Sigma) = \Div(X_\Sigma)$. 
\end{proposition}
\begin{proof}
We show the `if' direction. Suppose $\CDiv(X_\Sigma) = \Div(X_\Sigma)$ and let $\sigma$ be any cone of $\Sigma$. The restriction $\Cl(X_\Sigma) \rightarrow \Cl(U_\sigma)$ is surjective by Theorem \ref{thm:importantES}, hence $\CDiv(U_\sigma) = \Div(U_\sigma)$. We also have $\Pic(U_\sigma) = 0$ (Proposition \ref{prop:picaffine}), so that $M \mapsto \Div_T(U_\sigma)$ is surjective. Choosing coordinates, this map is represented by a matrix $F^\top$ whose rows are the primitive ray generators $u_\rho$ for $\rho \in \sigma(1)$. Surjectivity of $F^\top$ means that $\{u_\rho\}_{\rho \in \sigma(1)}$ can be extended to a $\Z$-basis of $N$, which means that $\sigma$ is smooth.
\end{proof}
The proof of Proposition \ref{prop:smoothiffcarisweil} is easily adapted to show the following \cite[Prop.~4.2.7]{cox2011toric}.

\begin{proposition} \label{prop:simpfinind}
$\Pic(X_\Sigma) \subset \Cl(X_\Sigma)$ has finite index if and only if $\Sigma$ is simplicial.
\end{proposition}

\begin{examplestar} \label{ex:chiara10}
Let $\Sigma$ be the fan from Example* \ref{ex:chiaramara}. We check using \texttt{Oscar.jl} that $D_3$ is not Cartier as follows.
\begin{minted}{julia}
X_Σ = NormalToricVariety(Σ)
D3 = ToricDivisor(X_Σ, [0,0,1,0])
iscartier(D3)
\end{minted}
Since $\Sigma$ is simplicial, Proposition \ref{prop:simpfinind} guarantees that there exists $\ell \in \Z_{>0}$ such that $\ell \cdot D_3$ is Cartier. We compute $\ell$ using \texttt{Oscar.jl}. The output of the following code
\begin{minted}{julia}
i = 0; l = 1;
while i < 1
    DD = iscartier(ToricDivisor(X_Σ, [0,0,l,0]))
    if DD == false
        global l += 1
    else
        print("l = ", l)
        i = 1
    end
end
\end{minted}
is $\ell = 6$, hence $6\cdot D_3$ is Cartier.
\end{examplestar}

Let $D \in \CDiv_T(X_\Sigma)$ be a torus invariant Cartier divisor. On each affine toric open subset $U_\sigma \subset X_\Sigma$, $D_{|U_\sigma}$ is a principal divisor corresponding to a character (Proposition \ref{prop:picaffine}). We now investigate which characters $m_\sigma \in M$ are such that $D_{|U_\sigma} = \div(\chi^{m_\sigma})_{|U_\sigma}$.

\begin{theorem} \label{thm:cartiercriterion}
Let $D = \sum_{\rho \in \Sigma(1)} a_\rho D_\rho \in \Div_T(X_\Sigma)$ be a Weil divisor. The following are equivalent: 
\begin{enumerate}
\item[(a)] $D$ is Cartier, 
\item[(b)] $D$ is principal on $U_\sigma$ for all $\sigma \in \Sigma$, 
\item[(c)] For each $\sigma \in \Sigma$, there is $m_\sigma \in M$ such that $\langle u_\rho, m_\sigma \rangle + a_\rho = 0$, $\rho \in \sigma(1)$,
\item[(d)] For each $\sigma \in \Sigma_{\max}$, there is $m_\sigma \in M$ such that $\langle u_\rho, m_\sigma \rangle + a_\rho = 0$, $\rho \in \sigma(1)$.
\end{enumerate}
Here $\Sigma_{\max}$ is the set of maximal cones of $\sigma$ with respect to inclusion. Moreover, if $D \in \CDiv_T(X_\Sigma)$ and the lattice points $\{ m_\sigma \}_{\sigma \in \Sigma}$ are as in (c), then 
\begin{enumerate}
\item[(1)] $m_\sigma$ is unique modulo $M(\sigma) = \sigma^\perp \cap M$, and
\item[(2)] if $\tau \preceq \sigma$ is a face, then $m_\sigma = m_\tau$ modulo $M(\tau)$. 
\end{enumerate}
\end{theorem}
\begin{proof}
The implications (a) $\Leftrightarrow$ (b) $\Leftrightarrow$ (c) $\Rightarrow$ (d) follow from Proposition \ref{prop:picaffine}. To show (d) $\Rightarrow$ (c), observe that if $\sigma$ is maximal and $m_\sigma$ is as in $(d)$, then for $\tau \preceq \sigma$ one chooses $m_\tau = m_\sigma$. 

For (1), suppose $\langle u_\rho, m_\sigma \rangle = \langle u_\rho, m_{\sigma'} \rangle = - a_\rho$ for all $\rho \in \sigma(1)$. Then $\langle u_\rho, m_\sigma - m_{\sigma'} \rangle = 0$ for all $\rho \in \sigma(1)$ means that $m_\sigma - m_{\sigma'} \in \sigma^\perp \cap M = M(\sigma)$. Point (1) and our observation in part (d) imply that, when $\tau \preceq \sigma$, $m_\sigma = m_\tau$ modulo $M(\tau)$, which is (2). 
\end{proof}

In the notation of Theorem \ref{thm:cartiercriterion}, the torus invariant Cartier divisor $D$ is locally given by $D_{|U_\sigma} = \div(\chi^{-m_\sigma})$. The minus sign comes from our notation \eqref{eq:Hrep1} for facet inequalities of polytopes. The tuples $\{ (U_\sigma, m_\sigma) \}_{\sigma \in \Sigma}$ form the \emph{Cartier data} for $D$, as they encode how $D$ is locally principal. Note also that part (1) of the theorem implies that $m_\sigma \in M$ is uniquely determined if and only if $\dim \sigma = n$.

We can regard $m_\sigma$ as a unique element of $M/M(\sigma)$. If $\tau \preceq \sigma$, the canonical map $M/M(\sigma) \rightarrow M/M(\tau)$ sends $m_\sigma$ to $m_\tau$. This gives an explicit descripition of $\CDiv_T(X_\Sigma)$ as follows (see \cite[Prop.~4.2.9]{cox2011toric}). Let $\Sigma_{\max} = \{\sigma_1, \ldots, \sigma_r \}$ and define the map
\[ \phi : \bigoplus_{i=1}^r M /M(\sigma_i) \rightarrow \bigoplus_{1 \leq i < j \leq r} M/M(\sigma_i \cap \sigma_j), \quad \text{given by} \quad (m_i)_i \mapsto (m_i - m_j)_{i,j}. \]
\begin{proposition}
With the notation introduced above, we have $\CDiv_T(X_\Sigma) \simeq \ker \phi$. 
\end{proposition}

\begin{example}[The divisor of a polytope] \label{ex:divfrompol}
Let $P = \{ m \in M_\R ~|~ \langle u_Q, m \rangle + a_Q \geq 0, Q \in {\cal Q} \}$ be a full-dimensional polytope in $\R^n$ whose set of facets is ${\cal Q}$ and $u_Q$ is the primitive inward pointing normal vector to $Q \in {\cal Q}$. The toric variety associated to $P$ is $X_P \simeq X_{\Sigma_P}$, where $\Sigma_P$ is the normal fan of $P$. The maximal cones $\sigma_v \in (\Sigma_P)_{\max}$ correspond to the vertices $v \in P$, and a vertex $v \in P$ satisfies 
\begin{equation} \label{eq:vertexeq}
\langle u_Q, v \rangle + a_Q = 0, \quad \text{for all } Q \ni v. 
\end{equation}
Hence $\{ (U_{\sigma_v}, v) \}_{v \text{ vertex of } P}$ are the Cartier data for $D_P = \sum_{Q \in {\cal Q}} a_Q D_Q \in \Div_T(X_{\Sigma_P})$. 

Translating our polytope $P$ by a lattice point $m$ does not change its normal fan. Therefore, as pointed out in Section \ref{sec:normalfans}, $P$ and $P + m$ give the same toric variety $X_P \simeq X_{P + m} \simeq X_{\Sigma_P}$. However, they give different divisors on it. In fact, one checks easily that $D_{P + m} = D_P + \div(\chi^m)$, so that $[D_P] = [D_{P + m}]$. By Theorem \ref{thm:classgroup}, the divisor class $[D_P]$ corresponds to such translates of $P$. 

As an explicit example, we consider the triangle in Figure \ref{fig:triang}, whose normal fan is shown in Figure \ref{fig:fanP2}.
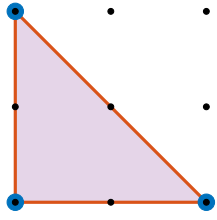
\begin{figure}
\centering
\begin{tikzpicture}[scale=1]
\begin{axis}[%
width=1.5in,
height=1.5in,
scale only axis,
xmin=-0.5,
xmax=2.5,
ymin=-0.5,
ymax=2.5,
ticks = none, 
ticks = none,
axis background/.style={fill=white},
axis line style={draw=none} 
]


\addplot [color=mycolor1,solid,fill opacity=0.2,fill = mycolor4,forget plot]
  table[row sep=crcr]{%
 0 0\\
0 2\\	
2 0\\
0 0 \\
};

\addplot [very thick, color=mycolor2,solid,fill opacity=0.2,fill = mycolor1,forget plot]
  table[row sep=crcr]{%
 0 0 \\
 0 2\\
};

\addplot [very thick, color=mycolor2,solid,fill opacity=0.2,fill = mycolor1,forget plot]
  table[row sep=crcr]{%
 0 0 \\
 2 0 \\
};

\addplot [very thick, color=mycolor2,solid,fill opacity=0.2,fill = mycolor1,forget plot]
  table[row sep=crcr]{%
 0 2\\
 2 0 \\
};

\addplot[only marks,mark=*,mark size=3.1pt,mycolor1
        ]  coordinates {
  (0,0) (2,0) (0,2)
};


\addplot[only marks,mark=*,mark size=1.1pt,black
        ]  coordinates {
   (0,0) (1,0) (2,0) (0,1) (1,1) (2,1) (0,2) (1,2) (2,2)
};

\end{axis}
\end{tikzpicture} 
\caption{A triangle in $\R^2$ whose normal fan is shown in Figure \ref{fig:fanP2}.}
\label{fig:triang}
\end{figure}
The three blue vertices correspond to the three blue maximal cones. If the lower left corner of the triangle is $(0,0)$, we see that 
\[ P = \left \{ m \in \R^2 ~ \Big | ~ \begin{bmatrix}
1 & 0 \\ 0 & 1 \\ -1 & -1 
\end{bmatrix} \begin{bmatrix}
m_1 \\ m_2
\end{bmatrix} + \begin{bmatrix}
0 \\ 0 \\ 2
\end{bmatrix} \geq 0 \right \}. \]
This means $D_P = 2 D_3$, where $D_3$ is as in Example \ref{ex:divisorsP2}. The vertex $v = (2,0)$ of $P$ corresponds to the cone $\sigma_1 \in \Sigma$. This cone contains the rays spanned by rows 1 and 3 in $F^\top$. We have
\[ \begin{bmatrix}
0 & 1 \\ -1 & -1 
\end{bmatrix} \begin{bmatrix}
v_1 \\ v_2
\end{bmatrix} + \begin{bmatrix}
0 \\ 2
\end{bmatrix} = 0, \]
which is \eqref{eq:vertexeq} for this vertex. For the ray $\rho = \R_{\geq 0} \cdot (-1,-1)$, one can choose $m_\rho$ to be any of the lattice points on the affine line containing the edge ${\rm Conv}((2,0),(0,2))$ of $\rho$. 
\end{example}

\begin{exercise}
Compute the coefficients of $D_P = a_1 D_1 + \cdots + a_4 D_4$, where the $D_i$ are as in Example \ref{ex:divisorsP1P1} and $P$ is the rectangle $[0,2] \times [0,7] \subset \R^2$.
\end{exercise}

\subsection{Sheaves of divisors}
The \emph{structure sheaf} of a normal variety $X$ is the sheaf of $\OO_X$-modules given by 
\[ U \longmapsto \OO_X(U) = \{ f \in \C(X)^* ~|~ \div(f)_{|U} \geq 0 \}. \]
A Weil divisor $D$ on $X$ defines a twisted version $\OO_X(D)$ of this sheaf, such that 
\[ U \longmapsto \OO_X(D)(U) = \{ f \in \C(X)^*~|~ (\div(f) + D)_{|U} \geq 0 \}. \]
To avoid double parentheses, we will use the notation $\Gamma(U,\OO_X(D)) = \OO_X(D)(U)$. These rational functions are called \emph{sections} of the sheaf $\OO_X(D)$ on $U$, and $\Gamma(X, \OO_X(D))$ consists of the \emph{global sections} of $\OO_X(D)$. For normal toric varieties, they have a nice, explicit description. 

Here is a lemma that we need to prove the main result of this section, see \cite[Lem.~1.1.16]{cox2011toric}. It involves a torus $T$ and its character lattice $M$.

\begin{lemma} \label{lem:stablesubspace}
Consider the action of $T$ on $\C[M]$ given by 
$t \cdot f = ( p \mapsto f(t^{-1} \cdot p))$. If $V \subset \C[M]$ is a $\C$-subspace which is stable under this action, then 
\[ V = \bigoplus_{\chi^m \in V} \C \cdot \chi^m. \]
\end{lemma}

\begin{proposition} \label{prop:globalsec}
Let $\Sigma$ be a fan in $N_\R \simeq \R^n$ and let $D \in \Div_T(X_\Sigma)$. We have 
\[ \Gamma(X_\Sigma, \OO_{X_\Sigma}(D)) = \bigoplus_{\div(\chi^m) + D \geq 0 } \C \cdot \chi^m. \]
\end{proposition}
\begin{proof}
A rational function $f$ belongs to $\Gamma(X_\Sigma, \OO_{X_\Sigma}(D))$ if and only if $\div(f) + D \geq 0$. Then in particular $(\div(f) + D)_{|T} \geq 0$, and $f \in \C[M]$ by Lemma \ref{lem:useful}. Hence $\Gamma(X_\Sigma, \OO_{X_\Sigma}(D)) \subset \C[M]$ is a subspace. It is stable under the action of $T$ on $\C[M]$ since the divisor $D$ is torus invariant. The statement now follows from Lemma \ref{lem:stablesubspace}.
\end{proof}
Let $D  = \sum_{\rho \in \Sigma(1) } a_\rho D_\rho \in \Div_T(X_\Sigma)$. The direct sum in Proposition \ref{prop:globalsec} is over all lattice points $m \in M$ such that $\div(\chi^m) + D$ is effective. By Corollary \ref{cor:divchar}, these are the lattice points in the polyhedron
\[ P_D = \{ m \in M_\R ~|~ \langle u_\rho, m \rangle + a \geq 0 \} .\] 
In Example \ref{ex:divfrompol} we saw how to associate a divisor $D_P$ to a polytope $P$. This section shows how to associate a polyhedron $P_D$ (i.e., a finite intersection of half spaces which is not necessarily bounded) to a torus invariant divisor $D$. The vertices of this polyhedron are not necessarily lattice points, see Exercise \ref{ex:PDnonlattice}. The following example and Exercise \ref{ex:invdivfrompol} show how the constructions $D \mapsto P_D$ and $P \mapsto D_P$ are, to some extent, inverse to each other.
\begin{example} \label{ex:grading0}
Let $X_\Sigma = \PP^2$ and $D = 2 D_3$, where $D_3$ is as in Example \ref{ex:divisorsP2}. One checks easily that the polyhedron $P_D$ is the triangle from Figure \ref{fig:triang}. By Proposition \ref{prop:globalsec}, the global sections of $\OO_{\PP^2}(D)$ are Laurent polynomials supported in this triangle. These correspond to rational functions on $\PP^2$ with poles of order up to 2 along $D_3$, and no poles elsewhere. If $D_3$ is given by the vanishing of the homogeneous coordinate $x_3$, these are of the form 
\[ \frac{a_0 x_3^2 + a_1 x_1x_3 + a_2 x_2x_3 + a_3 x_1x_2 + a_4 x_1^2 + a_5 x_2^2}{x_3^2} \sim a_0 + a_1 t_1  + a_2 t_2 + a_3 t_1t_2 + a_4 t_1^2 + a_5 t_2^2.\]
Note that here $X_{\Sigma} = X_{P_D}$ and $D_{P_D} = D$. See Exercise \ref{ex:invdivfrompol} for a general statement. 
\end{example}
If, in Example \ref{ex:grading0}, we use $D' = 2D_1$ instead of $D = 2D_3$, $P_{D'}$ is a translated version of the triangle in Figure \ref{fig:triang}. The vector space $\Gamma(X,\OO_X(D'))$ is isomorphic to $\Gamma(X,\OO_X(D))$, via multiplication with $\chi^{(2,0)}$. This is an instance of the following fact \cite[Prop.~4.0.29]{cox2011toric}. 
\begin{proposition}
If $D, D' \in \Div_T(X_\Sigma)$ are linearly equivalent, i.e.\ $[D - D'] = 0$, we have $\Gamma(X_\Sigma, \OO_{X_\Sigma}(D)) \simeq \Gamma(X_\Sigma, \OO_{X_\Sigma}(D'))$.
\end{proposition}
We end the section by hinting at a nice connection with \emph{tropical geometry} \cite{maclagan2009introduction}.
\begin{example}[Tropical hypersurfaces] \label{ex:tropical}
Let $P \subset \R^n$ be a full-dimensional convex lattice polytope and let $D_P$ be the corresponding divisor on $X_P$, see Example \ref{ex:divfrompol}. The vector space $\Gamma(X_P, \OO_{X_P}(D_P))$ consists of all Laurent polynomials whose exponents are contained in $P$. Let $f \in \Gamma(X_P, \OO_{X_P}(D_P))$ be such a Laurent polynomial. We write $f$ as 
\[ f = \sum_{m \in \A} c_m \, t^m, \quad \text{ with } \A \subset P \cap M \text{ and } c_m \in \C^*. \]
This defines a hypersurface $V(f) \in (\C^*)^n$ given by $V(f) = \{ t \in (\C^*)^n ~|~ f(t) = 0 \}$. The \emph{initial form} of $f$ with respect to a weight vector $w \in \R^n$ is 
\[ {\rm in}_w(f) = \sum_{m \in \A_w} c_m \, t^m, \quad \text{ with } \A_w = \{ m \in \A ~|~ \langle w, m \rangle = \min_{m' \in \A} \langle w, m' \rangle \}. \]
In tropical geometry, $f$ defines a \emph{tropical hypersurface} in $\R^n$ given by 
\[ {\rm Trop}(V(f)) = \{ w \in \R^n ~|~ {\rm in}_w(f) \text{ has $\geq 2$ terms} \}. \]
One easily checks that if all vertices of $P$ belong to $\A$, ${\rm Trop}(V(f))$ is the support of the $(n-1)$-skeleton $\Sigma_{P,n-1} = \{ \sigma \in \Sigma_P ~|~ \dim \sigma \leq n-1 \}$ of the normal fan $\Sigma_P$. In this case, the \emph{tropical compactification} of the hypersurface $V(f)$ is its closure in $X_{\Sigma_{P,n-1}}$. Fore more details on such compactifications and their desirable properties, see \cite[Sec.~1.8 and 6.4]{maclagan2009introduction}.
\end{example}
\begin{exercise}
With the notation of Example \ref{ex:divisorsP1P1}, describe $\Gamma(\PP^1 \times \PP^1, \OO_{\PP^1 \times \PP^1}(2 D_1 + 5D_3))$. 
\end{exercise}
\begin{exercise} \label{ex:PDnonlattice}
Show that $D = D_1 + D_2 \in \Div_T(X_\Sigma) \setminus \CDiv_T(X_\Sigma)$, where $X_\Sigma$ is the toric surface from Exercise \ref{ex:doublepillow}. Describe $P_D$ and $\Gamma(X_\Sigma,\OO_{X_\Sigma}(D))$. 
\end{exercise}
\begin{exercise} \label{ex:invdivfrompol}
Show that if $X_\Sigma \simeq X_{P_D}$ for some divisor $D \in \Div_T(X_\Sigma)$ such that $P_D$ is full-dimensional, then $D_{P_D} = D$. Conversely, if $P$ is a full-dimensional lattice polytope and $D_P \in \CDiv_T(X_P)$ is the corresponding divisor, then $P_{D_P} = P$.  
\end{exercise}
\section{Homogeneous coordinates on toric varieties} \label{sec:homogeneous}
Fix a fan $\Sigma$ in $N_\R$ and let $X_\Sigma$ be the corresponding normal toric variety. We will sometimes drop the index $\Sigma$ and write $X = X_\Sigma$ to simplify the notation where there is no danger for confusion. In this section we describe $X$ as the quotient of a quasi-affine space by the action of a reductive group. This allows us to define global coordinates on $X$, which are useful for describing its subvarieties. The construction generalizes the well-known homogeneous coordinates on $\PP^n$. Homogeneous coordinates correspond to the variables of the \emph{homogeneous coordinate ring} or \emph{Cox ring} of $X$, which is a polynomial ring endowed with a particular grading and a distinguished ideal called the \emph{irrelevant ideal}. The discussion follows \cite[Sec.~5.4-5.5]{telen2020thesis}. For more details, the reader can consult \cite[Ch.~5]{cox2011toric}. We start with an example which identifies the relevant geometric and algebraic objects for the projective plane $X = \PP^2$. 

\begin{example} \label{ex:P2Cox}
The projective plane $\PP^2$ is defined as
\begin{equation} \label{eq:coxP2}
\PP^2 = \frac{\C^3 \setminus \{0\}}{\C^*},
\end{equation}
where the quotient is by the action $\C^* \times (\C^3 \setminus \{0\}) \rightarrow (\C^3 \setminus \{0\})$ given by $(\lambda, (x_1,x_2,x_3)) \mapsto (\lambda x_1, \lambda x_2, \lambda x_3)$. This action extends trivially to an action on $\C^3$. Subvarieties of $\PP^2$ are given by homogeneous ideals in the polynomial ring $S = \C[x_1,x_2,x_3]$. Here `homogeneous' is with respect to the $\Z$-grading 
$$ S = \bigoplus_{\alpha \in \Z} S_\alpha, $$
which is such that for $f \in S$ homogeneous, $V_{\C^3}(f)$ is stable under the $\C^*$-action. Equivalently, $V_{\C^3}(f)$ is a union of $\C^*$-orbits. In the ring $S$, the ideal $B = \langle x_1,x_2,x_3 \rangle$ plays a special role: its variety in $\PP^2$ is the empty set. The interplay between algebra and geometry in this construction is summarized in the following table. 
\[
\begin{matrix}
\text{Algebra} & & \text{Geometry} \\ \hline
S & \overset{\Specm(\cdot)}{\longrightarrow} & \C^3 \\
\Irrel & \overset{V_{\C^3}(\cdot)}{\longrightarrow} & \{0\} \\
\Z & \overset{\Hom_\Z(\cdot,\C^*)}{\longrightarrow} & \C^*
\end{matrix}
\]
For the purpose of generalizing this construction, we make the following observation. The quotient \eqref{eq:coxP2} comes from a toric morphism $\pi: \C^3 \setminus \{0\} \rightarrow \PP^2$ which is constant on $\C^*$-orbits. A toric morphism comes from a lattice homomorphism $N' \rightarrow N$ that is compatible with fans $\Sigma'$ and $\Sigma$ in $N_\R'$ and $N_\R$ respectively (see Section \ref{sec:tormorph}). In our case $\Sigma'$ is the fan of $\C^3 \setminus \{0\}$ and $\Sigma$ is the fan of $\PP^2$. The lattices are $N' = \Z^3$ and $N = \Z^2$, and the morphism $\pi$ comes from $F: N' \rightarrow N$ where $F$ is a $2 \times 3$ integer matrix whose columns are the primitive ray generators of $\Sigma(1)$. The fans and the matrix $F$ are shown in Figure \ref{fig:latticehom}.
\begin{figure} 
\centering
\tdplotsetmaincoords{60}{120}
\begin{tikzpicture}[baseline=(O.base),scale=1.5,tdplot_main_coords]
\coordinate (O) at (0,0,0);
\coordinate (A) at (2,0,0);
\coordinate (B) at (0,2,0);
\coordinate (C) at (0,0,2);
\coordinate (X) at (1.5,0,0);
\coordinate (Y) at (0,1.5,0);
\coordinate (Z) at (0,0,1.5);
							
\draw[opacity=0.2,fill = mycolor2,] (O)--(A)--(C)--cycle;
\draw[opacity=0.2,fill = mycolor3,] (O)--(A)--(B)--cycle;
\draw[opacity=0.2,fill = mycolor1,] (O)--(B)--(C)--cycle;
							
\draw[-latex, very thick, red] (O) -- (X) node[anchor=west] {};
\draw[-latex, very thick, red] (O) -- (Y) node[anchor=south] {};
\draw[-latex, very thick, red] (O) -- (Z) node[anchor=south] {};
			
\draw[-latex] (O) -- (A) node[anchor=west] {};
\draw[-latex] (O) -- (B) node[anchor=south] {};
\draw[-latex] (O) -- (C) node[anchor=south] {};
			
\draw (0,0.5,0.5) node{};
\draw (0.5,0,0.5) node{};
\draw (0.5,0.5,0) node{};
\end{tikzpicture} 
\Large $\xrightarrow{\quad F = \normalsize \begin{bmatrix}
1 & 0 & -1\\ 0 & 1 & -1
\end{bmatrix} \quad}$ \quad \normalsize
\begin{tikzpicture}[baseline=-1 cm,scale=1.0]
\coordinate (O) at (0,0);
\coordinate (u1) at (1,0);
\coordinate (uu1) at (2,0);
\coordinate (u12) at (2,-2);
\coordinate (u2) at (-1,-1);
\coordinate (uu2) at (-2,-2);
\coordinate (u23) at (-2,2);
\coordinate (u3) at (0,1);
\coordinate (uu3) at (0,2);
\coordinate (u31) at (2,2);
\coordinate (X) at (1.5,0);
\coordinate (Y) at (0,1.5);
\draw[-latex] (O) -- (X) node[anchor=west]{} ;
\draw[-latex] (O) -- (Y) node[anchor=south]{} ;
\draw[-latex, very thick, red] (O) -> (u1); 
\draw[-latex, very thick, red] (O) -> (u2);
\draw[-latex, very thick, red] (O) -> (u3); 
\draw[opacity=0.2,fill = mycolor1,] (O)--(uu1)--(u12)--(uu2)--cycle;
\draw[opacity=0.2,fill = mycolor2,] (O)--(uu2)--(u23)--(uu3)--cycle;
\draw[opacity=0.2,fill = mycolor3,] (O)--(uu3)--(u31)--(uu1)--cycle;
\draw (0.6,-1) node{};
\draw (-1,0.5) node{};
\draw (1,1) node{};
\end{tikzpicture} 
\caption{An illustration of the $\Z$-linear map $F: N' \rightarrow N$ from Example \ref{ex:P2Cox}. The ray generators of $\Sigma'(1), \Sigma(1)$ are depicted as red arrows and the two dimensional cones are colored in blue, orange and yellow. 
}
\label{fig:latticehom}
\end{figure}
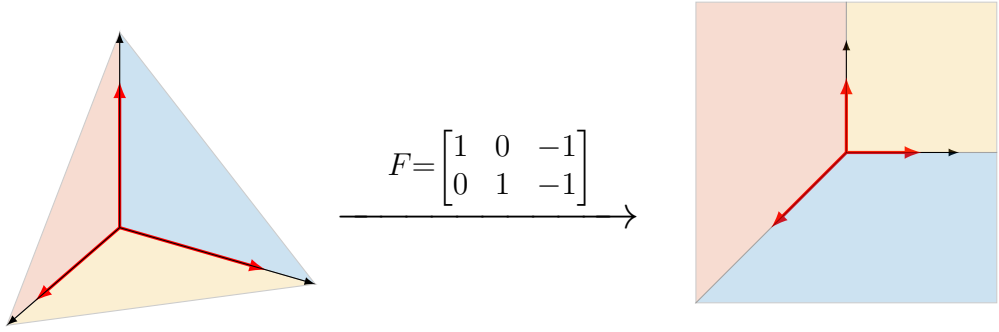
The \emph{compatibility} of the map $F$ with the fans $\Sigma'$ and $\Sigma$ comes down to the fact that each cone of $\Sigma'$ is mapped (under the $\R$-map $F_\R = F \otimes_\Z \R$ associated to $F$) into a cone of $\Sigma$. In Figure \ref{fig:latticehom} the 2-dimensional cones have matching colors according to this association. Note that the three dimensional cone $\sigma = {\rm Cone}(e_1,e_2,e_3)$ of the positive orthant in $\R^3$ is not mapped to a cone of $\Sigma$ under $F_\R$. Therefore, this cone does not belong to $\Sigma'$. Taking this three dimensional cone out of $\Sigma'$ corresponds to taking the origin out of $\C^3$ (Example \ref{ex:removecones}). Hence $\C^3 \setminus \{0\} = X_{\Sigma'}$.
\end{example}

\subsection{Toric varieties as GIT quotients} \label{subsec:coxconstr}
In this subsection we describe the construction of a toric variety $X$ as the image of a map
$$ \pi: \C^k \setminus Z \rightarrow X,$$
where $Z \subset \C^k$ is a subvariety and $\pi$ is invariant under an algebraic group action $G \times (\C^k \setminus Z) \rightarrow (\C^k \setminus Z)$. This construction is due to Cox \cite{cox1995homogeneous}. It generalizes Example \ref{ex:P2Cox} to toric varieties $X = X_\Sigma$ without torus factors. In the case of $X = \PP^n$,  we have $k = n+1$, $Z = \{0\}$ and $G = \C^*$. More generally, we will see that
\begin{itemize}
\item $k = |\Sigma(1)|$ is the number of rays in $\Sigma$, 
\item $Z$ is a union of coordinate subspaces,
\item $G$ is a quasi-torus acting algebraically on $\C^k \setminus Z$.
\end{itemize}
We should mention that the result had been described in the analytic category by Audin, Delzant and Kirwan, see \cite[Chapter 6]{audin2012topology} and references therein. 

In what follows, it is instructive to keep Example \ref{ex:P2Cox} in mind as a reference. Let $\Sigma(1) = \{\rho_1, \ldots, \rho_k\}$ and let $u_i \in N$ be the primitive ray generator of $\rho_i$. We collect the $u_i$ in a matrix 
$$ F = [u_1 ~ \cdots ~ u_k] \in \Z^{n \times k}.$$
This gives a lattice homomorphism $F: N' \rightarrow N$ where $N' = \Z^k$. Consider the fan given by the positive orthant in $\R^k$ and all its faces. We let $\Sigma'$ be the subfan of all the cones whose image under $F_\R$ is contained in a cone of $\Sigma$. By construction, the lattice homomorphism $F$ is compatible with the fans $\Sigma'$ and $\Sigma$ in $N_\R'$ and $N_\R$ respectively. It follows that $F$ gives a toric morphism $\pi: X_{\Sigma'} \rightarrow X_\Sigma$, where $X_{\Sigma'} = \C^k \setminus Z$ and $Z$ is a union of coordinate subspaces. We now give a description of $Z$ as the affine variety of a radical monomial ideal. Let $S =  \C[x_1,\ldots,x_k]$ be the coordinate ring of $\C^k$ and for each $\sigma \in \Sigma$, consider the monomial 
\begin{equation} \label{eq:xsigmahat}
x^{\hat{\sigma}} = \prod_{\rho_i \not \subset \sigma} x_i,
\end{equation}
where the product ranges over all $i \in \{1, \ldots, k\}$ such that $\rho_i \not \subset \sigma$.
\begin{proposition}
The subvariety $Z \subset \C^k$ for which $X_{\Sigma'}  = \C^k \setminus Z$ is $Z = V_{\C^k}(\Irrel)$ with 
\begin{equation} \label{eq:irrelideal}
 \Irrel = \left \langle x^{\hat{\sigma}} ~|~ \sigma \in \Sigma  \right \rangle = \left \langle x^{\hat{\sigma}} ~|~ \sigma \in \Sigma_{\max} \right \rangle.
 \end{equation}
\end{proposition}
\begin{proof}
See \cite[Prop.~5.1.9]{cox2011toric}. Note that $\left \langle x^{\hat{\sigma}} ~|~ \sigma \in \Sigma  \right \rangle = \left \langle x^{\hat{\sigma}} ~|~ \sigma \in \Sigma_{\max} \right \rangle$ simply follows from the fact that $x^{\hat{\sigma}}$ for $\sigma \in \Sigma_{\max}$ are the minimal generators.
\end{proof}
A geometric description of $Z$ as a union of coordinate subspaces uses \emph{primitive collections}.
\begin{definition}[Primitive collection]
A subset $C \subset \Sigma(1)$ of rays is a \emph{primitive collection} if 
\begin{enumerate}
\item $C \not \subset \sigma(1)$ for all $\sigma \in \Sigma$ and 
\item for every proper subset $C' \subsetneq C$, there is $\sigma \in \Sigma$ for which $C' \subset \sigma(1)$. 
\end{enumerate}
\end{definition}
\begin{proposition} \label{prop:primcol}
The subvariety $Z \subset \C^k$ for which $X_{\Sigma'}  = \C^k \setminus Z$ is 
\[ Z = \bigcup_{C \text{ primitive collection}} V_{\C^k}(~x_i ~|~ \rho_i \in C~). \]
\end{proposition}
\begin{proof}
See \cite[Prop.~5.1.6]{cox2011toric}.
\end{proof}
\begin{example} \label{ex:baselocusP2}
Let $X = \PP^2$ and let $F$ be as in Example \ref{ex:P2Cox}. With $\sigma_i$ as in Figure \ref{fig:fanP2} we have 
\[ x^{\hat{\sigma}_0} = x_3, \quad  x^{\hat{\sigma}_1} = x_1, \quad  x^{\hat{\sigma}_2} = x_2, \]
and hence $B = \langle x_1, x_2, x_3 \rangle$ as expected. The only primitive collection is $C = \Sigma(1)$. 
\end{example}
\begin{exercise} \label{ex:baselocusP1P1}
With $F$ as in Example \ref{ex:divisorsP1P1}, show that for $X = \PP^1 \times \PP^1$ we have $B = \langle x_1x_2, x_1x_4, x_3x_2, x_3x_4 \rangle$. Identify the primitive collections and verify Proposition \ref{prop:primcol}.
\end{exercise}
The morphism $\pi$ is an extension of a map of tori $\pi_{|(\C^*)^k}: (\C^*)^k \rightarrow (\C^*)^n$, which has an easy description based on the matrix $F$. It is given by the Laurent monomial map
\begin{equation} \label{eq:monmap}
\pi |_{(\C^*)^k} = F \otimes_\Z \C^* : (\C^*)^k \rightarrow (\C^*)^n \quad \text{where} \quad (t_1,\ldots,t_k) \mapsto (t^{F_{1,:}}, \ldots, t^{F_{n,:}}).
\end{equation}
This uses the short notation $z^a = z_1^{a_i}\cdots z_k^{a_k}$ and $F_{i,:}$ for the $i$-th row of $F$. The kernel of $\pi |_{(\C^*)^k}$ (as a group homomorphism) is given by 
\begin{equation} \label{eq:G}
G = \{ g \in (\C^*)^k ~|~ g^{F_{1,:}} = \cdots = g^{F_{n,:}} = 1 \}.
\end{equation}
This is a subgroup $G \subset (\C^*)^k$ which acts on $\C^k$ by 
$$ (g_1, \ldots, g_k) \cdot (x_1, \ldots, x_k) \mapsto (g_1x_1, \ldots, g_k x_k)$$
(this is the restriction of the action of $(\C^*)^k$ on $\C^k$ to $G$). 
\begin{example}
For $X = \PP^2$, $G = \{ (g_1,g_2,g_3) \in (\C^*)^3 ~|~ g_1 g_3^{-1} = g_2 g_3^{-1} = 1 \} = \{ (\lambda,\lambda,\lambda)  ~|~ \lambda \in \C^* \} \simeq \C^*$, and the action is the usual coordinate-wise scaling. 
\end{example}
The morphism $\pi$ is constant on $G$-orbits in $\C^k \setminus Z$. Indeed, by equivariance (Proposition \ref{prop:equivariance}), $\pi(g \cdot x) = \pi(g) \cdot \pi(x) = \pi(x)$ for all $g \in G$.  The following theorem uses some terminology for geometric invariant theory (GIT) quotients from \cite[Sec.~5.0]{cox2011toric}. 
\begin{theorem} \label{thm:cox}
The morphism $\pi: \C^k \setminus Z \rightarrow X_\Sigma$ coming from $F = [u_1 ~ \cdots ~ u_k]$ is an almost geometric quotient for the action of $G$ on $\C^k \setminus Z$. The open subset $U \subset X_\Sigma$ for which $\pi |_{\pi^{-1}(U)}$ is a geometric quotient is such that $(X_\Sigma \setminus U)$ has codimension at least 3 in $X_\Sigma$.
\end{theorem}
\begin{proof}
See \cite[Thm.~2.1]{cox1995homogeneous} or \cite[Thm.~5.1.11]{cox2011toric}.
\end{proof}
Here is a longer, equivalent formulation of Theorem \ref{thm:cox} which uses less terminology. 
\begin{theorem} \label{thm:coxsimple}
Consider the action of the group $G$ in \eqref{eq:G} on $\C^k \setminus Z$. There is a one-to-one correspondence 
$$ \{\text{ closed $G$-orbits in $\C^k \setminus Z$ } \} \longleftrightarrow \{ \text{ points in $X_\Sigma$ } \}.$$
Moreover, there is an open subset $U \subset X_\Sigma$ which is such that ${\rm codim}_{X_\Sigma} (X_\Sigma \setminus U) \geq 3$ for which there is a one-to-one correspondence 
$$ \{\text{ $G$-orbits in $\pi^{-1}(U)$ } \} \longleftrightarrow \{ \text{ points in $U$ } \}.$$
These correspondences are realized by the toric morphism $\pi: \C^k \setminus Z \rightarrow X_\Sigma$ coming from $F$.
\end{theorem}
\begin{remark} \label{rem:U}
The open subset $U \subset X_\Sigma$ in Theorems \ref{thm:cox} and \ref{thm:coxsimple} is the toric variety of the largest simplicial subfan of $\Sigma$, see for instance the proof of Theorem 5.1.11 in \cite{cox2011toric}. That is, \[X_\Sigma \setminus U = \bigcup_{\sigma \text{ non-simplicial}} O(\sigma) \] 
It follows immediately that $X_\Sigma \setminus U$ has codimension at least 3 in $X_\Sigma$: all cones of dimension $\leq 2$ are simplicial. If $\Sigma$ is a simplicial fan, then $\pi: \C^k \setminus Z \rightarrow X_\Sigma$ is a \emph{geometric quotient}, meaning that the nicest possible correspondence holds: $G$-orbits in $\C^k \setminus Z$ are points in $X_\Sigma$.
\end{remark}
\begin{example} \label{ex:pointsinWPS}
The toric variety $X = \PP^2$ is simplicial, so $U = X$. The matrix $F$, the variety $Z$ and the ideal $\Irrel$ for $X_\Sigma = \PP^2$ were computed in previous examples. The (real part of the) closure of three $G$-orbits in $\C^3$ are shown in Figure \ref{fig:Gorbits}. This corresponds to the familiar fact that points in $\PP^2$ are lines through the origin in $\C^3$.
\begin{figure}[h!]
\centering
\begin{minipage}{0.45\textwidth}
\centering
\includegraphics[scale = 1.2]{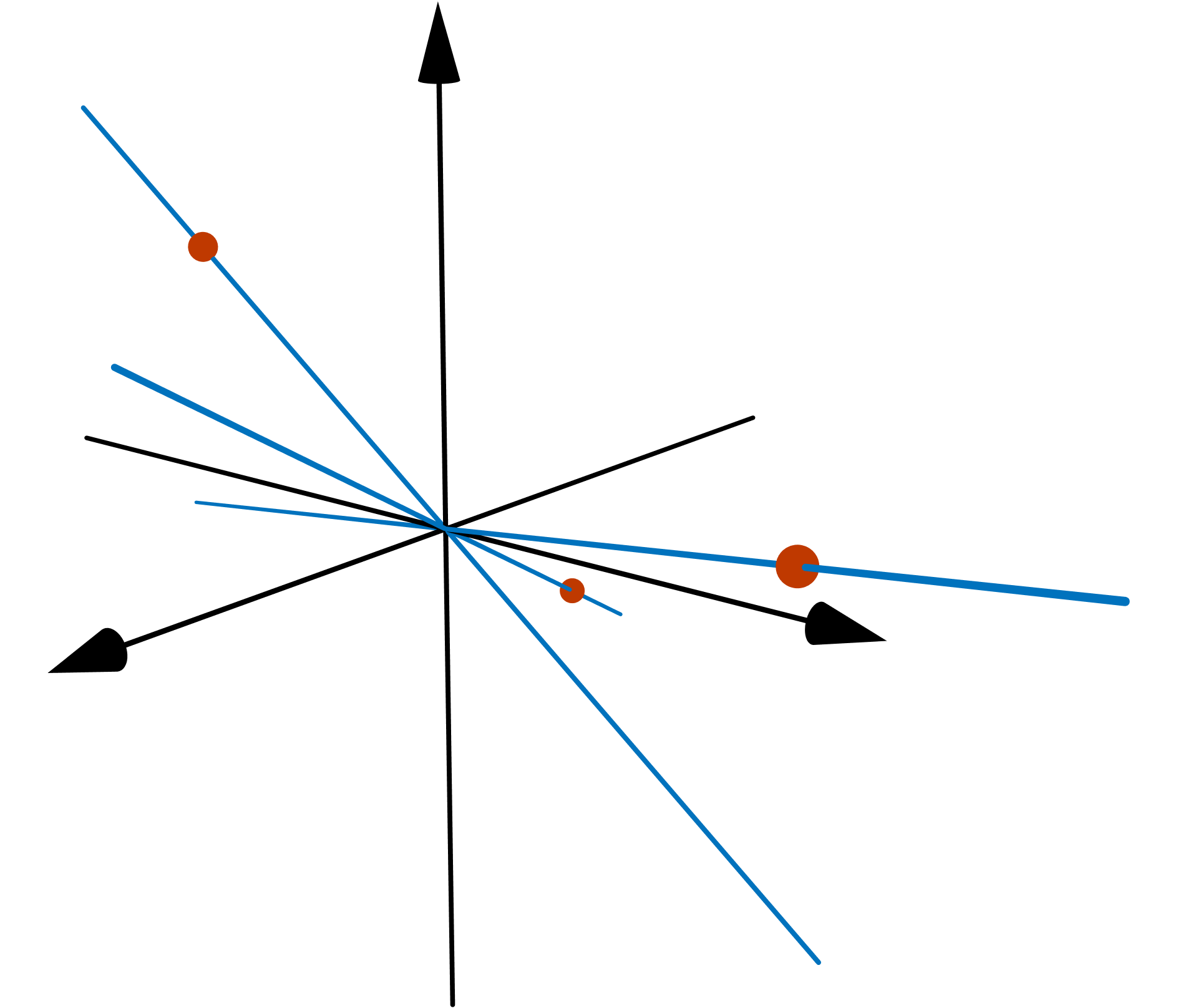}
\end{minipage}
\begin{minipage}{0.45\textwidth}
\centering
\includegraphics[scale = 1.2]{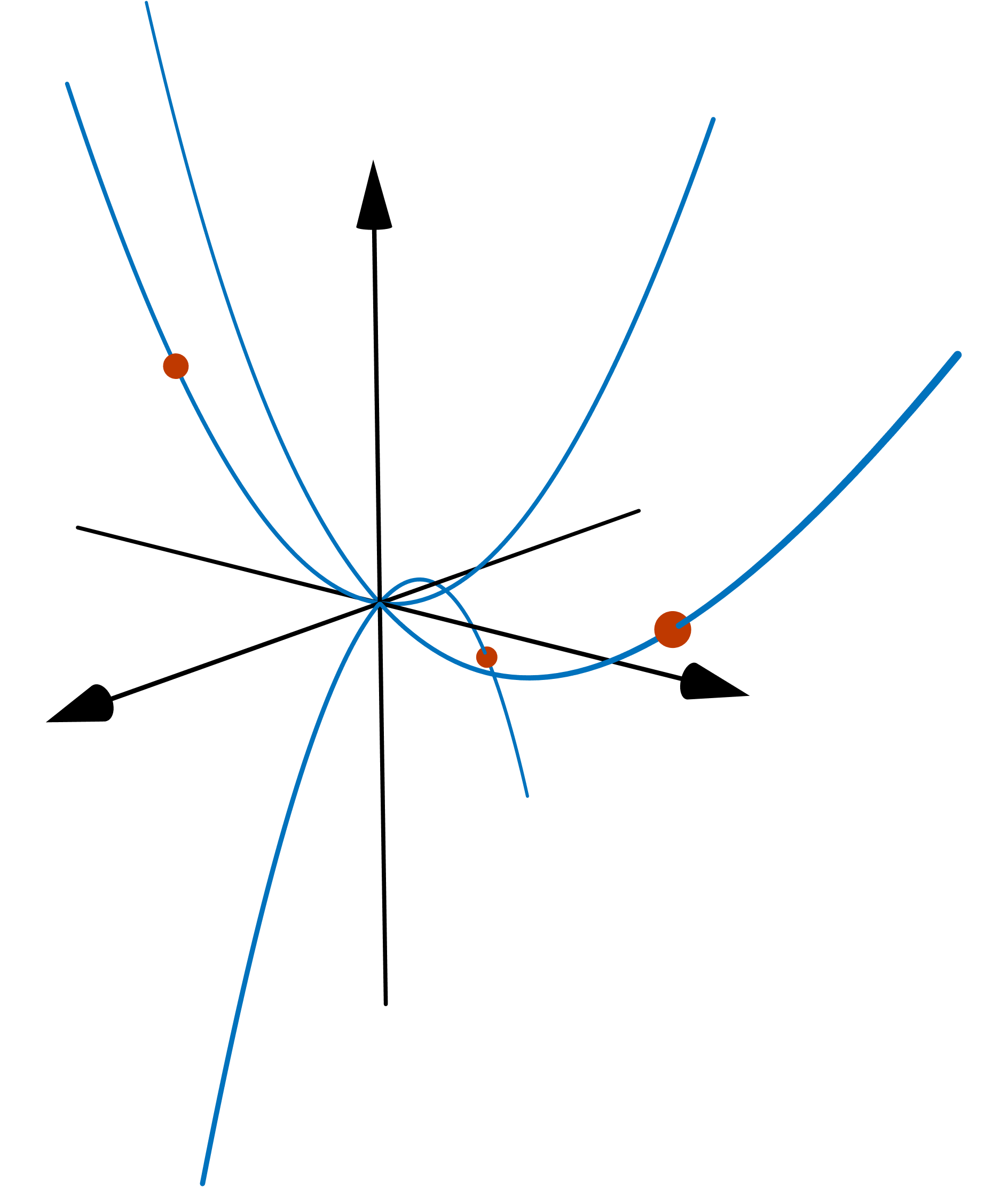}
\end{minipage}
\caption{Real $G$-orbits (closures in $\R^3$) of three points (orange dots) in the quotient construction of $\PP^2$ (left) and $\PP_{(1,2,1)}$ (right). }
\label{fig:Gorbits}
\end{figure}
We now consider the complete fan $\Sigma$ in $\R^2$ whose rays are given by 
$$ F = \begin{bmatrix}
1 & 0 & -1\\
0 & 1 & -2
\end{bmatrix}.$$
One checks that $Z = \{0\}$ and $G = \{(\lambda, \lambda^2, \lambda) ~|~ \lambda \in \C^*\} \simeq \C^*$. Some orbits are shown in the right part of Figure \ref{fig:Gorbits}. The toric variety $X_\Sigma$ is the \emph{weighted projective space} $\PP_{(1,2,1)}$. The figure suggests that we can think of points in $\PP_{(1,2,1)}$ as `curves through the origin in $\C^3$'.
\end{example}

\begin{exercise}[A non-simplicial example] \label{ex:nonsimplicial}
Example 3.1 in \cite{duff2020polyhedral} considers the toric threefold $X = X_\Sigma$ corresponding to a pyramid in $\R^3$ whose normal fan $\Sigma$ has rays with generators
$$ F = \begin{bmatrix}
0 & 1 & 0 & -1 & 0 \\ 0 & 0 & 1 & 0 & -1 \\ 1 & -1 & -1 & -1 & -1 
\end{bmatrix} = \begin{bmatrix}
u_1 & u_2 & u_3 & u_4 & u_5
\end{bmatrix}.$$
Find  $B$, $Z$ and identify the subset $U$ using Remark \ref{rem:U} by describing its fan.  
\end{exercise}

\subsection{The Cox ring of a toric variety}
In order to associate the ring $S$ (with its distinguished ideal $\Irrel$) to our toric variety $X_\Sigma$, we will equip it with a grading such that homogeneous elements in $S$ define varieties in $\C^k$ which are stable under the action of $G$. The grading will be by the \emph{(divisor) class group} $\Cl(X_\Sigma)$ of $X_\Sigma$, which is the group of Weil divisors modulo linear equivalence \cite[Ch.~4]{cox2011toric}. For toric varieties, the class group is easy to describe explicitly. We start by recalling the procedure in the notation of this section and give an alternative description of the group $G$. 

Let $D_1, \ldots, D_k$ be the torus invariant prime divisors on $X_\Sigma$ corresponding to $\rho_1, \ldots, \rho_k \in \Sigma(1)$ respectively. These are the closures of the codimension 1 torus orbits of $X_\Sigma$, see Theorem \ref{thm:orbitcone}. The divisors $D_1, \ldots, D_k$ generate the free group $\Div_T(X_\Sigma)$. Identifying $\Div_T(X_\Sigma) \simeq \Z^k$, Theorem \ref{thm:classgroup} gives a short exact sequence 
\begin{equation} \label{eq:SEScox}
 0 \longrightarrow M \overset{F^\top}{\longrightarrow} \Z^k \longrightarrow \Cl(X_\Sigma) \longrightarrow 0
\end{equation}
where the first map is $F^\top = \div$ and the second map takes a torus invariant divisor to its class in $\Cl(X_\Sigma)$. This gives $\Cl(X_\Sigma) \simeq \Z^k/\im F^\top$. Note that taking $\Hom_\Z(-, \C^*)$ of \eqref{eq:SEScox} gives us back the map of tori $\pi_{|(\C^*)^k}: (\C^*)^k \rightarrow (\C^*)^n$ from the geometric construction discussed above. We conclude that the group $G$ is isomorphic to $\Hom_\Z(\Cl(X_\Sigma), \C^*) \subset (\C^*)^k$. 

For an element $\alpha = [\sum_{i=1}^k a_i D_i] \in \Cl(X_\Sigma)$, we define the $\C$-vector subspace 
\begin{equation} \label{eq:gradedpieces}
 S_\alpha = \bigoplus_{F^\top m + a \geq 0} \C \cdot x^{F^\top m + a} \subset S,
 \end{equation}
where the sum ranges over all $m \in M$ satisfying $\langle u_i,m \rangle + a_i \geq 0$, for $i = 1, \ldots, k$. One can check that this definition is independent of the chosen representative for $\alpha$: setting $a' = a + F^\top m'$ for some $m' \in M$ gives the same vector subspace $S_\alpha$. 
We consider the grading 
\begin{equation} \label{eq:grading}
S = \bigoplus_{\alpha \in \Cl(X_\Sigma)} S_\alpha 
\end{equation} 
\begin{definition}[Cox ring]
The ring $S$ with its \emph{irrelevant ideal} \eqref{eq:irrelideal} and the grading \eqref{eq:grading} is called the \emph{homogeneous coordinate ring, total coordinate ring} or \emph{Cox ring} of $X_\Sigma$. 
\end{definition}
If $f = \sum_{F^\top m + a \geq 0} c_m x^{F^\top m + a} \in S_\alpha$ is homogeneous of degree $\alpha$, then for $g \in G \subset (\C^*)^k$,
$$ f(g \cdot x) = \sum_{F^\top m + a \geq 0} c_m (g \cdot x)^{F^\top m + a} = g^a f(x),$$
where we use that $g^{F^\top m } = 1$ by definition of $G$. Hence $V_{\C^k}(f)$ is stable under the action of $G$.
We define the \emph{variety defined by $f$ in $X$} as
\begin{equation} \label{eq:subhyper}
V_{X_\Sigma}(f) = \pi( V_{\C^k}(f)  \cap (\C^k \setminus Z)) = \{ p \in X_\Sigma ~|~ f(x) = 0 \textup{ for some } x \in \pi^{-1}(p) \}.
\end{equation}
The generalized definition for homogeneous ideals $I \subset S$ is 
\begin{equation} \label{eq:subvar}
V_{X_\Sigma}(I) = \{ p \in X_\Sigma ~|~ f(x) = 0 \textup{ for some } x \in \pi^{-1}(p) \text{ for all } f \in I \}.
\end{equation}
When $X$ is simplicial, Theorem \ref{thm:coxsimple} and the observation that $V_{\C^k}(I)$ is stable under $G$ imply $\pi^{-1}(V_X(I)) = V_{\C^k}(I)  \cap (\C^k \setminus Z)$. In the non-simplicial case, the inclusion $\supset$ might be strict.
\begin{example}
For the toric variety from Exercise \ref{ex:nonsimplicial}, the ideal $I = \langle x_2, x_3, x_4, x_5 \rangle \subset S = \C[x_1, x_2, x_3, x_4, x_5]$ is homogeneous, and $V_{X}(I) = \{ p \} = X \setminus U$ is a point. The fiber $\pi^{-1}(p)$ contains a unique closed $G$-orbit, equal to $V_X(I) \cap (\C^k \setminus Z)$. However, it contains many more non-closed $G$-orbits, see \cite[Example 3.1]{duff2020polyhedral}.
\end{example}
We generalize the table in Example \ref{ex:P2Cox} for toric varieties $X_\Sigma$ and add some terminology:
\[
\begin{matrix}
&\text{Algebra} & & \text{Geometry} \\ \hline
\text{Cox ring} &S & \overset{\Specm(\cdot)}{\longrightarrow} & \C^k & \text{total coordinate space} \\
\text{irrelevant ideal}& \Irrel & \overset{V_{\C^k}(\cdot)}{\longrightarrow} & Z & \text{base locus} \\
\text{class group} &\Cl(X_\Sigma) & \overset{\Hom_\Z(\cdot,\C^*)}{\longrightarrow} & G & \text{reductive group}
\end{matrix}
\]
We point out that under our assumption that $\Sigma$ is complete, all of the graded pieces $S_\alpha, \alpha \in \Cl(X_\Sigma)$ are finite dimensional $\C$-vector spaces \cite[Prop.~4.3.8]{cox2011toric}.
\begin{remark}
In this construction, there is a one-to-one correspondence between 
\begin{enumerate}
\vspace{-0.0cm}\item the variables $x_1, \ldots, x_k$ of $S$, 
\vspace{-0.2cm}\item the rays $\rho_1, \ldots, \rho_k$ of $\Sigma(1)$, 
\vspace{-0.2cm}\item the columns $u_1, \ldots, u_k$ of $F$,
\vspace{-0.2cm}\item the torus invariant prime divisors $D_1,\ldots, D_k$,
\vspace{-0.2cm}\item[5.] (if $\Sigma = \Sigma_P$) the facets of $P$.
\end{enumerate}
We have that $D_i = V_{X_\Sigma}(x_i)$ and $\pi(x) \in D_i \Leftrightarrow x_i = 0$.
\end{remark}
\begin{example} \label{ex:gradingP2}
Let $X_\Sigma = \PP^2$. The class group is $\Cl(\PP^2) \simeq \Z^3/\im F^\top \simeq \Z$, see Example \ref{ex:divisorsP2}. 
Using, for instance, the identification $\Z^3/\im F^\top \rightarrow \Z$ given by $(a_1,a_2,a_3) + \im F^\top = (0,0,a_1+a_2+a_3) + \im F^\top \mapsto a_1 + a_2 + a_3 \in \Z$ (the divisors $a_1D_1 + a_2D_2 + a_3D_3$ and $(a_1+a_2+a_3)D_3$ are linearly equivalent), we see that the $\Z$-grading on $S$ is the usual grading of the homogeneous coordinate ring of $\PP^2$: $\deg(x_1^{a_1}x_2^{a_2}x_3^{a_3}) = a_1 + a_2 + a_3$ and 
$$S_{[d D_3]} = \bigoplus_{\substack{m_1 \geq 0\\m_2 \geq 0\\ d - m_1 - m_2 \geq 0}} \C \cdot x_1^{m_1} x_2^{m_2} x_3^{d-m_1-m_2}$$
is spanned by monomials of `degree' $d$, in the classical sense. 
\end{example}
\begin{example} \label{ex:hirz3}
The Hirzebruch surface $\mathscr{H}_2$ is the toric variety of the fan from Figure \ref{fig:H2}. The matrix $F$ is 
$$ F = \begin{bmatrix}
1&0&-1&0\\0&1&2&-1
\end{bmatrix}.$$
The Cox ring $S = \C[x_1, x_2,x_3,x_4]$ is graded by $ \Cl(\mathscr{H}_2)\simeq \Z^4/\im F^\top \simeq \Z^2$, with $\deg(x^a) = \deg(x_1^{a_1}x_2^{a_2}x_3^{a_3}x_4^{a_4}) = (a_1 - 2a_2+a_3, a_2 + a_4)$. The reductive group and base locus are given by 
$G = \{(\lambda, \mu, \lambda, \lambda^2 \mu) ~|~ (\lambda,\mu)\in (\C^*)^2 \} \subset (\C^*)^4$ and $Z = \V_{\C^4}(x_1,x_3) \cup \V_{\C^4}(x_2,x_4) \subset \C^4$ respectively.
Since $\mathscr{H}_2$ is smooth, it is simplicial (in the notation from above, $U = \mathscr{H}_2$).
\end{example}
\begin{exercise}
Show that the Cox ring of $\PP^1 \times \PP^1$ is $\C[x_0,x_1,y_0,y_1]$, where $\deg(x_i) = (1,0) \in \Z^2$ and $\deg(y_i) = (0,1) \in \Z^2$. 
\end{exercise}

\begin{remark}
Some limited functionality regarding Cox rings is available in \texttt{Oscar.jl} (v0.8.1), see for instance the command \texttt{cox\_ring}. It would be interesting to extend this text with more examples in a next version, when this has been further developed. 
\end{remark}

\subsection{Homogenization} \label{sec:homog}
Homogeneous ideals in the Cox ring $S$ of a toric variety $X$ give closed subvarieties via \eqref{eq:subvar}. In fact, all closed subvarieties in this way \cite{cox1995homogeneous}. Here, we focus on codimension 1 subvarieties obtained in the following way. Let $\hat{f} \in \C[M]$ be a non-zero Laurent polynomial. This defines a hypersurface $V_T(\hat{f})$ in the torus $T \simeq (\C^*)^n$. Its closure in $X \supset T$ is denoted $\overline{ V_T(\hat{f}) } \subset X$. We are interested in finding the defining equation of $\overline{ V_T(\hat{f}) } \subset X$. More precisely, we will determine a homogeneous polynomial $f \in S$ such that $\overline{ V_T(\hat{f}) } = V_X(f)$, see \eqref{eq:subhyper}. The construction mirrors the usual \emph{homogenization} when $X = \PP^n$.

\begin{example} \label{ex:homogenizationP2}
Let $M = \Z^2$, $\C[M] = \C[t_1^{\pm 1}, t_2^{\pm 1} ]$, $T = (\C^*)^2$ and consider
\begin{equation} \label{eq:fP2}
\hat{f} = a_0 + a_1 t_1 + a_2 t_2 + a_3 t_1 t_2 + a_4 t_1^2 + a_5 t_2^2 \quad \in \C[M]. 
\end{equation}
The coefficients $a_i$ are complex numbers, not all equal to zero. This defines a curve $V_T(\hat{f}) \subset T$. Its closure in $\PP^2 \supset T$ is the zero locus $V_{\PP^2}(f)$ of the homogeneous polynomial 
\[ f = a_0 x_3^2 + a_1 x_1x_3 + a_2 x_2x_3 + a_3 x_1 x_2 + a_4 x_1^2 + a_5 x_2^2 \quad \in S_2. \]
Here $S_2 = S_{[2 D_3]}$, see Example \ref{ex:gradingP2}. \emph{Homogenization} is passing from $\hat{f}$ to $f$. 

Consider the affine open covering $\PP^2 = U_{\sigma_1} \cup U_{\sigma_2} \cup U_{\sigma_3}$ from Example \ref{ex:coveringP2}. The affine variety $V_{\PP^2}(f) \cap U_{\sigma_3} \subset U_{\sigma_3}$ is obtained via \emph{de-homogenization}: 
\[ \hat{f}_3 = \frac{f}{x_3^2} = a_0 + a_1 \left ( \frac{x_1}{x_3} \right) + a_2 \left (\frac{x_2}{x_3} \right) + a_3 \left (\frac{x_1}{x_3} \right) \left ( \frac{x_2}{x_3} \right) + a_4 \left (\frac{x_1}{x_3} \right)^2 + a_5 \left ( \frac{x_2}{x_3} \right)^2. \]
The polynomial $\hat{f}_3$ is to be interpreted as an element of $\C[U_{\sigma_3}] = \C[t_1, t_2]$ where $t_1 = x_1/x_3, t_2 = x_2/x_3$, and $V_{\PP^2}(f) \cap U_{\sigma_3} = V_{U_{\sigma_3}}(\hat{f}_3)$. Repeating this for the chart $U_{\sigma_1}$, we define
\[ \hat{f}_1 = \frac{f}{x_1^2} = a_0 \left ( \frac{x_3}{x_1} \right )^2 + a_1 \left ( \frac{x_3}{x_1} \right ) + a_2 \left ( \frac{x_2}{x_1} \right ) \left ( \frac{x_3}{x_1} \right ) + a_3 \left ( \frac{x_2}{x_1} \right ) + a_4 + a_5 \left ( \frac{x_2}{x_1} \right )^2, \]
which is an element of $\C[U_{\sigma_1}] = \C[t_1^{-1}, t_1^{-1}t_2]$, as $t_1^{-1} = x_3/x_1, t_1^{-1}t_2 = x_2/x_1$. The dehomogenization $\hat{f}_2$ is interpreted analogously as an element of $\C[U_{\sigma_2}] = \C[t_1t_2^{-1}, t_2^{-1}]$. An important observation is that, when restricted to $U_{\sigma_1 \cap \sigma_3} = U_{\sigma_1} \cap U_{\sigma_3}$, we have the equality 
\begin{equation} \label{eq:agreeonoverlap}
\hat{f_1}_{|U_{\sigma_1 \cap \sigma_3}} = \frac{1}{t_1^2} \, \hat{f_3}_{|U_{\sigma_1 \cap \sigma_3}}. 
\end{equation}
Here dividing by $t_1^2$ makes sense because $\C[U_{\sigma_1 \cap \sigma_3}] = \C[U_{\sigma_3}]_{t_1}$. This means that, although $\hat{f}_1$ and $\hat{f}_3$ do \emph{not} glue to a function on $U_{\sigma_1} \cup U_{\sigma_3}$ (note the factor $t_1^{-2}$), they do agree on \emph{where they vanish} in the overlap $U_{\sigma_1} \cap U_{\sigma_3}$. Indeed, Equation \eqref{eq:agreeonoverlap} shows that for $p \in U_{\sigma_1} \cap U_{\sigma_3}$, $\hat{f}_1(p) = 0$ if and only if $\hat{f}_3(p) = 0$. This means that the set 
\[ V_{\PP^2}(\hat{f}) = \{ \, p \in \PP^2 ~|~ \hat{f}_i(p) = 0 \text{ for any $i$ such that $p \in U_{\sigma_i}$ } \} \]
is well-defined. Moreover, it is not hard to check that $V_{\PP^2}(\hat{f}) = V_{\PP^2}(f) = \overline{V_T(\hat{f})}$. For the reader familiar with vector bundles, we point out that this describes $V_{T}(\hat{f})$ as the zero locus of a global section of a line bundle. Locally, the section is given by $\hat{f}_i: U_{\sigma_i} \rightarrow \C$, and the transition functions are Laurent monomials, such as $t_1^{-2}$ in \eqref{eq:agreeonoverlap}.
\end{example}

We generalize Example \ref{ex:homogenizationP2} to the following setting. Let $\hat{f} \in \C[M] = \C[\Z^n] = \C[t_1^{\pm 1}, \ldots, t_n^{\pm 1}]$ be a Laurent polynomial, given by 
\[ \hat{f} = \sum_{m \in \A} c_m \, t^m, \quad c_m \in \C^*.\]
Its \emph{Newton polytope} $P$ is defined as the convex hull of $\A$. That is, $P = {\rm Conv}(\A) \subset \R^n$. We assume $P$ to be full-dimensional. We consider the toric variety $X_P$ associated to $P$, and determine $\overline{V_T(\hat{f})}$, where the closure is in $X_P \supset T$. Moreover, we determine a homogeneous element $f \in S$ in the Cox ring $S$ of $X_P$ such that $\overline{V_T(\hat{f})} = V_{X_P}(f)$. 

Choose $\ell \in \Z_{>0}$ such that $\ell \cdot P$ is very ample, see Definition \ref{def:veryample}. For each vertex $v \in P \cap M$, define the affine semigroup $\S_v = \N \cdot \{ (\ell \cdot P \cap M) - \ell v  \}$ and $\hat{f}_v = t^{-v} \cdot \hat{f} \in \C[\S_v]$. 
\begin{exercise} \label{ex:fv}
Show that $\S_v$ is saturated in $M$ and $\C[\S_v] = \C[U_{\sigma_v}]$, where $\sigma_v \in \Sigma_P$ is the cone in the normal fan $\Sigma_P$ corresponding to $v$. Check that $\hat{f}_v \in \C[\S_v]$, and that $\hat{f}_i, i = 1, 2, 3$ from Example \ref{ex:homogenizationP2} correspond to the vertices $v$ of the Newton polytope of \eqref{eq:fP2}. 
\end{exercise}
Let $v, v' \in P \cap M$ be two vertices of $P$ and let $\tau = \sigma_v \cap \sigma_{v'}$ be the intersection of the corresponding cones in the normal fan of $P$. Restricting $f_v$ and $f_{v'}$ to $U_\tau$, we find that 
\[ (\hat{f}_{v'})_{|U_\tau} = t^{v-v'} \cdot (\hat{f}_v)_{|U_{\tau}}, \quad \text{ so that } \quad \hat{f}_{v'}(p) = 0 \Longleftrightarrow \hat{f}_v(p) = 0 \text{ for $p \in U_\tau$}.\]
Dividing by $t^{v'-v}$ is justified by $\C[U_\tau] = \C[U_{\sigma_v}]_{t^{v'-v}}$, see Proposition \ref{prop:normalfanglue}. We set
\[ V_{X_P}(\hat{f}) = \{ \, p \in X_P ~|~ f_v(p) = 0 \text{ for some } v \text{ such that } p \in U_{\sigma_v} \, \}. \]
\begin{proposition}
With the notation introduced above, we have $V_{X_P}(\hat{f})= \overline{V_T(\hat{f})}$.
\end{proposition}
\begin{proof}
The set $V_{X_P}(\hat{f})$ is closed in $X_P$. Its intersection with $U_{\sigma_v}$ is $V_{U_{\sigma_v}}(\hat{f}_v)$. The statement follows from the fact that $V_{U_{\sigma_v}}(\hat{f}_v)$ is the closure of $V_T(\hat{f})$ in $U_\sigma$. To see this, let $m \in M$ be a lattice point in the interior of $\sigma_v^\vee$. We have $\C[M] = \C[U_{\sigma_v}]_{t^m}$, and $\langle \hat{f}_v \rangle \subset \C[U_{\sigma_v}]$ is the \emph{contraction} of $\langle \hat{f} \rangle \subset \C[M]$ in $\C[U_{\sigma_v}]$ (i.e., the largest ideal that localizes to $\langle \hat{f} \rangle$).
\end{proof}
Now comes homogenization. Our task is to find $f \in S_\alpha$ for some $\alpha \in \Cl(X_P)$ such that $V_{X_P}(f) = V_{X_P}(\hat{f})$. We first determine $\alpha$. Recall that the torus invariant divisor on $X_P$ associated to $P$ is $D_P = \sum_{\rho \in \Sigma_P(1)} a_\rho D_\rho$, where $a_\rho$ comes from the minimal facet representation
\[ P = \{ m \in \R^n ~|~ \langle u_\rho, m \rangle + a_\rho \geq 0 \} = \{ m \in \R^n ~|~ F^\top m + a \geq 0 \}. \]
Here $F \in \Z^{n \times k}$ is an integer matrix with the primitive ray generators $u_\rho$ for its columns. This was described in Example \ref{ex:divfrompol}. We have also seen that 
\[ \Gamma(X_P, \OO_{X_P}(D_P)) = \bigoplus_{m \in P \cap M} \C \cdot t^m, \]
so that $\hat{f} \in \Gamma(X_P, \OO_{X_P}(D_P))$. The sum in this description of the global sections of $\OO_{X_P}(D_P)$ is similar to the sum in \eqref{eq:gradedpieces}, where we defined the graded pieces of $S$. In particular, for $\alpha = [D_P] \in \Cl(X_P)$ we have 
\[ S_{\alpha} = \bigoplus_{m \in P \cap M} \C \cdot x^{F^\top m + a } .\]
The natural linear map $\Gamma(X_P, \OO_{X_P}(D_P))  \rightarrow S_{\alpha}$ given by $t^m \mapsto x^{F^\top m + a}$  is \emph{homogenization}. In particular, the image of our Laurent polynomial $\hat{f}$ is $f = \sum_{m \in \A} c_m \, x^{F^\top m + a} \in S_\alpha$.

\begin{exercise}
Check that homogenization is invariant under translations of $P$. That is, the homogenization of $\hat{f}$ is equal to that of $t^m \cdot \hat{f}$, for any Laurent monomial $t^m$. 
\end{exercise} 

One can define homogenization for torus invariant divisors on a normal toric variety $X_\Sigma$.
\begin{definition}[Homogenization] \label{def:homogenization}
Let $X_\Sigma$ be a normal toric variety with no torus factor and let $D \in \Div_T(X_\Sigma)$ be a torus invariant divisor. \emph{Homogenization} with respect to $D$ is
\[\eta_D : \Gamma(X_\Sigma, \OO_{X_\Sigma}(D)) \longrightarrow S_{[D]} , \quad t^m \longmapsto x^{F^\top m + a}. \]
\end{definition}

In the rest of this section, we keep using the homogenization $f = \eta_{D_P}(\hat{f})$.
It remains to show that $V_{X_P}(f) = V_{X_P}(\hat{f})$. We will check this locally, on the affine charts $U_{\sigma_v}$, where $v \in P \cap M$ is a vertex. Let $x^{\hat{\sigma}_v}$ be as in \eqref{eq:xsigmahat}. By the orbit-cone correspondence (Theorem \ref{thm:orbitcone}), the toric variety $\C^k \setminus V(x^{\hat{\sigma}_v}) \subset \C^k \setminus Z$ corresponds to the fan $\Sigma'_v$ 
consisting of the cone ${\rm Cone}(e_i ~|~ \rho_i \in \sigma_v(1))$ and all its faces. The map $F : \R^k \rightarrow \R^n$ is compatible with $\Sigma'_v$ and the fan of $U_{\sigma_v}$. The corresponding toric morphism is the restriction $\pi_{\sigma_v} = \pi_{|\C^k \setminus V(x^{\hat{\sigma}_v})}: \C^k \setminus V(x^{\hat{\sigma}_v}) \rightarrow U_{\sigma_v}$. At the level of $\C$-algebras, this map is 
\[ \pi_{\sigma_v}^* : \C[U_{\sigma_v}] \longrightarrow S_{x^{\hat{\sigma}_v}}, \quad t^m \mapsto x^{F^\top m}. \]
Here $m \in \S_v$ ensures that only the variables of $x^{\hat{\sigma}_v}$ appear in the denominator of $x^{F^\top m}$.

\begin{proposition}
With the notation introduced above, we have $V_{X_P}(f) = V_{X_P}(\hat{f}) = \overline{V_T(\hat{f})}$.
\end{proposition}
\begin{proof}
We have $V_{X_P}(\hat{f}) \cap U_{\sigma_v} = V_{U_{\sigma_v}}(\hat{f}_v)$. The pullback of $\hat{f}_v$ under $\pi_{\sigma_v}$ is 
\begin{equation} \label{eq:proof64}
 \pi_{\sigma_v}^*(\hat{f}_v) = \pi_{\sigma_v}^* \left ( \sum_{m \in \A} c_m\,  t^{m -v} \right ) = \sum_{m \in \A} c_m \, x^{F^\top (m-v)} = \frac{f}{x^{F^\top v + a}}. 
\end{equation}
This means that, when $\pi(x) \in U_{\sigma_v}$, $\hat{f}_v(\pi(x)) = f(x)/x^{F^\top v + a}$. If $p \in V_{X_P}(f)$, then there is $x \in \pi^{-1}(p)$ such that $f(x) = 0$. Let $v$ be any vertex of $P$ such that such that $p \in U_{\sigma_v}$. By Lemma \ref{lem:imagedistpoints}, $\pi^{-1}(U_{\sigma_v}) = \C^k \setminus V(x^{\hat{\sigma}_v})$, so that $x_i \neq 0$ for all $i$ such that $\rho_i \notin \sigma(1)$. This means that we have $x^{F^\top v + a} \neq 0$, and hence $\hat{f}_v(\pi(x)) = \hat{f}_v(p) = 0$ by \eqref{eq:proof64}, so that $p \in V_{X_P}(\hat{f})$. This shows $V_{X_P}(f) \subset V_{X_P}(\hat{f})$. The proof of the reverse inclusion is analogous.
\end{proof}

\begin{example} \label{ex:hirz}
Let $\hat{f} = 1 + t_1 + t_2 + t_1t_2 + t_1^2t_2 + t_1^3t_2 \in \C[t_1^{\pm 1}, t_2^{\pm 1}]$. The toric variety $X_P$ is the Hirzebruch surface from Example \ref{ex:hirz3}. In the notation used there, $\hat{f}$ homogenizes to 
\[ f = x_3x_4 + x_1x_4 + x_2x_3^3 + x_1 x_2 x_3^2 + x_1^2x_2x_3 + x_1^3x_2. \]
Its degree is $[D_3 + D_4] \in \Cl(\mathscr{H}_2)$. 
\end{example}

\subsection{Complete intersections on toric varieties}
Let $M = \Z^n$ and let $\hat{f}_1, \ldots, \hat{f}_n \in \C[M] = \C[t_1^{\pm 1},\ldots, t_n^{\pm 1}]$ be Laurent polynomials. In this section, we consider the system of equations 
\begin{equation} \label{eq:hatsys}
\hat{f}_1 = \cdots = \hat{f}_n = 0
\end{equation}
with solutions $V_T(\hat{f}_1, \ldots, \hat{f}_n) = \{ p \in T ~|~ \hat{f}_i(p) = 0, i = 1, \ldots, n \}$. Under suitable assumptions, $V_T(\hat{f}_1, \ldots, \hat{f}_n)$ consists of finitely many points. We let $\A_i \subset \Z^n$ be such that 
\begin{equation} \label{eq:Ai}
 \hat{f}_i = \sum_{m \in \A_i} c_{i,m} \, t^m, \quad i = 1, \ldots, n. 
 \end{equation}
We are particularly interested in the number $\delta$ of solutions to \eqref{eq:hatsys}, when the exponents $\A_i$ are fixed. B\'ezouts theorem gives a full answer when $\A_i = (d_i \cdot \Delta_n) \cap \Z^n$ is a dilate of the standard simplex. The case $\A_1 = \cdots = \A_n$ was dealt with in Section \ref{sec:kush}: it is covered by Kushnirenko's theorem. Both results rely on the fact that if the question is reformulated over a different, compact solution space $X \supset T$, the number $\delta$ is the same for \emph{all} systems with finitely many solutions (if we count multiplicities). Indeed, in B\'ezout's theorem $X = \PP^n$, and in Section \ref{sec:kush}, $X$ is the projective toric variety $X_{\A}$. Having a compact solution space is desirable, for instance, to apply results from \emph{intersection theory}. Here, we will discuss what $X$ is for \eqref{eq:hatsys} where the $\A_i$ are arbitrary. We will also state the \emph{Bernstein-Khovanskii-Kushnirenko (BKK) theorem}, which generalizes B\'ezout and Kushnirenko to this setting. 

Defining $X$ involves the binary operation of \emph{Minkowski addition} on polytopes.
\begin{definition}[Minkowski sum] \label{def:minkowski}
Let $P$ and $Q$ be polytopes in $M_\R$. The \emph{Minkowski sum} of $P$ and $Q$ is 
$$P+Q = \{p+q ~|~ p \in P, q \in Q\} \subset M_\R.$$
\end{definition}
One easily checks commutativity and associativity. Minkowski sums appear in the definition of an important integer associated to a tuple of $n$ lattice polytopes. 
\begin{definition}[Mixed volume]\label{def:mixedvolume}
Let $P_1, \ldots, P_n$ be polytopes in $\R^n$. The function $(\lambda_1, \ldots, \lambda_n) \mapsto {\rm Vol}(\lambda_1 \cdot P_1 + \cdots + \lambda_n \cdot P_n)$ is a homogeneous polynomial function of degree $n$ on $\R_{\geq 0}^n$, see \cite[Ch.~7, Sec.~4, Prop.~4.9]{cox2006using}. Its coefficient standing with the mixed monomial $\lambda_1 \lambda_2 \cdots \lambda_n$ is an integer, called the \emph{mixed volume} ${\rm MV}(P_1, \ldots, P_n)$.
\end{definition}
A useful formula for the mixed volume in the case where $n = 2$ is 
\begin{equation}\label{eq:MV2d}
{\rm MV}(P_1,P_2) = {\rm Vol}(P_1 + P_2) - {\rm Vol}(P_1) - {\rm Vol}(P_2). 
\end{equation}
Let $P_i = {\rm Conv}(\A_i) \subset \R^n$ and $P = P_1 + \cdots + P_n$. We assume that $P$ is full-dimensional. Our space of solutions $X$ is the toric variety $X_P$, with Cox ring $S$. To each $\hat{f}_i$ we will associate a divisor $E_i \in \Div_T(X_P)$ and a homogeneous element $f_i = \eta_{E_i}(\hat{f}_i) \in S_{[E_i]}$ such that $V_{X_P}(f_i) = \overline{V_T(\hat{f}_i)}$. Here the closure is in $X_P$, $\eta_{E_i}$ is the homogenization map from Definition \ref{def:homogenization} and $V_{X_P}(f_i)$ is as in \eqref{eq:subhyper}. The BKK theorem counts the number $\delta$ of points in \[V_{X_P}(f_1, \ldots, f_n) = V_{X_P}(f_1) \cap \cdots \cap V_{X_P}(f_n).\] If $\delta$ is finite, we say that $\hat{f_1}, \ldots, \hat{f_n}$ define a \emph{complete intersection} on $X_P$. 

Before stating the theorem, we define the divisors $E_i$. When $P$ and $P_i$ have the same normal fan, it is natural to set $E_i = D_{P_i}$. In general $P_i$ is only a Minkowski summand of $P$. This entails that the normal fan of $P$ \emph{refines} that of $P_i$, see e.g.~Proposition 6.2.5 in \cite{cox2011toric}. Let $F \in \Z^{n \times k} = [u_1 ~ \cdots ~ u_k]$ be the matrix of primitive ray generators of $\Sigma_P(1)$. We have
\[ P_i = \{ m \in \R^n ~|~ F^\top m + a_i \geq 0 \}, \]
where $a_i = (a_{i,1}, \ldots, a_{i,k})$ with $a_{i,j} = -\min_{m \in \A_i} \langle u_j, m \rangle$. We point out that this is in general \emph{not} a minimal facet representation. We set $E_i = \sum_{j = 1}^k a_{i,j} \, D_j$.

\begin{theorem}[Toric BKK theorem] \label{thm:toricbkk}
Let $\hat{f}_1, \ldots, \hat{f}_n \in \C[M]$ be as in \eqref{eq:Ai} and let $X_P$ be the toric variety of the polytope $P = P_1 + \cdots + P_n$, where $P_i = {\rm Conv}(\A_i)$. If $V_{X_P}(f_1, \ldots, f_n)$ consists of $\delta < \infty$ points on $X_P$, counting multiplicities, then $\delta$ is given by ${\rm MV}(P_1,\ldots,P_n)$. For generic choices of the coefficients $c_{i,m}$, the number of solutions in the torus $T_{X_P} \simeq T_N  = (\C^*)^n$ is exactly equal to ${\rm MV}(P_1,\ldots,P_n)$ and all solutions have multiplicity one.
\end{theorem}
\begin{proof}
See \cite[Sec.~5.5]{fulton1993introduction}. For the final statement, see \cite{bernshtein1975number} for the original proof, \cite{huber1995polyhedral} for a proof based on homotopy continuation, and \cite[Ch.~7, Sec.~5]{cox2006using} for a sketch.
\end{proof} 

\begin{example}
When $P_1 = \cdots = P_n = P$, we have ${\rm MV}(P, \ldots, P) = n!\, {\rm Vol}(P)$, so that Theorem \ref{thm:toricbkk} implies Theorem \ref{thm:kush}. When $P_i = d_i \cdot \Delta_n$ is the $d_i$-dilation of the standard simplex, the mixed volume ${\rm MV}(P_1, \ldots, P_n)$ evaluates to the B\'ezout number $d_1 \cdots d_n$. 
\end{example}

\begin{example} \label{ex:hirz1}
Consider the system of Laurent polynomial equations $\f_1= \f_2 = 0$ given by 
\begin{align*}
\f_1 &= 1 + t_1 + t_2 + t_1t_2 + t_1^2t_2 + t_1^3t_2,\\
\f_2 &= 1 + t_2 + t_1t_2 + t_1^2t_2.
\end{align*}
These are elements of $\C[M] = \C[t_1^{\pm 1}, t_2^{\pm 1}]$ and give the polygons $P_1, P_2, P$ shown in Figure \ref{fig:hirz1}.
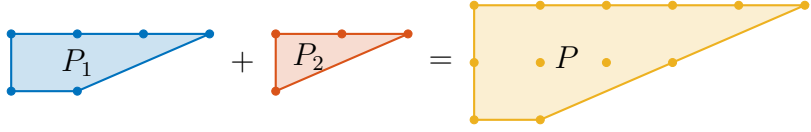
\begin{figure}
\centering
\begin{tikzpicture}
\begin{axis}[%
width=4.500in,
height=0.900in,
scale only axis,
xmin=-0.5,
xmax=12.5,
xtick = \empty,
ymin=-1,
ymax=2,
ytick = \empty,
axis background/.style={fill=white},
axis line style={draw=none} 
]
\addplot [color=mycolor1,solid,thick, fill opacity = 0.2, fill = mycolor1,forget plot]
  table[row sep=crcr]{%
0	0\\
1	0\\
3	1\\
0	1\\
0	0\\
};
\addplot[only marks,mark=*,mark size=1.5pt,mycolor1
        ]  coordinates {
    (0,0) (1,0) (3,1) (2,1) (1,1) (0,1)
};

\node at (axis cs:1.0,0.5) {$P_1$};
\node at (axis cs:3.5,0.5) {$+$};

\addplot [color=mycolor2,solid,thick, fill opacity = 0.2, fill = mycolor2,forget plot]
  table[row sep=crcr]{%
4	0\\
6	1\\
4	1\\
4	0\\
};
\addplot[only marks,mark=*,mark size=1.5pt,mycolor2
        ]  coordinates {
    (4,0) (6,1) (5,1) (4,1)
};
\node at (axis cs:4.5,0.6) {$P_2$};
\node at (axis cs:6.5,0.5) {$=$};

\addplot [color=mycolor3,solid,thick, fill opacity = 0.2, fill = mycolor3,forget plot]
  table[row sep=crcr]{%
7	-0.5\\
8	-0.5\\
12	1.5\\
7	1.5\\
7	-0.5\\
};

\addplot[only marks,mark=*,mark size=1.5pt,mycolor3
        ]  coordinates {
    (7,-0.5) (8,-0.5) (7,0.5) (8, 0.5) (9, 0.5) (10, 0.5) (7, 1.5) (8,1.5) (9,1.5) (10,1.5) (11,1.5) (12,1.5)
};

\node at (axis cs:8.4,0.6) {$P$};
\end{axis}
\end{tikzpicture}%
\caption{Polytopes from Example \ref{ex:hirz1}.}
\label{fig:hirz1}
\end{figure}
The toric variety $X_P$ is the Hirzebruch surface $\mathscr{H}_2$ from Example \ref{ex:hirz3}.
Note that $X_{P_2} \neq X_{P_1} = X_{P}$, since $P_2$ has a different normal fan. However, $P_2$ is a Minkowski summand of $P$, so it can be written as $P = \{ m \in \R^2 ~|~ F^\top m + a_2 \geq 0 \}$. Here $F$ is from Example \ref{ex:hirz3}, and $a_2 = (0,0,0,1)^\top$ gives $E_2 = D_4$. We established in Example \ref{ex:hirz} that $E_1 = D_3 + D_4$ (note that $\f = \f_1$). The system has only one solution in the torus, namely $(t_1,t_2) = (-1,-1)$. We encourage the reader to check this. However, the formula \eqref{eq:MV2d} gives ${\rm MV}(P_1,P_2) = 3$. We will show that $V_{X_P}(\f_1,\f_2)$ consists of 3 points. Homogenization gives
\begin{align*}
f_1 &= x_3x_4 + x_1x_4 + x_2x_3^3 + x_1x_2x_3^2 + x_1^2x_2x_3 + x_1^3x_2 \in S_{[D_3 + D_4]}, \\
f_2 &= x_4 + x_2x_3^2 + x_1x_2x_3 + x_1^2x_3 \in S_{[D_4]}.
\end{align*} 
We now see that the vanishing locus $V_{X_P}(f_1, f_2)$ on $X_P$ consists of three points $\pi(z_i), i = 1, \ldots, 3$, with homogeneous coordinates
$$z_1 = (-1,-1,1,1), ~ z_2 = (0,-1,1,1), ~ z_3 = (1,-1,0,1).$$
Note that $\pi(z_1)$ is the toric solution $(-1,-1)$ and the other solutions are on the boundary of the torus: $\pi(z_2) \in D_1, \pi(z_3) \in D_3$. 
\end{example}

The original BKK theorem considers solutions in $T = (\C^*)^n$. It has inspired efficient symbolic-numeric methods for solving polynomial equations, such as polyhedral homotopies \cite{huber1995polyhedral} and sparse resultants \cite{canny2000subdivision,emiris1999matrices}. The toric interpretation of the BKK theorem given in Theorem \ref{thm:toricbkk} is important for designing stable numerical algorithms \cite{bender2020toric,duff2020polyhedral,telen2019numerical}.

\small
\bibliographystyle{abbrv}
\bibliography{references.bib} 

\normalsize
\noindent
{\bf Author's address:}

\smallskip

\noindent Simon Telen, MPI-MiS Leipzig
\hfill {\tt simon.telen@mis.mpg.de}

\end{document}